\documentclass[12pt]{amsart}

\usepackage{graphicx}
\usepackage{bbm}
\usepackage{ marvosym }
\usepackage{calligra,mathrsfs}
\usepackage[all]{xy}
\usepackage{float, comment}
\usepackage{mathtools}
\usepackage{amsmath}
\usepackage{amsthm}
\usepackage{amssymb}
\usepackage{amsbsy}
\usepackage{amstext}
\usepackage{amsopn}
\usepackage{mathrsfs} 
\usepackage{enumerate}
\usepackage{xcolor}
\usepackage{graphicx} 
\usepackage{microtype} 
\usepackage[margin=1in,marginparwidth=0.8in, marginparsep=0.1in]{geometry}
\usepackage[pagebackref, bookmarks=true, bookmarksopen=true, bookmarksdepth=3,bookmarksopenlevel=2, colorlinks=true, linkcolor=blue, citecolor=blue, filecolor=blue, menucolor=blue, urlcolor=blue]{hyperref}
\usepackage{tikz}
\usepackage{bbm}
\usepackage[all]{xy}
\usetikzlibrary{arrows,calc,positioning,decorations.pathreplacing} 

\usepackage[shortlabels]{enumitem}

\numberwithin{equation}{section}

\usepackage{todonotes}

\usepackage{xcolor}

\usepackage{stmaryrd}

\usepackage{tikz-cd}

\usepackage{enumitem}

\usepackage{quiver}

\numberwithin{equation}{section}
\newtheorem{thm}{Theorem}[section]
\newtheorem{prop}[thm]{Proposition}
\newtheorem{lem}[thm]{Lemma}
\newtheorem{cor}[thm]{Corollary}
\newtheorem{rem}[thm]{Remark}
\newtheorem{definition}[thm]{Definition}
\newtheorem{example}[thm]{Example}

\newtheorem{conj}[thm]{{\bf Conjecture}}

\newtheorem*{cor*}{Corollary}

\newcommand{\nc}{\newcommand}

\nc{\bA}{\mathbb A}
\nc{\bC}{\mathbb C}
\nc{\bc}{{\bf c}}
\nc{\bD}{\mathbb D}
\nc{\bd}{\mathbb d}
\nc{\bG}{\mathbb G}
\nc{\bi}{\bold i}
\nc{\bL}{\mathbb L}
\nc{\bN}{\mathbb N}
\nc{\bO}{\mathbb O}
\nc{\bP}{\mathbb P}
\nc{\bQ}{\mathbb Q}
\nc{\bR}{\mathbb R}
\nc{\bS}{\mathbb S}
\nc{\bu}{\mathbb u}
\nc{\bv}{\bold v}
\nc{\bw}{\bold w}
\nc{\bW}{\mathbb W}
\nc{\bX}{\mathbb X}
\nc{\bY}{\mathbb Y}
\nc{\bZ}{\mathbb Z}

\nc{\cA}{\mathcal A}
\nc{\cB}{\mathcal B}
\nc{\cC}{\mathcal C}
\nc{\cD}{\mathcal D}
\nc{\cE}{\mathcal E}
\nc{\cF}{\mathcal F}
\nc{\cG}{\mathcal G}
\nc{\cH}{\mathcal H}
\nc{\cI}{\mathcal I}
\nc{\cK}{\mathcal K}
\nc{\cL}{\mathcal L}
\nc{\cM}{\mathcal M}
\nc{\cN}{\mathcal N}
\nc{\cO}{\mathcal O}
\nc{\cP}{\mathcal P}
\nc{\cQ}{\mathcal Q}
\nc{\cR}{\mathcal R}
\nc{\cT}{\mathcal T}
\nc{\cU}{\mathcal U}
\nc{\cV}{\mathcal V}
\nc{\cW}{\mathcal W}
\nc{\cX}{\mathcal X}

\nc{\al}{\alpha}
\nc{\be}{\beta}
\nc{\la}{\lambda}
\nc{\La}{\Lambda}
\nc{\ve}{\varepsilon}
\nc{\om}{\omega}
\nc{\bPsi}{\boldsymbol{\Psi}}

\nc{\gl}{\mathfrak{gl}}
\nc{\fsl}{\mathfrak{sl}}
\nc{\g}{\mathfrak{g}}
\nc{\gh}{\widehat\g}
\nc{\h}{\mathfrak{h}}
\nc{\fb}{{\mathfrak b}}
\nc{\fg}{{\mathfrak g}}
\nc{\fgh}{{\widehat{\mathfrak g}}}
\nc{\fh}{{\mathfrak h}}
\nc{\fl}{\mathfrak{l}}
\nc{\fm}{{\mathfrak m}}
\nc{\fM}{{\mathfrak M}}
\nc{\fp}{{\mathfrak p}}
\nc{\ft}{\mathfrak{t}}

\nc{\fn}{{\mathfrak n}}
\nc{\fQ}{\mathfrak{Q}}

\nc{\Aut}{\mathrm{Aut}}
\nc{\ch}{{\mathop {\rm ch}}}
\nc{\tr}{{\mathop {\rm tr}\,}}
\nc{\im}{{\mathop {\rm im}}}
\nc{\id}{{\mathop {\rm id}}}
\nc{\ad}{{\mathop {\rm ad}}}
\nc{\gr}{\mathrm{gr}}
\nc{\ord}{\mathrm{ord}}
\nc{\red}{\mathrm{red}}
\nc{\End}{\operatorname{End}}
\nc{\Spec}{\operatorname{Spec}}
\nc{\Spf}{\operatorname{Spf}}
\nc{\Proj}{\operatorname{Proj}}
\nc{\Pic}{\operatorname{Pic}}
\nc{\Lie}{\operatorname{Lie}}
\nc{\Der}{\operatorname{Der}}
\nc{\Coh}{\operatorname{Coh}}
\nc{\coh}{\mathrm{coh}}
\nc{\qcoh}{\mathrm{Qcoh }}
\nc{\Gal}{\operatorname{Gal}}
\nc{\Hom}{\mathrm{Hom}}
\nc{\Rhom}{\mathrm{RHom}}
\nc{\cHom}{\mathcal{Hom}}
\nc{\Ann}{\mathrm{Ann}}
\nc{\Vect}{\mathrm{Vect}}
\nc{\wt}{\mathrm{wt}}
\nc{\hw}{\mathrm{hw}}
\nc{\rk}{\operatorname{rank}}
\nc{\Gr}{{\mathrm {Gr}}}
\nc{\Fl}{\mathrm{Fl}}
\nc{\spn}{\mathrm{span}}
\nc{\Rep}{\operatorname{Rep}}
\nc{\Irrep}{\mathrm{Irrep }}
\nc{\supp}{\operatorname{supp}}
\nc{\tp}{\mathrm{top}}
\nc{\codim}{\mathrm{codim}}
\nc{\IC}{\operatorname{IC}}
\nc{\Res}{\mathrm{Res}}
\nc{\modules}{\mathrm{-mod}}
\nc{\Perv}{\mathrm{Perv}}
\nc{\Forg}{\operatorname{Forg}}
\nc{\Maps}{\mathrm{Maps}}
\nc{\Frac}{\operatorname{Frac}}
\nc{\Stab}{\operatorname{Stab}}
\nc{\shhom}{\mathop{\mathcal{H}\! \mathit{om}}\nolimits}
\nc{\chom}{\mathop{\mathcal{H}\! \mathit{om}}\nolimits}
\nc{\cEnd}{\mathop{\mathcal{E}\! \mathit{nd}}\nolimits}
\nc{\mods}{\mathrm{-mod}}
\nc{\dgmods}{\mathrm{-dgmod}}
\nc{\exo}{\mathrm{exo}}
\nc{\Ext}{\operatorname{Ext}}
\nc{\dg}{\mathrm{dg}}
\nc{\pr}{\mathrm{pr}}
\nc{\op}{\mathrm{op}}
\nc{\ukey}{{\underline{k}}}
\nc{\St}{\operatorname{St}}
\nc{\GL}{\operatorname{GL}}
\nc{\PGL}{\operatorname{PGL}}

\nc{\GfL}{{(G \times \bC^\times, \fl \oplus \bC D)}}
\nc{\Gfl}{{(G \times \bC^\times, \fl \oplus \bC D)}}

\nc{\eO}{\EuScript{O}}

\nc{\bra}{\langle}
\nc{\ket}{\rangle}
\nc{\pa}{\partial}
\nc{\ld}{\ldots}
\nc{\cd}{\cdots}
\nc{\hk}{\hookrightarrow}
\nc{\T}{\otimes}
\nc{\ov}{\overline}
\nc{\wh}{\widehat}
\nc{\wti}{\widetilde}
\nc{\svee}{{\!\scriptscriptstyle\vee}}
\nc{\ula}{{\underline{\la}}}
\nc{\umu}{{\underline{\mu}}}
\nc{\conv}{{\widetilde \times}}
\nc{\lach}{{\la^\svee}}
\nc{\alch}{{\al^\svee}}
\nc{\omch}{{\omega^\svee}}
\nc{\much}{{\mu^\svee}}
\nc{\md}{\text {--mod}}
\nc{\pt}{\mathrm{pt}}
\nc{\torus}{\bC^\times}

\nc{\Tr}{{\mathop {\rm Tr}\,}}
\nc{\Id}{{\mathop {\rm Id}}}

\nc{\msl}{\mathfrak{sl}}
\nc{\mgl}{\mathfrak{gl}}

\nc{\U}{\mathrm U}

\nc{\Q}{\mathfrak Q}

\nc{\on}{\operatorname} \nc\ol{\overline} \nc\ul{\underline}

\nc{\BA}{{\mathbb{A}}} \nc{\BC}{{\mathbb{C}}} \nc{\BF}{{\mathbb{F}}}
\nc{\BD}{{\mathbb{D}}} \nc{\BG}{{\mathbb{G}}} \nc{\BQ}{{\mathbb{Q}}}
\nc{\BM}{{\mathbb{M}}} \nc{\BN}{{\mathbb{N}}} \nc{\BO}{{\mathbb{O}}}
\nc{\BP}{{\mathbb{P}}} \nc{\BR}{{\mathbb{R}}}
\nc{\BZ}{{\mathbb{Z}}} \nc{\BS}{{\mathbb{S}}} \nc{\BW}{{\mathbb{W}}}

\nc{\CA}{{\mathcal{A}}} \nc{\CL}{{\mathcal{L}}} \nc{\CV}{{\mathcal{V}}} \nc{\CW}{{\mathcal{W}}}
\nc{\CalD}{{\mathcal{D}}}

\nc{\sic}{{\on{sc}}}
\nc{\add}{{\on{add}}}

\nc{\Quiv}{Q}
\nc{\Core}{\dot{\mathfrak{L}}}

\newcommand\iso{\,\vphantom{j^{X^2}}\smash{\overset{\sim}{\vphantom{\rule{0pt}{0.20em}}\smash{\longrightarrow}}}\,}

\title[Noncommutative resolutions and canonical bases]{Noncommutative resolutions of affine Schubert varieties in type A and canonical bases}

\author{Ilya Dumanski}
\address{Ilya Dumanski:\newline
		Department of Mathematics, MIT, Cambridge, MA 02139, USA.
	}
	\email{ilyadumnsk@gmail.com}

\author{Vasily Krylov}
\address{Vasily Krylov: \newline
Department of Mathematics
Harvard University and CMSA
\newline
1 Oxford Street,
Cambridge, MA 02138,
USA}
\email{vkrylov@math.harvard.edu, krylovasya@gmail.com}

\begin{document}

\begin{abstract}
Given a resolution $\widetilde{\mathrm{Gr}}^{\underline{\lambda}} \rightarrow \overline{\mathrm{Gr}}^\lambda$ of an affine Schubert variety for $\GL_n$, we define its noncommutative version --- a sheaf of algebras on $\overline{\mathrm{Gr}}^\lambda$, derived equivalent to $\widetilde{\mathrm{Gr}}^{\underline{\lambda}}$ as well as its Steinberg versions in both zero and sufficiently large positive characteristics. 
This, in particular, allows us to define the perversely-exotic t-structure on the derived category of equivariant coherent sheaves on  $\widetilde{\mathrm{Gr}}^{\underline{\lambda}}$, analogously to Bezrukavnikov--Mirkovi\'c in the case of the Springer resolution. We study the basis of classes of irreducible objects in the equivariant K-theory, and explicitly identify it with the (parabolic) Kazhdan--Lusztig canonical basis in a certain cell quotient. This allows us to relate it to  the canonical basis for the quantum affine group.
In the course of the proof, we establish some properties of coherent-constructible equivalences.
\end{abstract}

\maketitle

\tableofcontents










\section{Introduction}
\subsection{Bases from t-structures}
Many representation-theoretic objects possess important bases with  canonical properties. One way to construct a canonical basis in a vector space is to identify this vector space with a Grothendieck group of a finite length abelian category, and take classes of simple objects.

Often, the role of such a category is played by the category of perverse (constructible) sheaves on an appropriate variety.
In such a way one can construct canonical bases in Hecke algebras \cite{KL79, KL80}, quantum groups \cite{Lus90, Lus10}, and related spaces.

Lusztig \cite{Lus98, Lus99} suggested that important representation-theoretic bases should also appear in (equivariant) K-theory of varieties. Although equivariant K-theory is by definition the Grothendieck group of the abelian category of equivariant coherent sheaves, this category is typically not of finite length, and hence there is no obvious natural basis.

A strategy to overcome this difficulty was suggested in \cite{BM13}: one can define a nonstandard t-structure on the derived category of equivariant coherent sheaves, whose heart is of finite length. Classes of simple objects in the heart of this t-structure then give a basis in equivariant K-theory. In this manner, in particular, the authors of \cite{BM13} proved Lusztig's conjectures about the canonical basis in the K-theory of Springer fibers.

One can also realize the Kazhdan--Lusztig canonical basis of the affine Hecke algebra in this way. Namely, let $G$ be a connected reductive  group, and let $\wti \cN \rightarrow \cN$ be the associated Springer resolution; recall \cite{KL87} (see also \cite{CG97}) that  one has a natural isomorphism of algebras \[K^{G \times \bG_m} (\wti \cN \times_\cN \wti \cN) \simeq \cH_{G^\vee};\] 
here $\bG_m$ stands for the contracting action on $\cN$, and $\cH_{G^\vee}$ is the affine Hecke algebra of the Langlands dual group. 
One also identifies $K^{G \times \bG_m}(\wti \cN)$ with the anti-spherical $\cH_{G^\vee}$-module. One defines the {\it perversely-exotic} t-structure on $D_\coh^{b, G \times \bG_m} (\wti \cN)$ as follows. First, one constructs the {\it noncommutative Springer resolution} \cite{Bez06b} --- a noncommutative $\cO(\cN)$-algebra $A$, for which the derived equivalence $D_\coh^{b, G \times \bG_m} (\wti \cN) \simeq D^b(A\mods^{G \times \bG_m})$ holds. Then, one considers the perverse coherent t-structure \cite{Bez00, AB10} on $D^b(A\mods^{G \times \bG_m})$. It follows from \cite[Theorem~6.2.1]{BM13} that classes of simple objects in the heart of this t-structure correspond to the Kazhdan--Lusztig canonical basis in the anti-spherical module $K^{G \times \bG_m}(\wti \cN)$. One can similarly define the perversely-exotic t-structure on $\wti \cN \times^\bL_\g \wti \cN$, classes of simple objects in whose heart give the canonical basis of $\wh \cH_{G^\vee}$, see \cite[Theorem~54]{Bez16}.

\subsection{The main results}
\subsubsection{Noncommutative resolutions and perversely-exotic t-structures}
The goal of this paper is to implement a similar strategy to define and study canonical bases in the equivariant K-theory of the resolutions of affine Schubert varieties in type A, as well as their Steinberg-type versions.

Namely, consider the group $\operatorname{GL}_n$, its affine Grassmannian $\Gr = \Gr_{\operatorname{GL}_n}$, and a resolution $p\colon \wti \Gr^{\ula} \rightarrow \ol \Gr^\la$ by the convolution diagram of an affine Schubert variety (here $\ula = (\la_1, \hdots, \la_m)$ are fundamental coweights, $\la = \sum_{i} \la_i$). 
Our first result is the construction of a \textit{noncommutative counterpart of this resolution}. Similarly to the Springer case, we construct a $\operatorname{GL}_n(\cO)\rtimes \bG_m$-equivariant vector bundle $\cE_\ula$ on $\wti \Gr^\ula$, which is a relative tilting generator: the functor $\bR p_* \shhom(\cE_\ula, -)$ defines derived equivalences (see Theorem~\ref{thm: noncommutative resolution of affine Schubert variety})
\begin{align*}
D^b(\Coh(\wti \Gr^\ula)) &\simeq D^b(\cA_\ula\mathrm{-mod});\\
D^b(\Coh^{\operatorname{GL}_n(\cO) \rtimes \bG_m} (\wti \Gr^\ula)) &\simeq D^b(\cA_{\ula}\mathrm{-mod}^{\operatorname{GL}_n(\cO) \rtimes \bG_m}),
\end{align*}
where $\cA_\ula = \bR p_* \shhom(\cE_\ula, \cE_\ula)$ is a sheaf of noncommutative algebras on $\ol \Gr^\la$ (here $\bG_m$ stands for the loop rotation torus).  

For two sequences of fundamental coweights $\ula^1, \ula^2$ such that $\sum_i |\la^1_i| = \sum_i |\la^2_i| = N$, we consider the (derived) Steinberg-type varieties $\wti \Gr^{\ula^1} \times^\bL_{\Gr} \wti \Gr^{\ula^2}$ (in fact, one needs to take the derived product not over ${\Gr}$, but over a certain Beilinson--Drinfeld deformation; we explain it and the reason for considering it in Sections~\ref{sec: modified Lusztig's correspondence} and \ref{subsec: tiltings for steinbergs}, but ignore this issue in the Introduction); here $|\om_i| = i$ for a fundamental $\om_i$. 
There is a tilting object $\cE_{\ula^1} \boxtimes \cE^*_{\ula^2}$, which similarly defines the derived equivalence (see~\eqref{eq: derived equivalence for steinbergs}):
\begin{equation*}
D^b \Coh^{\operatorname{GL}_n(\cO) \rtimes \bG_m} ( \wti \Gr^{\ula^1} \times^\bL_{\Gr} \wti \Gr^{\ula^2}) \simeq D^b(\cA_{\ula^1} \otimes^\bR_{\cO_{\Gr}} \cA_{\ula^2}^\op\mods^{\operatorname{GL}_n(\cO) \rtimes \bG_m}).
\end{equation*}

This allows us to define the perversely-exotic t-structures on $D^b\Coh^{\operatorname{GL}_n(\cO) \rtimes \bG_m} (\wti \Gr^\ula)$ and $D^b \Coh^{\operatorname{GL}_n(\cO) \rtimes \bG_m} ( \wti \Gr^{\ula^1} \times^\bL_{\Gr} \wti \Gr^{\ula^2})$ as the perverse coherent t-structures on the categories of        $\operatorname{GL}_n(\cO) \rtimes \bG_m$-equivariant $\cA_{\ula}$ and $\cA_{\ula^1} \otimes^\bR_{\cO_{\Gr}} \cA_{\ula^2}^\op$-modules respectively.
The base field may have characteristic zero or sufficiently large positive characteristic.

Note that for such t-structure to be well-behaved, it is essential that dimensions of $G(\cO) \rtimes \bG_m$-orbits in $\ol \Gr^\la$ have the same parity (see \cite{AB10}).

\subsubsection{Description of the basis}
The hearts of the t-structures described above are of finite length, and hence the classes of simple objects form bases in $K^{\GL_n(\cO) \rtimes \bG_m} (\wti \Gr^\ula)$ and $K^{\GL_n(\cO) \rtimes \bG_m} (\wti \Gr^{\ula^1} \times_{\Gr} \wti \Gr^{\ula^2})$. We call them the perversely-exotic bases. Our goal is to describe it algebraically: namely, we first explicitly identify it with the parabolic Kazhdan--Lusztig basis in a certain cell quotient and then, using the result of \cite{SV00}, relate it to the Lusztig canonical basis for the quantum loop group. Let us describe the main results.

\

For $\ula = (\la_1, \hdots, \la_m)$ as above, take integers $k_i$ such that $\la_i = \om_{k_i}$; denote $\ukey = (k_1, \hdots, k_m)$. Let $d \in \bZ_{\geq 0}$. We define 
$'\ukey = (\underbrace{n, \hdots, n}_d, k_1, \hdots, k_m)$, it is a composition of $N+nd$, where $N=\sum_i k_i$. 
To a pair $\underline{\lambda}^1, \underline{\lambda}^2$ with $|\underline{\lambda}^1|=|\underline{\lambda}^2|$ one can associate tuples $'\underline{k}^1, '\underline{k}^2$ which determine the {\emph{parabolic}} version of the affine Hecke algebra modules that we denote by ${}^{'\underline{k}^1}\mathcal{H}^{'\underline{k}^2}$. By definition, ${}^{'\underline{k}^1}\mathcal{H}^{'\underline{k}^2}$ is obtained from $\mathcal{H}$ (affine Hecke algebra for $\operatorname{GL}_{N+nd}$) via left symmetrization by the idempotent coming from ${'\underline{k}^1}$ and right symmetrization by the idempotent coming from ${'\underline{k}^2}$ (see Section~\ref{subsec: canonical bases in parabolic hecke} for details). Similarly, one can define the parabolic anti-spherical module ${}^{'\underline{k}}\mathcal{H}^{\mathrm{asph}}$ obtained from the anti-spherical module of $\mathcal{H}$ by partial symmetrization. The spaces ${}^{'\underline{k}^1}\mathcal{H}^{'\underline{k}^2}$,  ${}^{'\underline{k}}\mathcal{H}^{\mathrm{asph}}$ have Kazhdan--Lusztig canonical bases. Moreover, ${}^{'\underline{k}^1}\mathcal{H}^{'\underline{k}^2}$ has a filtration by two-sided cells labeled by a subset of nilpotent orbits for $\operatorname{GL}_{N+nd}$. Similarly, ${}^{'\underline{k}}\mathcal{H}^{\mathrm{asph}}$ has a filtration by cells. We will denote by ${}^{'\underline{k}^1}\mathcal{H}^{'\underline{k}^2}_{\leq n}$, ${}^{'\underline{k}}\mathcal{H}^{\mathrm{asph}}_{\leq n}$ their quotients by the subspaces generated by canonical basis elements corresponding to nilpotents that have at least $n+1$ Jordan blocks. 
In the following Theorem we identify the perversely-exotic basis with the Kazhdan--Lusztig canonical basis.  
\begin{thm}[See  Theorem \ref{thm: exotic in convolution via Hecke}]
\label{thm:main thm hecke} 
\begin{enumerate}[(a)]
\
\item For any $d \geqslant 0$ there exist explicit homomorphisms:
\begin{align}\label{eq:maps from hecke to convolution}
{}^{'\underline{k}^1}\mathcal{H}^{'\underline{k}^2}_{\leq n} &\rightarrow 
K^{\GL_n(\cO) \rtimes \bG_m} (\wti \Gr^{\ula^1} \times_{\Gr} \wti \Gr^{\ula^2}), & {}^{'\underline{k}}\mathcal{H}^{\mathrm{asph}}_{\leq n} &\rightarrow K^{\GL_n(\cO) \rtimes \bG_m} (\wti \Gr^\ula) 
\end{align}
which are injective for $d=0$ and are isomorphisms for $d>0$.

\item The homomorphisms (\ref{eq:maps from hecke to convolution}) send Kazhdan--Lusztig canonical basis to the perversely-exotic basis.
\end{enumerate}
\end{thm}

By combining Theorem \ref{thm:main thm hecke} with the results of \cite{SV00} we obtain the following result describing the perversely-exotic basis via the quantum loop group. The language we use is that of the geometric quantum affine skew Howe duality, see \cite[Section~6]{CK18b}.

\begin{thm}[See Theorems~\ref{thm: quantum affine skew howe duality},~\ref{thm: main theorem description of basis}, Corollary~\ref{cor: action on module} for full statements] \label{thm:description via quantum group intro}
\begin{enumerate}[(a)]
\
\item \label{item: intro skew howe a} For any $d \geqslant 0$ there is an action of $U_v(\widehat \gl_{m + d})$ on
\begin{equation} \label{eq: intro skew howe homomorphism}
\bigoplus_{\ukey} K^{\GL_n(\cO) \rtimes \bG_m} (\wti \Gr^{\ukey} ),
\end{equation}
where the sum runs over all $\underline k = (k_1, \hdots, k_m)$ with $1 \leq k_i \leq n-1$ and with fixed $\sum_{i} k_i$.

\item \label{item: intro skew howe b} Perversely-exotic basis in~\eqref{eq: intro skew howe homomorphism} corresponds to the Lusztig canonical basis for quantum affine groups. It is compatible with embeddings $\dot U_v(\widehat \gl_{m + d}) \hookrightarrow \dot U_v(\widehat \gl_{m + d + 1})$.
\end{enumerate}
\end{thm}

Note that part~\ref{item: intro skew howe a} of the theorem was already proved in \cite[Section 6]{CK18b} and constitutes an important step in the diagrammatic realization of \eqref{eq: intro skew howe homomorphism}. Our proof is different and has some advantages\footnote{There may be a minor issue (which does not affect the main results of this paper) with Cautis--Kamnitzer's argument in \cite[Section 6]{CK18b} for the case $m + d = 2$ (i.e., affine type A$_1$). The  relation (5) in \cite[Section 5.1]{CK18b} is not correct as stated in this case and should instead be replaced by the more complicated relation
\begin{equation*}\label{eq:correct_Serre}
E_1E_0^{(3)}-E_0E_1E_0^{(2)}+E_0^{(2)}E_1E_0-E_0^{(3)}E_1=0.
\end{equation*}
One should then verify this relation manually in the course of the proof of \cite[Lemma 5.5]{CK18b}. By contrast, our proof does not require checking any relations and goes through in this case without modification.}. For example, in \cite[10.1]{CK18b} it is conjectured that a certain homomorphism from affine Hecke algebra is surjective (even before tensoring with $\bC(q)$). This is evident from our construction. The kernel of this map has a geometric interpretation (compare with the discussion in Section \ref{sec:descr_of_exotic_via_canonical}). We also note that we prove considerably more than state here (see Theorem~\ref{thm: quantum affine skew howe duality}) and that the precise statement of part~\ref{item: intro skew howe b} is elaborated in Corollary~\ref{cor: action on module}.

We also mention that compatibility of bases for different $d$ in Theorem~\ref{thm:description via quantum group intro} is proved using the reduction to positive characteristic and studying quantizations therein, see Section~\ref{sec:tiltings_for_lambda_and_shift}.

\subsection{Outline of arguments}

Note that a result related to ours was proved in \cite{FF21}: there, the perverse coherent basis in $K^{\GL_n(\cO) \rtimes \bG_m}(\Gr_{\GL_n})$ was related to the Lusztig dual canonical basis. Our strategy for the proof is similar.

As we have already mentioned, Theorem \ref{thm:description via quantum group intro} is a formal consequence of Theorem \ref{thm:main thm hecke} combined with  \cite{SV00}. So, we explain how to prove Theorem \ref{thm:main thm hecke}.

\subsubsection{Modified Lusztig correspondences}
 As a first step, we employ an idea due to Lusztig \cite{L81}, later developed by Mirkovi\'c--Vybornov \cite{MV07, MVK22} and Cautis--Kamnitzer \cite[Section 3.3]{CK18b}, that the geometry of affine Grassmannian is related to the geometry of nilpotent orbits in type A. We propose a slight modification of this construction. 

Namely, recall that to $\ula$ as above we associate $'\underline{k}$.
Consider the parabolic Springer resolution $\wti \cN^{'\ukey} \rightarrow \cN^{'\ukey}$ for the group $\GL_{N + dn}$, where $N = \sum_i k_i$. Recall that $\wti\cN^{'\ukey}$ parametrizes pairs $(x,F_\bullet)$ consisting of a nilpotent element $x$ and a partial flag $F_{\bullet}$ that is strictly preserved by $x$. We will denote by  $\wti{\cN}^{'\ukey}_{\leq n}$ the open subset consisting of pairs $(x,F_{\bullet})$ such that $x$ has at most $n$ Jordan blocks.

To prove Theorem \ref{thm:description via quantum group intro}, we define functors between derived categories of equivariant coherent sheaves:
\begin{equation}\label{eq:funct nilp to conv}
\begin{aligned}
D^b\operatorname{Coh}^{\GL_{N + dn} \times \bG_m}(\wti \cN^{'\ukey}_{\leq n}) &\rightarrow  D^b\operatorname{Coh}^{\GL_n(\cO) \rtimes \bG_m} (\wti \Gr^\ula); \\
D^b\operatorname{Coh}^{\GL_{N + dn} \times \bG_m}(\wti \cN^{' \ukey^1}_{\leq n} \times^\bL_{\gl_{N + dn}} \wti \cN^{'\ukey^2}_{\leq n})  &\rightarrow D^b\operatorname{Coh}^{\GL_n(\cO) \rtimes \bG_m} (\wti \Gr^{\ula^1} \times^\bL_\Gr \wti \Gr^{\ula^2})
\end{aligned}
\end{equation}
and prove that they are compatible with convolution, perversely-exotic t-structures, and send irreducible perversely-exotic sheaves to irreducible ones. These functors are given by what we call {\emph{modified Lusztig’s correspondences}}. 
We then prove that for $d > 0$ functors \eqref{eq:funct nilp to conv} induce {\it isomorphisms} on equivariant K-theory (note that the original construction for $d=0$ does not induce an isomorphism, see e.g. \cite[Section~5]{FF21}):
\begin{equation*} 
\begin{aligned}
K^{\GL_{N + dn} \times \bG_m}(\wti \cN^{'\ukey}_{\leq n}) &\xrightarrow{\sim} K^{\GL_n(\cO) \rtimes \bG_m} (\wti \Gr^\ula); \\
K^{\GL_{N + dn} \times \bG_m}(\wti \cN^{' \ukey^1}_{\leq n} \times_{\gl_{N + dn}} \wti \cN^{'\ukey^2}_{\leq n}) &\xrightarrow{\sim} K^{\GL_n(\cO) \rtimes \bG_m} (\wti \Gr^{\ula^1} \times_\Gr \wti \Gr^{\ula^2}).
\end{aligned}
\end{equation*}
See Section~\ref{sec: modified Lusztig's correspondence} for details on the construction and the results.










\subsubsection{Parabolic Bezrukavnikov equivalences}
The above construction allows us to reduce the problem about coherent sheaves on affine Grassmannian-type varieties to a problem about coherent sheaves on Springer-type varieties. As the next step, we need to construct   the identifications
\begin{align*}
{}^{'\underline{k}^1}\mathcal{H}^{'\underline{k}^2}_{\leq n} &\simeq K^{\GL_{N + dn} \times \bG_m}(\wti \cN^{' \ukey^1}_{\leq n} \times_{\gl_{N + dn}} \wti \cN^{'\ukey^2}_{\leq n}), & {}^{'\underline{k}}\mathcal{H}^{\mathrm{asph}}_{\leq n} &\simeq K^{\GL_{N + dn} \times \bG_m}(\wti \cN^{'\ukey}_{\leq n})
\end{align*}
and prove that the Kazhdan--Lusztig canonical basis in the LHS corresponds to the perversely-exotic basis in the RHS. For this, we use the parabolic variant of Bezrukavnikov equivalence \cite{Bez16}, recently proved by Chen and Dhillon \cite{CD23}, which reduces the problem to one about constructible perverse sheaves on an affine flag variety, for which the powerful tools (e.g. Deligne weights theory) are available, and which is well-studied.

In order to reduce the problem to a statement in the constructible world, we study what happens with t-structures under parabolic Bezrukavnikov's equivalences. The corresponding statement is proved in Theorem~\ref{thm: t-structures under bezrukav equivalences}. Note that for the full (non-parabolic) case, it is \cite[Theorem~54]{Bez16}. For the parabolic case, partial results appeared in \cite[2.D]{ACR18}, \cite[6.5]{BL23}, \cite[A.0.8(5)]{Pro26}.



\subsection{Directions for future research}

\subsubsection{Comparison with Cautis--Koppensteiner t-structures}
t-structures on the categories of equivariant coherent sheaves on $\wti \Gr^{\ula}$ were also defined in \cite{CK20} by completely different methods. 
We conjecture that these t-structures coincide.

\subsubsection{Projective basis and pairing}\label{subsec:projective basis and pairing}
In the heart of the t-structure we construct, there is also a ``dual'' basis, formed by the classes of projective objects (as opposed to simple objects). It would be interesting to understand it. Note that the equivariant K-theory of $\widetilde{\operatorname{Gr}}^{\underline{\lambda}}$ admits a natural nondegenerate $\Ext$-pairing which is well-defined since $\widetilde{\operatorname{Gr}}^{\underline{\lambda}}$ is smooth and projective. It would be interesting to see if this pairing is compatible with these two bases. Note that an analogous pairing together with bases formed by simple and projective objects is the central object of \cite[Sections 5, 6]{BM13} where the authors prove Lusztig's conjectures. In the situation of {\it loc.cit.} the pairing is between equivariant K-theories of Springer fibers and Slodowy varieties. In our picture both Springer fibers and Slodowy varieties (of type $A$) live inside $\widetilde{\operatorname{Gr}}^{\underline{\lambda}}$ and there are two ``transversal'' filtrations on $\widetilde{\operatorname{Gr}}^{\underline{\lambda}}$ whose associated graded is either the direct sum of equivariant K-theories of Springer fibers or  the direct sum of equivariant K-theories of Slodowy varieties.

\subsubsection{Other types and alternative descriptions}
The problem of constructing noncommutative resolutions of affine Schubert varieties and studying the corresponding perversely-exotic t-structure exists for groups of arbitrary type, not only for $\GL_n$. We think it is an interesting direction for research (see Conjecture~\ref{conjecture: tilting is restriction}). 
Our description of the basis has a (skew) Howe duality flavor. It would be interesting to generalize our results to other settings in which Howe duality-type statements are known (see \cite{LXY26} and references therein for the related geometry at the Springer side). It is possible that the quantum group appearing in Theorem~\ref{thm:description via quantum group intro} should be replaced by its $\imath$-quantum version (see, for example, \cite{SW24}).  

For example, motivated by \cite{BER24}, Elijah Bodish suggested that one could consider the additional $\bZ/2\bZ$-equivariance, coming from the automorphism of the Dynkin diagram of type A$_{2n - 1}$, and define an action of an $\imath$-quantum group $\dot U^\imath_q(\wh{ \mathfrak{so}}_m)$ on the equivariant K-theory of the form $K^{\GL_{2n} \times \bG_m \times \bZ/2\bZ}(\wti \Gr^{\ula}_{\GL_{2n}})$ as well as a homomorphism from $\dot U^\imath_q(\wh{ \mathfrak{so}}_m)$ to the equivariant K-theory of the corresponding Steinberg varieties. It would be interesting to investigate whether the Lusztig correspondences used in this paper are compatible with this additional equivariance and allow us to construct the desired maps. By analogy with \cite{CK18b}, one may hope that a homomorphism of the above type could help construct diagrammatic realizations of the corresponding equivariant K-theories. 

\vspace{0.2cm}

Another natural goal is an algebraic description of the canonical basis in $K^{\GL_n(\mathcal{O}) \rtimes \bG_m}(\widetilde{\operatorname{Gr}}^{\underline{\lambda}})$ not involving any Schur--Weyl duality-type statements. Namely, note that one can canonically identify
\begin{equation}\label{eq:ident tensor product intro}
K^{G(\mathcal{O}) \rtimes \bG_m}(\widetilde{\operatorname{Gr}}^{\underline{\lambda}}) \simeq K^{G(\mathcal{O}) \rtimes \bG_m}(\operatorname{Gr}^{\lambda_1}) \otimes_{K^{G(\mathcal{O}) \rtimes \bG_m}(\operatorname{pt})} \otimes \ldots \otimes_{K^{G(\mathcal{O}) \rtimes \bG_m}(\operatorname{pt})}  K^{G(\mathcal{O}) \rtimes \bG_m}(\operatorname{Gr}^{\lambda_m}) 
\end{equation}
for an arbitrary group $G$ assuming that $\lambda_i$ are minuscule. By geometric Satake, one has $K(\Gr^{\la_i}) \simeq H^\bullet(\Gr^{\la_i}) \simeq H_{\IC}(\Gr^{\la_i}) \simeq V_{\la_i}$ --- an irreducible $G^\vee$-representation, and $K^{G(\cO) \rtimes \bG_m}(\Gr^{\la_i})$ is its deformation over $K^{G(\cO) \rtimes \bG_m}(\pt)$. Already in type A, it would be very interesting to axiomatically describe the canonical basis we get in the RHS of \eqref{eq:ident tensor product intro}. The pairing mentioned in Section~\ref{subsec:projective basis and pairing} might be helpful (compare with Lusztig's description of the canonical basis in \cite[Conjectures 5.12,  5.16]{Lus99}). An axiomatic description would be useful as it may be generalizable to other types. 

\subsubsection{Perversely-exotic sheaves on symplectic resolutions} The t-structures we study here are the correct ``lifts'' of perverse coherent t-structure from $\ol \Gr^{\la}$ to the resolutions $\wti \Gr^\ula$. In \cite{Dum25}, the first-named author suggested the notion of perverse coherent sheaves on an arbitrary symplectic singularity. We expect that this notion can be lifted to the notion of perversely-exotic sheaves on symplectic resolutions.

\subsection{The paper is organized as follows}
Section~\ref{sec: coherent sheaves on parabolic steinbergs} is devoted to studying the properties of a parabolic version \cite{CD23} of Bezrukavnikov's equivalence \cite{Bez16} for any reductive group. We introduce the notion of perverse coherent modules over dg-algebras in Subsection~\ref{subsec: perv coh modules over dg-algebras}. We prove the properties of t-structures under the equivalences in Subsection~\ref{subsec: t-structures under coh-const}, and the properties of the parabolic Kazhdan--Lusztig basis in Subsection~\ref{subsec: canonical bases in parabolic hecke}.

In Section~\ref{sec: modified Lusztig's correspondence}, we define a modification of the Lusztig and Mirkovi\'c--Vybornov correspondences, which relate the geometry of the affine Grassmannian to the geometry of nilpotent orbits in type~A. The main new feature of the modified correspondences is the isomorphism on equivariant K-theory, which is Corollary~\ref{cor: Psi define iso on K-theory}.

In Section~\ref{sec:noncomm_resol_conv_diagr}, we apply these results to construct the desired noncommutative resolutions, perversely-exotic t-structures, and describe the corresponding basis in equivariant K-theory. The main results are  Theorems~\ref{thm: exotic in convolution via Hecke},~\ref{thm: main theorem description of basis}.

\subsection*{Acknowledgments}
We would like to thank Elijah Bodish for fruitful  discussions. We learned many of the algebraic aspects of this paper from him. 
We are also indebted to Roman Bezrukavnikov for his constant support. The second author is grateful to Joel Kamnitzer, who, among many other things, explained to him that slices in affine Grassmannians are precisely the natural splittings in the Lusztig correspondence.  The second author is also grateful to Minh-T\^am Trinh for helpful discussions at an early stage of this project. We are also grateful to Elijah Bodish, Sabin Cautis, Michael Finkelberg, Joel Kamnitzer, Dinakar Muthiah and Minh-T\^am Trinh for useful comments on an earlier version of this paper. The second author was supported by the Simons Foundation Award
888988 as part of the Simons Collaboration on Global Categorical Symmetries. 

\section{Coherent sheaves on parabolic Steinberg varieties} \label{sec: coherent sheaves on parabolic steinbergs}


In this section, we study the categories of equivariant coherent sheaves on parabolic Springer resolutions and Steinberg varieties, and generalize some of the constructions of \cite{BM13, Bez16} to the parabolic case.

We work over the field $\Bbbk = \ol \bQ_\ell$. All the results which do not involve coherent-constructible equivalences are true over any algebraically closed field of sufficiently large positive characteristic.

\subsection{Perverse coherent modules over dg-algebras} \label{subsec: perv coh modules over dg-algebras}
For this subsection, let $X$ be a finite type scheme and $H$  an algebraic group acting on $X$ with a finite number of orbits. Let $p$ be a strictly monotone and strictly comonotone perversity on the set of orbits, see \cite{AB10} for definitions (existence of $p$, in particular, implies that dimensions of adjacent orbits differ at least by 2). Consider the perverse coherent t-structure $( ^pD_X^{>0},  \ ^pD_X^{<0})$ on the bounded derived category $D_{\coh}^{b, H} X$ of $H$-equivariant coherent sheaves.

Suppose that $X$ is embedded in a smooth scheme $Y$, that $Y$ carries an action of $H$, and that the embedding is $H$-equivariant.
Consider the derived category $D^{b, H}_X Y$ of equivariant coherent sheaves on $Y$, set-theoretically supported on $X$ (see \cite[Section~2]{Orl11} for generalities on such categories, including equivalence of different natural ways to define them). One then easily defines the perverse coherent t-structure $( ^pD_X^{> 0} Y, \ ^pD_X^{< 0} Y)$ on this category, generalizing the construction of \cite{AB10} in a straightforward way. The construction of coherent IC-extension also passes without modification, and simple objects in the heart of this t-structure are the same as in the heart of $( ^pD_X^{>0},  \ ^pD_X^{<0})$.

Now, let $R$ be an $\cO_Y$-coherent $H$-equivariant sheaf of nonpositively graded $\dg$-algebras on $Y$, set-theoretically supported on $X$. Denote by $R\mods^H_\dg$ the dg-category of $\cO_Y$-coherent $H$-equivariant dg-modules over $R$.  
Consider also its (bounded) derived category, 
to be denoted $D^{b, H}_\coh R\mods$ (see \cite[Section~2]{AK17}, \cite[Tag~0FQS]{SP} for generalities). It has the standard dg-enhancement, which we keep in mind. 




Our goal in this subsection is to define the perverse coherent t-structure on $D^{b, H}_\coh R\mods$ and establish its properties, analogous to \cite{Bez00, AB10}. In the case of (non-dg) algebras it is implicit in \cite[6.2]{BM13} and \cite[11.2]{Bez16}; for modules over Harish-Chandra Lie algebroids, it is \cite[Section~3]{Dum25}. The dg-case is completely similar.

Namely, the existence of the t-structure itself follows formally from the existence of this t-structure on $D^{b, H}_X Y$ (see proof of Theorem~\ref{thm: perverse coherent t-structure for dg-algebras} below). A minute of reflection is needed to be assured about the existence of the IC-extension functor (analogous to \cite[Section~4]{AB10}). In view of it being well-documented in the literature in various contexts (see above), we do not repeat it here in full detail, but rather sketch it in the proof of Theorem~\ref{thm: perverse coherent t-structure for dg-algebras}.


Let $\Forg: D^{b, H}_\coh R\mods \rightarrow D^{b, H}_X Y$ be the functor which forgets the structure of $R$-module.
Let $^pD_R^{> 0}$ and $^pD_R^{<0}$ be the full subcategories of $D^{b, H}_\coh R\mods$, defined as $\Forg^{-1}(^pD_X^{>0} Y)$ and $\Forg^{-1}(^pD_X^{<0} Y)$ respectively.
\begin{thm} \label{thm: perverse coherent t-structure for dg-algebras}
Suppose the induction functor $- \otimes^\bR_{\cO_Y} R: D^{b, H}_X Y \rightarrow D^{b, H}_X Y$ is right t-exact with respect to perverse coherent t-structure. Then \ 
\begin{enumerate}
    \item 
$(^pD_R^{> 0},\ ^pD_R^{<0})$ define a t-structure on $D^{b, H}_\coh R\mods$.
The heart of this t-structure is called the category of perverse coherent $R$-modules, and will be denoted $\cP_\coh^H (R\mods)$.

\item For every locally closed $H$-invariant subscheme $Z \hookrightarrow X$, the functor of IC-extension $\IC(Z, -): \cP_\coh^H (R|_Z\mods) \rightarrow \cP_\coh^H (R\mods)$ is defined.
An object of $\cP_\coh^H (R\mods)$ is simple if and only if it is isomorphic to an object of the form $\IC(O, V)$, where $O$ is an $H$-orbit, and $V$ is a simple $H$-equivariant $R|_O$-module.

\item $\cP_\coh^H (R\mods)$ is an Artinian and Noetherian abelian category.
\end{enumerate}
\end{thm}

\begin{proof}
The first part of the theorem immediately follows from right t-exactness of induction by applying \cite[Theorem~2.1.2]{Pol07}.

Let us discuss the IC-extension functor. We go through the lemmata of \cite{AB10}, used in its construction, and mention how they fit into our context.

Lemmata~2.21,~2.22 of  \cite{AB10} are independent of the equivariance conditions, and hence are true in our setting.

For Lemma~2.13 of {\it loc.cit.}, see how in \cite[Lemma~3.1]{Dum25} it is deduced from the group equivariant setting, to the setting of simultaneous equivariance of a group and a Lie algebroid. Completely same argument works for the case of an additional structure of a module over a dg-algebra.

See \cite[Lemma~3.10]{Dum25} for a proof that a coherent module on an open $H$-equivariant subscheme has some coherent extension to the whole scheme, as well as any morphism between two such modules (note that it uses that the dimensions of adjacent orbits differ at least by 2, and shows that the usual non-derived pushforward gives the required extension). Of course, this argument works in our setting.
Note also that in the construction of the IC-extension, it is shown in \cite[Lemma~3.11, Theorem~3.12]{Dum25}, that it is sufficient to use that there is {\it some} coherent extension for the abelian category, and not for the derived category, as done in \cite[Lemma~4.3]{AB10}.

The above-mentioned lemmata are the only ingredients needed for the definition of the IC-extension. Using them, the construction proceeds along the lines of \cite{AB10}.

Now the classification of simples and the claims that $\cP_\coh^H (R\mods)$ is Artinian and Noetherian are proved along the same lines as in \cite{AB10} without changes.
\end{proof}

\begin{rem} \label{rem: simples are supported on zero coh}
We restrict our attention to the middle perversity function.
On an $H$-orbit $O$, we have 
\begin{equation*}
\cP^H_\coh(R|_O\mods) \simeq \bigl( H^0(R)|_O \mods \bigr) [\frac{1}{2} \dim O].
\end{equation*}
Here $H^0(R)$ is a sheaf of algebras --- the zeroth cohomology of $R$. In particular, simple objects appearing in the theorem above are simple $H_x$-equivariant modules over the algebra $H^0(R)|_x$ (here $x \in O$ is a closed point).
Note also that it follows that the action of $R$ on simple modules factors through $H^0(R)$. 
\end{rem}

It is possible that our assumption of right t-exactness is not necessary.



\subsection{Full flag variety case}

In this subsection, we treat the full flag variety case. Let $G$ be a connected reductive group, $\g$ its Lie algebra, $\cN \subset \g$ the nilpotent cone, and $T^* G/B = T^* \cB = \wti \cN \xrightarrow{s} \cN$ the Springer resolution. Let also $\wti \cN \times^\bL_\g \wti \cN$ be the derived Steinberg variety. We denote by $\bG_m$ the scaling action.

Recall that if $\cG$ is a generator of the category $D^b \Coh(X)$, then the functor $R\Hom(\cG, -)$ defines equivalences
\begin{align*}
D^b\Coh(X) &\simeq D^b(R\mods);\\
D^b\Coh^H(X) &\simeq D^b(R\mods^H).
\end{align*}
Here $R$ is the $\dg$-algebra of endomorphisms of $\cG$ in a dg-enhancement of $D^b\Coh(X)$ (defined up to quasi-isomorphism), and $H$ is any group acting on $X$ such that $\cG$ has an $H$-equivariant structure. If $\cG$ is also a tilting, then $R$ is a usual (non-dg) algebra. Note that even if one only wants to obtain the equivariant version of the equivalence, one would need to check the generator and possibly the tilting property in the non-equivariant category (and take the non-equivariant endomorphisms algebra).

In \cite{BM13}, there is a construction of tilting generators, parametrized by alcoves, see Theorem~1.8.2 {\it loc.cit.}; in the present paper, we work only with the one corresponding to the fundamental alcove (and occasionally, with the one corresponding to the antifundamental, which yields the dual tilting).
Here is the statement, see \cite[1.5]{BM13}:

\begin{thm} \label{thm: BM, tiltings}
There is a $G \times \bG_m$-equivariant vector bundle $\cE$ (tilting generator) on $\wti \cN$, such that 
\begin{enumerate}[a)]
\item the functor $\cF \mapsto \bR\Hom(\cE, \cF)$ defines equivalences
\begin{align*}
D^b \Coh(\wti \cN) &\simeq D^b A\mods; \\
D^b \Coh^{G \times \bG_m}(\wti \cN) &\simeq D^b A\mods^{G \times \bG_m},
\end{align*}
where $A = \Hom(\cE, \cE)$, and $\Ext^{> 0}(\cE, \cE) = 0$. 

\item For any closed $S \hookrightarrow \cN$, denote $\wti \cN_S = \wti \cN \times^\bL_\g S$ and $\cE_S$ --- the (derived) restriction of $\cE$ to $\wti \cN_S$. The functor $\cF \mapsto \bR\Hom(\cE_S, \cF)$ defines equivalences
\begin{align*}
D^b \Coh(\wti \cN_S) &\simeq D^b A_S\mods; \\
D^b \Coh^{G'}(\wti \cN_S) &\simeq D^b A_S\mods^{G'},
\end{align*}
where $A_S = \bR\Hom(\cE_S, \cE_S)$ is a $\dg$-algebra (defined up to a quasi-isomorphism), and $G'$ is any subgroup of $G \times \bG_m$ such that $S$ is $G'$-invariant.

\item If $S$ is an exact base change (see \cite[1.3]{BM13} for definitions), then $\wti \cN \times^\bL_\g S \simeq \wti \cN \times_\g S$, and $\bR\Hom(\cE_S, \cE_S) = \Hom(\cE_S, \cE_S)$ (higher Ext's vanish), so $A_S$ is an ordinary (non-derived) algebra.
\end{enumerate}
\end{thm}

\begin{proof}
When $S$ is an exact base change, this is~\cite[Theorem~1.5.1]{BM13}. The proof that $\cE_S$ is a weak generator works without assumption of base change exactness. Hence, an equivalence with derived category of $\dg$-modules over the $\dg$-algebra $\bR\Hom(\cE_S, \cE_S)$ follows.
\end{proof}

Viewing $A$ as a sheaf of algebras on $\cN$, one can consider the perverse-coherent t-structure on the category $D^b A\mods^{G \times \bG_m}$. We call the corresponding t-structure on $D^b \Coh^{G \times \bG_m}(\wti \cN)$ the {\it perversely-exotic}, see~\cite[6.2]{BM13}.

Now consider the (derived) Steinberg variety $\wti \cN \times^\bL_\g \wti \cN$. 
Consider the tilting generator $\cE$ on $\wti \cN$ and let $A$ be the corresponding noncommutative resolution. $\cE^*$ is also a tilting generator, see \cite[Theorem~1.8.2]{BM13}. The corresponding noncommutative resolution is $A^{\text{op}}$.
\begin{cor} \label{cor: sheaves on steinberg equivalent to bimodules}
We have the derived equivalence
\begin{equation} \label{eq: equivalence for noncommutative resolutions on steinberg}
D^b \Coh^{G \times \bG_m} (T^* \cB \times^\bL_\g T^* \mathcal B) \simeq D^b (A \otimes^\bR_{\cO_\g} A^{\text{op}}\mods)^{G \times \bG_m}.
\end{equation}
\end{cor}
\begin{proof}
Apply Theorem~\ref{thm: BM, tiltings} for the group $G \times G$ and the (non-exact) base change $\g \xhookrightarrow{\Delta} \g \times \g $ (compare also with \cite[11.2]{Bez16}).
\end{proof}
The above claim means that the sheaf $\cE \boxtimes \cE^*$ on $\wti \cN \times^\bL_\g \wti \cN$ is a generator (but not a tilting) of the derived (non-equivariant) category, and $A \otimes^\bR_{\cO_\g} A^{\text{op}}$ is its $\dg$-algebra of derived (non-equivariant) endomorphisms.

\begin{cor} \label{cor: perversely-exotic t-structure on steinberg}
Consider $A \otimes^\bR_{\cO_\g} A^{\text{op}}$ as a sheaf of $\dg$-algebras on $\g$, set-theoretically supported on $\cN$.
There is a well-defined perverse coherent t-structure on categories 
\eqref{eq: equivalence for noncommutative resolutions on steinberg}.
\end{cor}
\begin{proof}
Tensor product with each of $A, A^{\text{op}}$ is perverse right-t-exact because they both are concentrated in the zero homological degree. Since tensor product with $A \otimes^\bR_{\cO_\g} A^{\text{op}}$ is the composition of tensor products, it is also perverse right-t-exact. Hence Theorem~\ref{thm: perverse coherent t-structure for dg-algebras} is applicable.
\end{proof}

\subsection{Parabolic case}

Fix $P, Q \subset G$ parabolic subgroups. Let $L, M$ be the corresponding Levi subgroups. Denote $\cP = G/P$, $\cQ = G/Q$, $\cB = G/B$.
There are morphisms 
\begin{equation*}
T^*\cP \times^\bL_\g T^*\cQ \xleftarrow{\pr_1} (T^*\cP \times_\cP \cB) \times^\bL_\g (T^*\cQ \times_\cQ \cB)  \xhookrightarrow{i} T^* \cB \times^\bL_\g T^*\cB.
\end{equation*}
These are the Steinberg versions of the morphisms for the Springer case, considered in \cite[Section~4]{BM13}, see also \cite[6.0.3]{CD23}. 

Assume that the derived group of $G$ is simply connected. Then it has the characters $\rho_L, \rho_M$ --- the half-sums of positive roots of the corresponding Levi subgroups. We consider the adjoint derived functors
\begin{equation} \label{eq: correspondenses for parabollic steinbergs}
\begin{aligned} 
\pi^{\star} : D^{b}(\Coh^{G \times \bG_m} ( T^*\cP \times^\bL_\g T^*\cQ) )
&\to D^{b}(\operatorname{Coh}^{G \times \bG_m}(T^* \cB \times^\bL_\g T^*\cB)), 
\\ 
\pi^{\star}\mathcal{F} &= i_{*}pr_{1}^{*}(\mathcal{F}) \otimes \cO(\rho_L, \rho_M); \\
\pi_{\star} : D^{b}(\operatorname{Coh}^{G \times \bG_m}(T^* \cB \times^\bL_\g T^*\cB)) 
&\to D^{b}(\operatorname{Coh}^{G \times \bG_m}(T^*\cP \times^\bL_\g T^*\cQ)), 
\\ \pi_{\star}\mathcal{G} &= (pr_{1})_{*}i^{*}(\mathcal{G}(-\rho_L, -\rho_M)).
\end{aligned}
\end{equation}
By $\cO(\la_1, \la_2)$ here we denote the natural line bundle on $T^* \cB \times^\bL_\g T^* \cB$.

As above, we consider the sheaf $\cE_{BB} := \cE \boxtimes \cE^* \in D^b \Coh^{G \times \bG_m}(T^* \cB \times^\bL_\g T^* \cB)$.
\begin{lem} \label{lem: E_PQ is a generator}
The sheaf $\cE_{PQ} := \pi_\star \cE_{BB}$ is a generator of $D^b \Coh (T^* \cP \times^\bL_\g T^* \cQ)$.
\end{lem}
\begin{proof}
Again, apply \cite[Theorem~4.2]{BM13} for the group $G \times G$ and the (non-exact) base change $\g \xhookrightarrow{\Delta} \g \times \g$. Note that the exactness of the base change in {\it loc.cit.} is used only for the claim that $\pi_\star \cE_{BB}$ is a tilting object, which we do not make.
\end{proof}

It is easy to see that $\cE_{PQ}$ is of the form $\cE_P \boxtimes \cE^*_Q$, where $\cE_P$ and $\cE^*_Q$ are obtained by applying the Springer variants of the functor $\pi_\star$ to $\cE_B$ and $\cE_B^*$. We denote by $A_P$ and $A_Q^{\text{op}}$ the algebras of endomorphisms of $\cE_P$ and $\cE^*_Q$ respectively. Lemma~\ref{lem: E_PQ is a generator} then guarantees that there is an equivalence
\begin{equation} \label{eq: equivalence for noncommutative resolutions on parabolic steinberg}
D^b \Coh^{G \times \bG_m} (T^* \cP \times^\bL_\g T^* \mathcal Q) \simeq D^b (A_P \otimes^\bR_{\cO_\g} A_Q^{\text{op}}\mods^{G \times \bG_m}),
\end{equation}
and similarly for the $G$ (as opposed to $G \times \bG_m$)-equivariant case.


A complete analog of Corollary~\ref{cor: perversely-exotic t-structure on steinberg} then guarantees that there is the perverse t-structure on the RHS of~\eqref{eq: equivalence for noncommutative resolutions on parabolic steinberg}. We also call the corresponding t-structures on $\Coh^{G} (T^* \cP \times^\bL_\g T^* \mathcal Q)$ and $D^b \Coh^{G \times \bG_m} (T^* \cP \times^\bL_\g T^* \mathcal Q)$ the {\bf perversely-exotic t-structures}.

The following is the equivariant Steinberg version of \cite[Theorem~4.2]{BM13} (which deals with the Springer case, and the non-equivariant categories and t-structure), see also Remark~\ref{rem: simples to simples for pi^*}. 
\begin{prop} \label{prop: pi star is t-exact} 
\begin{enumerate}[(a)]
    \item 
 The functor $\pi^\star$ is t-exact with respect to the perversely-exotic t-structures on both sides.
\item The perversely-exotic t-structure on $D^b(\Coh^{G}(T^*\cP \times^\bL_\g T^*\cQ))$ is the unique t-structure, such that $\pi^\star$ is t-exact for the perversely-exotic t-structure on $D^b(\Coh^{G}(T^*\cB \times^\bL_\g T^* \cB))$.
\end{enumerate}
\end{prop}
\begin{proof}
Consider $\cF \in D_\exo^{> 0} (\Coh^{G \times \bG_m}(T^*\cP \times^\bL_\g T^*\cQ))$ --- an object in the positive part of the perversely-exotic t-structure. By definition, this is equivalent to $\bR\Hom(\cE_{PQ}, \cF)$ lying in the positive part of the perverse coherent t-structure on $\cN$. 
Rewrite
\begin{equation*}
\bR\Hom(\cE_{PQ}, \cF) \simeq \bR\Hom(\pi_\star \cE_{BB}, \cF) \simeq \bR\Hom(\cE_{BB}, \pi^\star \cF),
\end{equation*}
so this is equivalent to $\pi^\star \cF$ lying in the positive part of the perversely-exotic t-structure on $\Coh^{G \times \bG_m}(T^*\cB \times^\bL_\g T^*\cB)$. The same argument works for the negative parts of t-structures.

Thus both claims follow.
\end{proof}

\subsection{t-structures under coherent-constructible equivalences} \label{subsec: t-structures under coh-const}

Let $G^\vee$ be the Langlands dual group to $G$. We consider $G^\vee$ and related (ind-) varieties to be defined over a field $k$. Consider parabolic subgroups $P^\vee, Q^\vee \subset G^\vee$, corresponding to $P, Q \subset G$. Let $I^\vee_{P}, I^\vee_{Q}$ be the corresponding parahoric subgroups of the loop group $G^\vee_\cK = G^\vee(\cK)$, defined as pre-images of $P^\vee$ and $Q^\vee$ respectively under the evaluation $G^\vee(\cO) \rightarrow G^\vee$ map. Denote also $I^\vee = I^\vee_B$.

Let $D_{P^\vee Q^\vee}$ be the derived category of $\ell$-adic constructible sheaves on $I^\vee_{P} \backslash G^\vee_\cK / I^\vee_{Q} $. The following equivalence is the parabolic version of the ``two realizations of an affine Hecke algebra'' theorem, proved for $P = Q = B$ in \cite{Bez16}, and in \cite[Theorems~A,~B]{CD23} in the stated generality (based on earlier ideas of \cite{BL23}):
\begin{equation} \label{eq: parabolic two realizations}
D_{P^\vee Q^\vee} \simeq D^b \Coh^{G} (T^*\cP \times^\bL_\g T^*\cQ).
\end{equation}

Now we recall the functors which correspond to $(\pi_\star, \pi^\star)$ on the constructible side. Consider the natural projection $f: I^\vee_B \backslash G^\vee_\cK / I^\vee_B \rightarrow I^\vee_P \backslash G^\vee_\cK / I^\vee_Q$. There are the functors 
\begin{align*}
f^\star &:= f^*[\dim P^\vee / B^\vee + \dim Q^\vee / B^\vee] = f^![- \dim P^\vee / B^\vee - \dim Q^\vee / B^\vee], \\ 
f_\star &:= f_*[\dim P^\vee / B^\vee + \dim Q^\vee / B^\vee] = f_![ \dim P^\vee / B^\vee + \dim Q^\vee / B^\vee].
\end{align*}
They form an adjoint pair $(f_\star, f^\star)$. 

It is proved in \cite[Section~6.0]{CD23} that the equivalence~\eqref{eq: parabolic two realizations} is compatible with functors
\begin{equation} \label{eq: diagram of compatibilities for bezrukav equivalences}
\begin{tikzcd}[column sep=huge]
	{D_{B^\vee B^\vee}} & {D^b \text{Coh}^{G} (T^*\mathcal B \times^\bL_\mathfrak g T^*\mathcal B)} \\
	{D_{P^\vee Q^\vee}} & {D^b \text{Coh}^{G} (T^*\mathcal P \times^\bL_\mathfrak g T^*\mathcal Q)}
	\arrow["\cong", from=1-1, to=1-2]
	\arrow["{f_\star}", from=1-1, to=2-1]
	\arrow["{\pi_\star}", from=1-2, to=2-2]
	\arrow["{f^\star}", shift left=3, from=2-1, to=1-1]
	\arrow["\cong", from=2-1, to=2-2]
	\arrow["{\pi^\star}", shift left=3, from=2-2, to=1-2]
\end{tikzcd}
\end{equation}

There are the natural perverse (constructible) t-structures on the left-hand side of \eqref{eq: parabolic two realizations}. We also have the perversely-exotic t-structures on the right-hand sides. 
We show that these t-structures correspond to each other:
\begin{thm} \label{thm: t-structures under bezrukav equivalences}
Under derived equivalences \eqref{eq: parabolic two realizations}, the perverse constructible t-structure on the left-hand sides, corresponds to the perversely-exotic coherent t-structure on the right-hand sides. In particular, simple objects in the hearts are mapped to each other. 
\end{thm}

\begin{proof}

First, consider the case $P = Q = B$. Then the argument is completely in line with \cite[Theorem~54]{Bez16}, which is the same result for the monodromic equivalence.

Now, we deduce the case of arbitrary $P$, $Q$ from the case $P = Q = B$. Note that the functor $f^\star$ in \eqref{eq: diagram of compatibilities for bezrukav equivalences} is t-exact with respect to perverse t-structure (since $f$ is smooth).
According to Proposition~\ref{prop: pi star is t-exact}, the perversely-exotic t-structure on  ${D^b \text{Coh}^{G} (T^*\mathcal P \times^\bL_\mathfrak g T^*\mathcal Q)}$ is the unique t-structure, for which $\pi^\star$ is t-exact. The claim follows.
\end{proof}

For an affine Weyl group element $w$, we denote by $\cP_w$ the simple perversely-exotic coherent sheaf, corresponding to $w$ under the bijection of Theorem~\ref{thm: t-structures under bezrukav equivalences}.

\begin{rem}
Note that for the case of parabolic Springer (as opposed to Steinberg) variety and the corresponding coherent-constructible equivalence, an analog of Theorem~\ref{thm: t-structures under bezrukav equivalences} was proved in \cite[2.D.]{ACR18}, \cite[A.0.8(5)]{Pro26}. 
\end{rem}

\begin{rem} \label{rem: simples to simples for pi^*}
Note that for the non-equivariant case, \cite[Theorem~4.2(b)]{BM13} asserts that $\pi^\star$ maps simples to simples in the hearts of the corresponding t-structures only if a certain condition is satisfied (which does not always hold outside of type A, see \cite[Remark~2.29]{HKM24}). In our equivariant case, this property of $\pi^\star$ is always true, since it is easy to see and well-known for $f^\star$.
\end{rem}

We will need the following property of Bezrukavnikov-type equivalences. Assume that the base field on the constructible side is $\overline{\mathbb F}_q$ for a finite field $\mathbb F_q$. Then the categories on the constructible side admit the Frobenius automorphism. For the coherent side, consider the automorphism $\bold q$, induced by the scaling of $\g$ by $q \in \Bbbk$, $\bold q: \g \rightarrow \g$, $x \mapsto qx$.
\begin{lem} \label{lem: Frobenius under Roma}
Under the equivalences \eqref{eq: diagram of compatibilities for bezrukav equivalences}, the Frobenius automorphism  is intertwined with the functor $\bold q^*$. 
\end{lem}

As mentioned in \cite[11.1]{Bez16}, the proof is parallel to the proof of \cite[Proposition~1]{AB09}.


\subsection{Kazhdan--Lusztig bases for parabolic affine Hecke algebras and anti-spherical modules} \label{subsec: canonical bases in parabolic hecke}

Let $T \subset B \subset G$ be the maximal torus in Borel $B$. 
Recall that $L$ and $M$ are Levi subgroups, corresponding to parabolic $P$ and $Q$ respectively.
Let $\Delta_L$, $\Delta_M$, $\Delta=\Delta_G$ be their sets of roots and $I=I_G \subset \Delta_G$ be the set of simple roots. 
Let $\Lambda = \Lambda_T$ be the character lattice of $T$. 
Let $\widehat{W}:=W \ltimes \Lambda$ be the (extended) affine Weyl group and let $W_L$, $W_M$ be the finite parabolic subgroups corresponding to subsets $I_L, I_M$ of the set $I$ of simple roots in~$\Delta$. Let $\ell$ be the length function.


We consider the affine Hecke algebra $\mathcal{H}(\widehat{W})$, 
whose definition we now recall 
(we follow the normalization of \cite[Section~6.1]{R}). 
It is generated by $\{T_s,\, s \in I_G\}$ and $\{\theta_x,\, x \in \Lambda\}$ over $\mathbb{C}[v^{\pm 1}]$, subject to the following relations. For $s,t \in I$, let $m_{s,t} \in \mathbb{Z}_{\geqslant 1}$ be the order of $st$ in $\wh W$.

\begin{enumerate}[label=(\arabic*)]
\item $\underbrace{T_sT_{t} \ldots}_{{m_{s,t}}} = \underbrace{T_tT_{s} \ldots}_{{m_{s,t}}}$;
\item $\theta_{x}\theta_{y} = \theta_{x+y}$;
\item $T_s\theta_x = \theta_xT_s$~\text{if}~$s(x)=x$;
\item $\theta_x=T_s\theta_{x-\alpha}T_s$~\text{if}~$s=s_\alpha$~\text{and}~$s(x)=x-\alpha$;
\item $(T_s+v^{-1})(T_s-v)=0$.
\end{enumerate}
The relations above imply the following relation:
\begin{equation}\label{eq:rel_theta_vs_s}
T_{s_\alpha} \theta_{x} = \theta_{s_\alpha(x)} T_{s_\alpha} + (v^{-1}-v)\frac{\theta_{s_\alpha(x)}-\theta_{x}}{1-\theta_{-\alpha}}.
\end{equation}

The algebra $\mathcal{H}(\widehat{W})$ has an alternative presentation: it is a free module over $\mathbb{C}[v^{\pm 1}]$ with generators $\{T_w,\, w \in \widehat{W}\}$ subject to relations:
\begin{enumerate}[label=(\arabic*), leftmargin=*, align=left]
\item $T_wT_{w'}=T_{ww'}$~\text{if}~$\ell(ww')=\ell(w)+\ell(w')$,
\item $(T_s+v^{-1})(T_s-v)=0$.
\end{enumerate}

\begin{rem}\label{rem:norm_hecke}
Sometimes the Hecke algebra is defined using the relation $(T_s'+1)(T_s'-q)=0$. The identification is given by $T'_s \mapsto vT_s$, $q \mapsto v^2$.    
\end{rem}

Fix $x \in \Lambda$ and let us write $x=x_1-x_2$ with $x_1, x_2$ being dominant. Then
$
\theta_x = T_{x_1}T_{x_2}^{-1}.
$

The basis  $\{T_{w},\, w \in \widehat{W}\}$ is called {\emph{standard}}.
The {\emph{Kazhdan--Lusztig involution}} on $\mathcal{H}(\widehat{W})$ is determined by 
$\overline{v}=v^{-1}$, $\overline{T}_w=T_{w^{-1}}^{-1}$.
The {\emph{canonical}} basis $\{C_w,\, w \in \widehat{W}\}$ is the unique self-dual basis such that 
\begin{equation*}
C_w \in T_w + v^{-1}\bZ[v^{-1}] \underset{y < w}{\on{Span}_{\bC}}\{T_y\}.
\end{equation*}

\begin{example}
For $s \in I$ we have $C_s=T_s+v^{-1}$.
\end{example}

We start with the following well-known proposition, 
we refer the reader to \cite[Proposition 6.1.5]{R} and the references therein. 
To $x \in \Lambda$ we associate the line bundle $\mathcal{O}_{\mathcal{B}}(x)$ on $\cB$, defined  as $G \times^B \Bbbk_{-x}$; its lift to $T^*\cB$ is denoted similarly. 


Following \cite[Section~6.1]{R}, to a simple root $\alpha$ we associate a reduced subvariety $S_\alpha' \subset T^*\mathcal{B} \times_{\mathfrak{g}} T^*\mathcal{B}$ as follows. Let $\mathcal{P}_\alpha$ be the partial flag variety corresponding to $\alpha$ and let $\pi_{\alpha}\colon \mathcal{B} \rightarrow \mathcal{P}_\alpha$ be the natural map. Then
\begin{equation*}
S_\alpha' := \{(\mathfrak{b},\mathfrak{b}',x) \in T^*\mathcal{B} \times_{\mathfrak{g}} T^*\mathcal{B}\,|\, \pi_{\alpha}(\mathfrak{b})=\pi_{\alpha}(\mathfrak{b}')\}.
\end{equation*}
It has two irreducible components, one of them is the diagonal $\Delta_{T^*\mathcal{B}}$ and the other one is
\begin{equation*}
Y_\alpha := \{(\mathfrak{b},\mathfrak{b}',x) \in S'_{\alpha}\,|\, x \in \pi_{\alpha}(\mathfrak{b})^{\mathrm{nil}}\},
\end{equation*}
where $\bullet^{\mathrm{nil}}$ denotes the nilradical. Clearly, $Y_{\alpha}$ is a vector bundle on $\mathcal{B} \times_{\mathcal{P}_\alpha} \mathcal{B}$ of rank $\on{dim}(\mathfrak{g}/\mathfrak{b})-1$. For example, for $\mathfrak{g}=\mathfrak{sl}_2$, we have $Y_\alpha=\mathbb{P}^1 \times \mathbb{P}^1$.


\begin{prop}\label{eq:prop_explicit_ident_hecke_steinberg}
\begin{enumerate}
    \item \label{item: iso of hecke and k-theory}  There is an isomorphism
\begin{equation}\label{eq:iso_hecke_to_steinberg}
 \mathcal{H}(\widehat{W}) \iso K^{G \times \bG_m}(T^*\mathcal{B} \times^{\bL}_{\mathfrak{g}} T^*\mathcal{B}), 
\end{equation}
explicitly given by
\begin{align*}
\theta_x &\mapsto [\mathcal{O}_{\Delta_{T^*\mathcal{B}}}(x)], & T_{s_{\alpha}} \mapsto -v^{-1}[\mathcal{O}_{Y_\alpha}(-\rho, & \rho-\alpha)] - v^{-1}= -v^{-1}[\mathcal{O}_{S'_\alpha}].
\end{align*}

\item \label{item: perv exotic to canonical for full hecke} The identification (\ref{eq:iso_hecke_to_steinberg}) sends $C_w$ to $[\cP_{w}]$.

\end{enumerate}
\end{prop}
Recall that $\cP_w$ here stands for a simple perversely-exotic sheaf, see Theorem~\ref{thm: t-structures under bezrukav equivalences}.
\begin{proof}
For part \eqref{item: iso of hecke and k-theory}, see \cite{KL87} and \cite{CG97}.

Let us discuss part~\eqref{item: perv exotic to canonical for full hecke}.
If we forget $\bG_m$-equivariance, Theorem~\ref{thm: t-structures under bezrukav equivalences} establishes that classes of simple perversely-exotic sheaves correspond to classes of simple objects on the constructible side, which give the canonical basis in the $v = 1$ degeneration.

The way to prove the required fact in the $G \times \bG_m$-equivariant category using Lemma~\ref{lem: Frobenius under Roma} is parallel to \cite[Section~4.2]{Bez06a}. 
Namely, consider the multiplicity of a simple object in a standard object. It possesses a torus action, and we need to prove that its weights are {{non-positive}}. As explained in \cite{Bez06a}, Lemma~\ref{lem: Frobenius under Roma} implies that we have a functor from $D^b \Coh^{G \times \bG_m} (T^* \cP \times^\bL_\g T^* \cQ)$ to Weil perverse sheaves on $\Fl_{G^\vee}$, and the grading on the spaces of multiplicities corresponds to the weight filtration. The fact that the weights are {{non-positive}} then follows from the proof of Kazhdan--Lusztig conjectures, see e.g. \cite{KT90, BBD82}.

The compatibility of this with identification of part~\eqref{item: iso of hecke and k-theory} follows, since the perversely-exotic t-structure is characterized by braid positivity (see \cite[proof of Theorem~54]{Bez16}), and the braid group action in question is constructed using the same formulae as in isomorphism~\eqref{eq:iso_hecke_to_steinberg}, see \cite{BR12, R}.
\end{proof}

\begin{rem}\label{rem:image of Cs}
Clearly, 
\begin{equation*}
C_{s_\alpha} \mapsto -v^{-1}[\mathcal{O}_{Y_\alpha}(-\rho,\rho-\alpha)].
\end{equation*}
For $\mathfrak{g}=\mathfrak{sl}_2$ we see that $C_{s_\alpha} \mapsto -v^{-1}[\mathcal{O}_{\mathbb{P}^1}(-1) \boxtimes \mathcal{O}_{\mathbb{P}^1}(-1)]$.
\end{rem}

\subsubsection{Spherical--spherical case}\label{sec:sph_sph} The goal of this section is to prove the parabolic version of Proposition~\ref{eq:prop_explicit_ident_hecke_steinberg} and thus obtain an explicit description of the classes of the irreducible perversely-exotic sheaves in the parabolic case.

Let $w_{L,0}$, $w_{M,0}$ be the longest elements of $W_L$, $W_M$. 
Define
\begin{align*}
C_{w_{L,0}} &= \sum_{w \in W_L}v^{\ell(w)-\ell(w_{L,0})}T_w, & C_{w_{M,0}} &= \sum_{w \in W_M}v^{\ell(w)-\ell(w_{M,0})}T_w.
\end{align*}

Consider the subspace of $\mathcal{H}=\mathcal{H}(\widehat{W})$, defined as 
\begin{equation*}
{}^L\mathcal{H}{}^M:=C_{w_{L,0}}\mathcal{H} \cap \mathcal{H}C_{w_{M,0}}.
\end{equation*}
Note that $C_{w_{L,0}}T_s=vC_{w_{L,0}}=T_sC_{w_{L,0}}$ for a simple reflection $s \in W_L$ and $C_{w_{M,0}}T_s=vC_{w_{M,0}}=T_sC_{w_{M,0}}$ for a simple reflection $s \in W_M$, so 
\begin{equation*}
{}^L\mathcal{H}{}^M = \{x \in \cH~|~ T_wx=v^{\ell(w)}x,\, xT_u=v^{\ell(u)}x~\text{for}~w \in W_L, \,u \in W_M\}.
\end{equation*}

Let us define the {\emph{standard}} and {\emph{canonical}} bases in ${}^L\mathcal{H}^M$. This is well-known, so we omit the details. The bases are labeled by  $W_L \backslash \widehat{W}/W_M$ and we denote their elements by ${}^LT_{w}^M$, ${}^LC_{w}^{M}$. Every double coset $W_LwW_M$ contains the unique longest element to be denoted $n_w$, see \cite[Proposition~23]{Kob11}. We have:

\begin{equation*}
{}^LT{}^M_{w}:=\sum_{u \in W_LwW_M} v^{\ell(u)-\ell(n_{w})}T_u.
\end{equation*}
Both embeddings $C_{w_{L,0}}\mathcal{H} \subset \mathcal{H}$, $\mathcal{H}C_{w_{M,0}} \subset \mathcal{H}$ respect the bar involution, so we get the involution on ${}^L\mathcal{H}{}^M$.

The canonical basis $\{C_{n_w}\}$ is defined axiomatically as the unique self-dual basis  of ${}^L\mathcal{H}{}^M$ such that 
\begin{equation*}
C_{n_w}-{}^LT{}^M_{n_w} \in v^{-1}\bZ[v^{-1}]\underset{y < w}{\on{Span}_{\bC}}({}^LT_y{}^M).
\end{equation*}

The following proposition relating the canonical bases of ${}^L\mathcal{H}^M$ and $\mathcal{H}$ is standard. 

\begin{prop}\label{prop:comp_canonical_basis_L_H_M_H}
The embedding ${}^L\mathcal{H}{}^M \subset \mathcal{H}$ 
maps elements of the canonical basis to elements of the canonical basis (preserving the labeling).
\end{prop}

We are now ready to prove the following proposition. 
\begin{prop}
The embedding 
\begin{equation}\label{eq:embed_parab_parab_hecke}
K^{G \times \bG_m}(T^*\mathcal{P} \times^{\mathbb{L}}_{\mathfrak{g}} T^*\mathcal{Q}) \hookrightarrow K^{G \times \bG_m}(T^*\mathcal{B} \times_{\mathfrak{g}}^{\mathbb{L}} T^*\mathcal{B}) \simeq \mathcal{H} 
\end{equation}
induced by the functor $\pi^{\star}$ gives the identification: 
\begin{equation}\label{eq:iso_L_H_M_K_theory}
K^{G \times \bG_m}(T^*\mathcal{P} \times^{\mathbb{L}}_{\mathfrak{g}} T^*\mathcal{Q}) \iso {}^L\mathcal{H}^M.
\end{equation}
Classes of simple perversely-exotic sheaves in $K^{G \times \bG_m}(T^*\mathcal{P} \times^{\mathbb{L}}_{\mathfrak{g}} T^*\mathcal{Q})$ are equal to $C_{n_{w}}$.
\end{prop}
\begin{proof}
It follows from Theorem~\ref{thm: t-structures under bezrukav equivalences} and Remark~\ref{rem: simples to simples for pi^*} combined with Proposition~\ref{eq:prop_explicit_ident_hecke_steinberg}\eqref{item: perv exotic to canonical for full hecke} that classes of simples in $K^{G \times \bG_m}(T^*\mathcal{P} \times^{\mathbb{L}}_{\mathfrak{g}} T^*\mathcal{Q})$ map under (\ref{eq:embed_parab_parab_hecke}) to $C_{n_w} \in \mathcal{H}$. Now the claim follows from Proposition~\ref{prop:comp_canonical_basis_L_H_M_H}. 
\end{proof}

\subsubsection{Spherical--anti-spherical case} We again start with classical results. 
Define the algebra homomorphism from the finite Hecke algebra $\chi\colon \mathcal{H}(W) \rightarrow \mathbb{C}[v^{\pm 1}]$ by $\chi(T_w)=(-v^{-1})^{\ell(w)}$. In other words, $\chi(C_s)=0$ for $s \in I$. We define the
induced module $\mathcal{H}^{\mathrm{asph}}$ (anti-spherical module over $\mathcal{H}(\widehat{W})$) by
\begin{equation*}
\mathcal{H}^{\mathrm{asph}} := \mathcal{H}(\widehat{W}) \otimes_{\mathcal{H}(W)} \mathbb{C}[v^{\pm 1}].
\end{equation*}
Let $\varphi\colon \mathcal{H} \twoheadrightarrow \mathcal{H}^{\mathrm{asph}}$ be the surjection $h \mapsto [h \otimes 1]$. 
Abusing notation, for $x \in \Lambda$ we denote by $\theta_x \in \mathcal{H}^{\mathrm{asph}}$ the image of $\theta_x$. Clearly, $\{\theta_x,\, x \in \Lambda\}$ is a basis of $\mathcal{H}^{\mathrm{asph}}$ (over $\mathbb{C}[v^{\pm 1}]$). It also has a standard basis  $T^{\mathrm{asph}}_w=\varphi(T_w)$ labeled by elements $w \in \widehat{W}$ such that $w$ is shortest in $wW$. 
It also has the canonical basis $C_{w}^{\mathrm{asph}}$ with $w$ as above. It is not hard to see that $C_{w}^{\mathrm{asph}}=\varphi(C_w)$.

Recall the isomorphism of algebras $\mathcal{H} \simeq K^{G \times \bG_m}(T^*\mathcal{B} \times_{\mathfrak{g}}^{\mathbb{L}} T^*\mathcal{B})$ of Proposition~\ref{eq:prop_explicit_ident_hecke_steinberg} above. The latter  acts on $K^{G \times \bG_m}(T^*\mathcal{B})$ via convolution.

\begin{prop} \label{prop: isomorphism of springer with anti-spherical}
(i) There is an isomorphism of $\mathcal{H}$-modules
\begin{equation}\label{iso_asph_K_theory}
\mathcal{H}^{\mathrm{asph}} \iso K^{G \times \bG_m}(T^*\mathcal{B}),
\end{equation}
explicitly given by 
\begin{equation*}
\theta_x \mapsto [\mathcal{O}_{T^*\mathcal{B}}(x)].
\end{equation*}

(ii) The identification (\ref{iso_asph_K_theory}) sends $C_w^{\mathrm{asph}}$ to $[\mathcal{P}_w]$.
\end{prop}

\begin{example}
For $\mathfrak{g}=\mathfrak{sl}_2$, we see from Remark~\ref{rem:image of Cs} that $C_s$ indeed acts on $[\mathcal{O}_{T^*\mathcal{B}}]$ by zero (use that $H^*(\mathbb{P}^1,\mathcal{O}(-1))=0$).   
\end{example}

As in Section \ref{sec:sph_sph}, we define the spherical--anti-spherical object:
\begin{equation*}
{}^L\mathcal{H}^{\mathrm{asph}}:=C_{w_{L,0}}\mathcal{H}^{\mathrm{asph}}=\{h \in \mathcal{H}^{\mathrm{asph}} ~|~ T_wh=v^{\ell(w)}h~\text{for}~w \in W_L\}.
\end{equation*}
It has standard, canonical and ``$\theta$-type'' bases. Standard and canonical bases ${}^LT_w^{\mathrm{asph}}$, ${}^LC_w^{\mathrm{asph}}$ are defined as in Section \ref{sec:sph_sph} and are labeled by the elements $w \in \widehat{W}$ such that $w \in wW$ is {\emph{shortest}} and $w \in W_Lw$ is {\emph{longest}}. Let us describe the third basis ${}^L\theta_x$; it is labeled by $L$-dominant $x \in \Lambda$.

We start with the following lemma describing ${}^L\mathcal{H}^{\mathrm{asph}}$ in ``lattice'' terms. 
\begin{lem}\label{lem:expl_ident_H_asph_symm_funct}
We have 
\begin{equation}\label{eq:H_asph_expl_desc}
{}^L\mathcal{H}^{\mathrm{asph}} = \Big(\theta_{\rho_{\mathfrak{l}}} \cdot \prod_{\alpha \in \Delta^+_{L}}(v-v^{-1}\theta_{-\alpha})\Big) \cdot \mathbb{C}[T]^{W_L}[v^{\pm 1}] \subset \mathcal{H}^{\mathrm{asph}}.
\end{equation}
\end{lem}
\begin{proof}
For $s \in I$ we have the following equality in $\mathcal{H}^{\mathrm{asph}}$ (it follows from (\ref{eq:rel_theta_vs_s})):
\begin{equation*}
T_s \cdot \theta_x = -v^{-1}\theta_{s_\alpha(x)} + (v^{-1}-v)\frac{\theta_{s_\alpha(x)}-\theta_{x}}{1-\theta_{-\alpha}}.
\end{equation*}
This easily implies that the element 
\begin{equation}\label{eq:element_by_which_multiply}
\theta_{\rho_{\mathfrak{l}}} \cdot \prod_{\alpha \in \Delta^+_{L}}(v-v^{-1}\theta_{-\alpha})
\end{equation}
indeed lies in ${}^L\mathcal{H}^{\mathrm{asph}}$. It follows that every element obtained from (\ref{eq:element_by_which_multiply}) by the multiplication by $f \in \mathbb{C}[T]^{W_L}[v^{\pm 1}]$ is also in ${}^L\mathcal{H}^{\mathrm{asph}}$. It remains to prove that every element of ${}^L\mathcal{H}^{\mathrm{asph}}$ is a multiple of (\ref{eq:element_by_which_multiply}). This follows from the fact that any $h \in \mathcal{H}^{\mathrm{asph}}$ is divisible by $(v\theta_{\alpha/2}-v^{-1}\theta_{-\alpha/2})$ for any $\alpha \in \Delta^+_{L}$ (use that $T_{s_\alpha}h=v^{\ell(s_\alpha)}h$). {\color{red}{$\alpha/2$}}
\end{proof}

\begin{example}
For $G = L = SL_2$, the element (\ref{eq:element_by_which_multiply}) is equal to $v\theta_{\alpha/2}-v^{-1}\theta_{-\alpha/2}$, and the element $T_s$ indeed acts on it via the multiplication by $v$. Conversely, if $h=\sum_{n \in \mathbb{Z}}a_n(v^{\pm 1})\theta_{\alpha/2}^n$ satisfies $T_sh=vh$, then:
\begin{equation*}
\sum_{n \in \mathbb{Z}}va_n(1-\theta_{-\alpha})\theta_{n\alpha/2} = \sum_{n \in \mathbb{Z}}a_n(v^{-1}\theta_{-\frac{n\alpha}{2}-\alpha}-v^{-1}\theta_{n\alpha/2}-v\theta_{-n\alpha/2}+v\theta_{n\alpha/2}),
\end{equation*}
which is equivalent to 
\begin{equation*}
 v(a_{m+2}-a_{-m})+v^{-1}(a_{-m-2}-a_m)= 0
\end{equation*}
for any $m \in \mathbb{Z}$. We see that any solution is of the form $a_n=vb_{n-1}-v^{-1}b_{n+1}$ for arbitrary $b_i$ such that $b_i=b_{-i}$.
\end{example}

We are now ready to define the ``$\theta$-type'' basis in ${}^L\mathcal{H}^{\mathrm{asph}}$. It is labeled by 
\begin{equation*}
\Lambda_L^{+}=\{\eta \in \Lambda~|~ \langle \alpha,\eta\rangle \geqslant 0\,~\forall \alpha \in I_L\}.
\end{equation*}
For $\eta \in \Lambda_L^{+}$ using the identification obtained in Lemma \ref{lem:expl_ident_H_asph_symm_funct} we set 
\begin{equation*}
{}^L\theta_\eta := \Big(\theta_{\rho_{\mathfrak{l}}} \cdot \prod_{\alpha \in \Delta^+_{\mathfrak{l}}}(1-v^{-2}\theta_{-\alpha})\Big) \cdot \on{ch}V(\eta),
\end{equation*}
where $V(\eta)$ is the irreducible representation of $L$ with highest weight $\eta$. It follows from Lemma \ref{lem:expl_ident_H_asph_symm_funct} that this is indeed a basis. 

\

We are now ready to compare the geometric and algebraic approaches. 
\begin{prop}\label{prop:explicit_ident_with_aspherical}
(i) The embedding $\pi^{\star}\colon K^{G \times \bG_m}(T^*\mathcal{P}) \hookrightarrow K^{G \times \bG_m}(T^*\mathcal{B})=\mathcal{H}^{\mathrm{asph}}$ identifies $K^{G \times \bG_m}(T^*\mathcal{P})$ with 
${}^L\mathcal{H}^{\mathrm{asph}} \subset \mathcal{H}^{\mathrm{asph}}$ as in Lemma \ref{lem:expl_ident_H_asph_symm_funct} and is explicitly  given by:
\begin{equation}\label{eq:exp_ident_K_theory_sph_asph}
[\mathcal{V}_{T^*\mathcal{P}}(\eta)] \mapsto {}^L\theta_\eta.
\end{equation}  
(ii) The identification (\ref{eq:exp_ident_K_theory_sph_asph}) sends ${}^LC_w^{\mathrm{asph}}$ to $[\mathcal{P}_w]$.
\end{prop}
\begin{proof}
    Our goal is to compute $[\pi^\star(\mathcal{V}(\eta))]$. We have $\pi^{\star}(\mathcal{V}(\eta)) = i_*\on{pr}_1^{*}(\mathcal{V}(\eta))(\rho_L)$, where $\on{pr}_1$, $i$ are the maps:
\begin{equation*}
T^*\cP  \xleftarrow{\pr_1} T^*\cP \times_\cP \cB   \xhookrightarrow{i} T^* \cB.
\end{equation*} 
Let us separately describe the maps
\begin{equation*}
\on{pr}_1^*\colon \mathbb{C}[T]^{W_L}[v^{\pm 1}]=K^{G \times \bG_m}(T^*\mathcal{P})    \rightarrow  K^{G \times \bG_m}(T^*\mathcal{P} \times_{\mathcal{P}} \mathcal{B}) = \mathbb{C}[T][v^{\pm 1}],
\end{equation*}
\begin{equation*}
i_*\colon \mathbb{C}[T][v^{\pm 1}]  = K^{G \times \bG_m}(T^*\mathcal{P} \times_{\mathcal{P}} \mathcal{B}) \rightarrow K^{G \times \bG_m}(T^*\mathcal{B})=\mathbb{C}[T][v^{\pm 1}].
\end{equation*}
The map $\on{pr}^*_1$ is induced by the pull-back map for the natural morphism $p\colon \mathcal{B} \rightarrow \mathcal{P}$ so it sends $\mathcal{V}_{T^*\mathcal{P}}(\eta)$ to the pull-back of the vector bundle $\mathcal{V}_{\mathcal{B}}(\eta)$ to $T^*\mathcal{P} \times_{\mathcal{P}} \mathcal{B}$. In other words, $\on{pr}_1^*$ induces the natural inclusion $\mathbb{C}[T]^{W_L}[v^{\pm 1}] \hookrightarrow \mathbb{C}[T][v^{\pm 1}]$.

The morphism $i_*$ corresponds to the multiplication by the Euler class of the bundle $G \times^B (\mathfrak{n}/\mathfrak{n}_{\mathfrak{p}})$ so is the multiplication by $\prod_{\alpha \in \Delta^+_{\mathfrak{l}}}(1-v^{-2}\theta_{-\alpha})$. 

We conclude that  $\pi^\star$ is nothing else but the natural embedding ${}^L\mathcal{H}^{\mathrm{asph}} \hookrightarrow \mathbb{C}[T][v^{\pm 1}]$ described in Lemma \ref{lem:expl_ident_H_asph_symm_funct}. This implies part (i). Part (ii) follows from Proposition~\ref{prop: isomorphism of springer with anti-spherical}(ii) and the fact that $\pi^\star$ maps simples to simples, as is evident from \eqref{eq: diagram of compatibilities for bezrukav equivalences}.
\end{proof}

\begin{rem}
It is not hard to see that  $\pi_\star$ induces the Demazure operator $\Delta_{w_{L,0}}$ in K-theory. Note that Proposition \ref{prop:explicit_ident_with_aspherical} also allows us to explicitly describe the identification \eqref{eq:iso_L_H_M_K_theory} above by using the embedding 
\begin{equation*}
K^{G \times \bG_m}(T^*\mathcal{P} \times_{\mathfrak{g}}^{\mathbb{L}} T^*\mathcal{Q}) \hookrightarrow \operatorname{Hom}_{K^{G \times \bG_m}(\on{pt})}(K^{G \times \bG_m}(T^*\mathcal{Q}),K^{G \times \bG_m}(T^*\mathcal{P}))
\end{equation*}
given by the  convolution.   
\end{rem}




To summarize, we obtain the following result:

\begin{thm} \label{thm: simple perversely-exotic are canonical in parabolic hecke}
\begin{enumerate}[label=(\roman*)]\
    \item
Under the identification induced by (\ref{eq:iso_hecke_to_steinberg}), classes of simple perversely-exotic sheaves in $\cP_\exo^{G \times 
\bG_m} (T^* \cP \times^\bL_\g T^* \cQ)$ 
coincide with the canonical basis in ${}^P\wh \cH^Q$.
\item
Under the explicit identification given by (\ref{eq:exp_ident_K_theory_sph_asph}), classes of simple perversely-exotic objects in $\cP_\exo^{G \times 
\bG_m} (T^* \cP)$ coincide with the canonical basis in the anti-spherical module ${}^P\wh \cH^\textrm{asph}$. 
\end{enumerate}
\end{thm}

\section{Modified Lusztig correspondence} \label{sec: modified Lusztig's correspondence}

In this section, we introduce (in the most general form) certain correspondences between convolution diagrams in affine Grassmannians and  cotangent bundles to partial flag varieties as well as their derived Steinberg versions. Various instances of these correspondences appeared in papers \cite{MV07}, \cite{MVK22}, \cite[Section~3]{CK18a}, \cite[Section~2]{an}, \cite[Section~4]{FF21}, \cite[(2.1.19)]{Z16} and go back to Lusztig \cite{L81}. We therefore call them modified Lusztig  correspondences.

From now on and throughout the rest of the paper, we work over an algebraically closed field of zero or large enough characteristic $\Bbbk$.

\subsection{Lusztig correspondence} \label{subsec: Lusztig's correspondences}
We work with the group $\GL_n$ and its affine Grassmannian $\Gr = \Gr_{\GL_n}$.
We fix a collection of dominant coweights $\ula = (\la_1,\ldots,\la_m)$ such that every $\la_i$ equals some fundamental coweight $\omega_{k_i}$. 
We denote $\ukey = (k_1, \hdots, k_m)$.
The corresponding affine Grassmannian convolution diagram   $\widetilde{\operatorname{Gr}}^{\ula}_{\on{GL}_n}=\widetilde{\operatorname{Gr}}^{\ula}$ is defined by:
\begin{equation*}
\widetilde{\operatorname{Gr}}^{\ula} = \{L_{m} \subset L_{m-1} \subset \ldots \subset L_{1} \subset L_0\,|\, \on{dim}L_{i-1}/L_{i}=k_{i},\, z(L_i) \subset L_{i+1}\},
\end{equation*}
where $L_0:=\Bbbk^n[z]$, and all $L_i$ are $\Bbbk[z]$-submodules of $L_0$.

Set 
$N:=\sum_{i=1}^m k_i$. 
Following \cite[Section 3]{CK18a} we define:
\begin{equation*}
   \mathbb{X}_N := \{L \subset L_0\,|\, zL \subset L,~\operatorname{dim}(L_0/L) = N\}. 
\end{equation*}
One can show that $\bX_N$ is isomorphic to the quotient by the symmetric group $S_N$ of the Schubert variety in the Beilinson--Drinfeld Grassmannian, corresponding to the tuple of coweights  $\underbrace{(\om_1, \hdots, \om_1)}_N$.

We have a natural morphism
\begin{equation*}
\widetilde{\on{Gr}}^{\ul{\la}} \rightarrow {\mathbb{X}}_N,~L_\bullet \mapsto (L_{m}     \subset L_0).
\end{equation*}
Fix $\ula^1$, $\ula^2$ and consider the corresponding (derived) Steinberg-type variety:
\begin{equation*}
\operatorname{St}^{\underline{\la}^1,\underline{\la}^2}_{\operatorname{GL}_n} := \widetilde{\on{Gr}}^{\ul{\la}^1} \times_{\mathbb{X}_N}^\bL \widetilde{\on{Gr}}^{\ul{\la}^2}. 
\end{equation*}


Let $\widetilde{\mathcal{N}}_{\underline{k}}$ be the following parabolic Springer resolution (for group $\GL_N$):
\begin{equation*}
\widetilde{\mathcal{N}}_{\ul{k}} = \{\{0\}=F_m \subset F_{m-1} \subset \ldots \subset F_{1} \subset F_0 = \Bbbk^N;\, x \in \mathfrak{gl}_N\,|\, \on{dim}F_{i-1}/F_{i}=k_i,\, x(F_i) \subset F_{i+1}\}.
\end{equation*}
Let $\mathfrak{gl}_N^{\leq n} \subset \mathfrak{gl}_N$ be the open subset obtained by deleting the nilpotent elements that have at least $n+1$ Jordan blocks. Let $\widetilde{\mathcal{N}}^{\leq n}_{\underline{k}} \subset \widetilde{\mathcal{N}}_\ukey$ be its preimage. 
We also consider the corresponding Steinberg varieties:
\begin{equation*}
\operatorname{St}^{\leq n}_{\ul{k}^1,\ul{k}^2} :=\widetilde{\mathcal{N}}^{\leq n}_{\ul{k}^1} \times_{\mathfrak{gl}_N}^{\bL} \widetilde{\mathcal{N}}^{\leq n}_{\ul{k}^2}.\end{equation*}

Let $\mathbb{Y}_N$ be the $\operatorname{GL}_N$-torsor over $\mathbb{X}_N$ given by:
\begin{equation*}
    \mathbb{Y}_N = \{(L \subset L_0,  \psi ) \, | \, \psi \colon L_0/L \simeq \Bbbk^N\}.
\end{equation*}
We then obtain $\operatorname{GL}_N$-torsors over $\widetilde{\operatorname{Gr}}^{\ul{\la}}$, $\operatorname{St}^{\underline{\la}^1,\underline{\la}^2}_{\operatorname{GL}_n}$ given by $\widetilde{\operatorname{Gr}}^{\ul{\la}} \times_{\mathbb{X}_N} \mathbb{Y}_N$, $\operatorname{St}^{\underline{\la}^1,\underline{\la}^2}_{\operatorname{GL}_n} \times_{\mathbb{X}_N} \mathbb{Y}_N$ respectively. 
We have a natural morphism:
\begin{equation}\label{eq:morphism Y to gl}
\mathbb{Y}_N \rightarrow \mathfrak{gl}^{\leq n}_{N},~ (L,\psi) \mapsto (z \curvearrowright L_0/L = \Bbbk^N),
\end{equation}
where the identification $L_0/L = \Bbbk^N$ is given by $\psi$. 
This morphism extends to natural morphisms 
\begin{align}\label{eq:morphisms_Lusztig_corresp}
\widetilde{\operatorname{Gr}}^{\ul{\la}} \times_{\mathbb{X}_N} \mathbb{Y}_N &\rightarrow \widetilde{\mathcal{N}}_{\ul{k}}^{\leq n}, & \operatorname{St}^{\underline{\la}^1,\underline{\la}^2}_{\operatorname{GL}_n} \times_{\mathbb{X}_N} \mathbb{Y}_N &\rightarrow \operatorname{St}^{\leq n}_{\ul{k}^1,\ul{k}^2}
\end{align}
such that the following diagrams are {\emph{Cartesian}}:
\begin{align} \label{eq:cartesian_diagr_Y_to_X}
&\xymatrix{\widetilde{\on{Gr}}^{\ul{\la}} \times_{\mathbb{X}_N} \mathbb{Y}_N \ar[d] \ar[r] & \widetilde{\mathcal{N}}^{\leq n}_{\ul{k}} \ar[d] \\
\mathbb{Y}_N \ar[r] & \mathfrak{gl}^{\leq n}_N}, & 
&\xymatrix{\operatorname{St}^{\underline{\la}^1,\underline{\la}^2}_{\operatorname{GL}_n} \times_{\mathbb{X}_N} \mathbb{Y}_N \ar[d] \ar[r] & \operatorname{St}^{\leq n}_{\ul{k}^1,\ul{k}^2} \ar[d] \\
\mathbb{Y}_N \ar[r] & \mathfrak{gl}^{\leq n}_N}.
\end{align}
For example, the first morphism in (\ref{eq:morphisms_Lusztig_corresp}) produces a flag consisting of $F_i:=L_i/L_m \subset L_0/L_m$ from $L_\bullet$. Diagrams (\ref{eq:cartesian_diagr_Y_to_X}) induce Cartesian diagrams of the corresponding quotient derived stacks:
\begin{align}\label{eq:cartesian_diagr_stacks}
&\xymatrix{\widetilde{\on{Gr}}^{\ul{\la}} \ar[d] \ar[r] & \left. \widetilde{\mathcal{N}}^{\leq n}_{\ul{k}} \right/ \on{GL}_N \ar[d] \\
\mathbb{X}_N \ar[r] & \left. \mathfrak{gl}^{\leq n}_N \right/ \GL_N}, &
&\xymatrix{\operatorname{St}^{\underline{\la}^1,\underline{\la}^2}_{\operatorname{GL}_n}  \ar[d] \ar[r] & \left. \operatorname{St}^{\leq n}_{\ul{k}^1,\ul{k}^2} \right/ \on{GL}_N \ar[d] \\
\mathbb{X}_N \ar[r] & \left. \mathfrak{gl}^{\leq n}_N \right/ \on{GL}_N}
\end{align}

The Lusztig correspondences are given by:
\begin{equation}\label{eq:lus_corresp}
\begin{array}{@{}l@{\qquad}l@{}}
\bX_{N} \leftarrow \bY_{N} \rightarrow \mathfrak{gl}^{\le n}_{N}, \; &
\widetilde{\Gr}^{\ula} \leftarrow
\widetilde{\Gr}^{\ula} \times_{\bX_{N}} \bY_{N} \rightarrow \widetilde{\mathcal N}^{\leq n}_{\ukey},\\[0.6ex]
\multicolumn{2}{c}{
\St^{\ula^{1},\ula^{2}}_{\GL_{n}} \leftarrow
\St^{\ula^{1},\ula^{2}}_{\GL_{n}}\times_{\bX_{N}} \bY_{N} \rightarrow
\St^{\leq n}_{\ukey^{1}, \ukey^{2}}. }
\end{array}
\end{equation}

\subsection{Modified Lusztig correspondences} \label{subsec: modified lusztig}
We now slightly modify the Lusztig correspondences. This modification will become crucial in Section \ref{subsec:iso_K_thery}. We start with the following simple lemma.

\begin{lem}\label{lem:iso_shifts_gr_st}
(1) The natural morphism
\begin{equation}\label{eq:iso_gr_lambda_shift_by_omega_n}
\widetilde{\on{Gr}}^{\ul{\la}} \iso \widetilde{\on{Gr}}^{(\omega_n, \ul{\la})},~ L_{\bullet} \mapsto (zL_m \subset zL_{m-1} \subset \ldots \subset zL_1 \subset zL_0 \subset L_0)
\end{equation}
is an isomorphism. 

(2) The closed embedding 
\begin{equation}\label{eq:emb_X_N_shift}
\mathbb{X}_{N} \hookrightarrow \mathbb{X}_{N+n},~ (L \subset L_0) \mapsto (zL \subset L_0)
\end{equation}
combined with (\ref{eq:iso_gr_lambda_shift_by_omega_n}) induces a morphism 
\begin{equation}\label{eq:iso_st_shift}
\operatorname{St}_{\on{GL}_n}^{\ul{\la}^1, \ul{\la}^2} \rightarrow \operatorname{St}_{\on{GL}_n}^{(\omega_n, \ul{\la}^1),(\omega_n, \ul{\la}^2)}
\end{equation}
which is an isomorphism on the classical loci.
\end{lem}

\begin{proof}
Let's prove (1). The inverse is given by 
\begin{equation*}
L_{m+1} \subset L_{m} \subset \ldots \subset L_{2} \subset L_1 \subset L_0 \mapsto z^{-1}L_{m+1} \subset z^{-1}L_m \subset \ldots \subset z^{-1}L_2 \subset L_0,
\end{equation*}
it is well-defined because the conditions $\operatorname{dim}(L_0/L_1)=n$, $zL_0 \subset L_1$ force $L_1=zL_0$.

Part (2) follows from (1). 
\end{proof}

Fix a number $d \in \mathbb{Z}_{\geqslant 0}$. Set 
\begin{equation*}
'\mathbb{Y}_{N} := \{(L \subset L_0,\psi) \,|\, (L \subset L_0)\in {\mathbb{X}}_N,\, \psi\colon L_0/z^dL \simeq \Bbbk^{N+dn}\}.
\end{equation*}
Note that the natural morphism $'\mathbb{Y}_{N} \rightarrow \mathbb{X}_N$ is a $\GL_{N+dn}$-torsor. 
Morphisms (\ref{eq:emb_X_N_shift}), (\ref{eq:iso_st_shift})  induce a $\GL_{N+dn}$-equivariant morphism 
\begin{equation}\label{eq:corresp st base change shift}
\operatorname{St}^{\underline{\la}^1,\underline{\la}^2}_{\operatorname{GL}_n} \times_{{\mathbb{X}}_N} {'\mathbb{Y}}_N \rightarrow \operatorname{St}^{'\underline{\la}^1,'\underline{\la}^2}_{\operatorname{GL}_n} \times_{{\mathbb{X}}_{N+dn}} {\mathbb{Y}}_{N+dn}.
\end{equation}

Composing  \eqref{eq:iso_gr_lambda_shift_by_omega_n} and \eqref{eq:corresp st base change shift} with the Lusztig correspondences \eqref{eq:lus_corresp}, we obtain the modified Lusztig correspondences
\begin{align*} 
\widetilde{\on{Gr}}^{\underline{\la}} \leftarrow \widetilde{\operatorname{Gr}}^{'\ul{\la}} &\times_{\mathbb{X}_{N+dn}} \mathbb{Y}_{N+dn} \rightarrow \widetilde{\mathcal{N}}^{n}_{'\ul{k}}, &\operatorname{St}^{\underline{\lambda}^{1},\underline{\lambda}^2}_{\operatorname{GL}_n} \leftarrow \operatorname{St}^{\underline{\la}^1,\underline{\la}^2}_{\operatorname{GL}_n} \times_{{\mathbb{X}}_N} {'\mathbb{Y}}_N & \rightarrow \operatorname{St}^{\leq n}_{'\ul{k}^1,'\ul{k}^2},
\end{align*}
where $'\ul{\la}=(\underbrace{\omega_n,\ldots,\omega_n}_d,\ul{\la})$, $'\ul{k}=(\underbrace{n,\ldots,n}_{d},\ul{k})$. In other words, we have the morphisms of derived stacks: 
\begin{align} \label{eq:modified_lusztig_morphism_of_stacks}
\widetilde{\on{Gr}}^{\underline{\la}} &\rightarrow \left. \widetilde{\mathcal{N}}^{\leq n}_{'\ul{k}} \right/ \operatorname{GL}_{N+dn}, &\operatorname{St}^{\underline{\lambda}^{1},\underline{\lambda}^2}_{\GL_n}  &\rightarrow \left. \operatorname{St}^{\leq n}_{'\ul{k}^1,'\ul{k}^2} \right/ \operatorname{GL}_{N+dn}.
\end{align}
From now on, to simplify the notations, we assume $d=1$ but all results remain valid for arbitrary $d$.
Let us  compare morphisms (\ref{eq:cartesian_diagr_stacks}) and  (\ref{eq:modified_lusztig_morphism_of_stacks}). 

Set 
\begin{align*}
\mathfrak{gl}^n_{N+n} &= \{x \in \mathfrak{gl}_{N+n}\,|\, \operatorname{dim}(\operatorname{ker}x)=n\}, & \operatorname{St}^{n}_{'\ul{k}^1,'\ul{k}^2} &= \widetilde{\mathcal{N}}^{\leq n}_{'\ul{k}^1} \times_{\mathfrak{gl}^n_{N+n}}^\bL \widetilde{\mathcal{N}}^{\leq n}_{'\ul{k}^2}. 
\end{align*}
Note that $\mathfrak{gl}^n_{N+n}$ is nonsingular. We have a natural $\GL_{N+n}$-equivariant closed embedding $\mathfrak{gl}^n_{N+n} \subset \mathfrak{gl}^{\leq n}_{N+n}$. We have a morphism 
\begin{equation*}
\operatorname{St}^n_{'\ul{k}^1,'\ul{k}^2} \rightarrow \operatorname{St}^{\leq n }_{'\ul{k}^1,'\ul{k}^2}
\end{equation*}
inducing {{isomorphism}} on the classical loci. 

\begin{rem}\label{rem:st conv map factors through n}
It is clear that the morphism $\operatorname{St}^{\ula^1,\ula^2}_{\GL_n} \rightarrow \operatorname{St}^{\leq n}_{'\ul{k}^1,'\ul{k}^2}/\GL_{N+n}$ factors through $\operatorname{St}^{\ula^1,\ula^2}_{\GL_n} \rightarrow \operatorname{St}^{n}_{'\ul{k}^1,'\ul{k}^2}/\GL_{N+n}$.   
\end{rem}



\begin{lem}\label{lem_com_our_morphisms}
There exist natural smooth morphisms of quotient derived stacks:
\begin{align}\label{eq:morphisms diff quotients of tilde N}
\left. \widetilde{\mathcal{N}}^{\leq n}_{'\ul{k}} \right/ \on{GL}_{N + n} &\rightarrow  \left. \widetilde{\mathcal{N}}^{\leq n}_{\ul{k}} \right/ \on{GL}_{N}, & \left. \operatorname{St}^{n}_{'\ul{k}^1,'\ul{k}^2} \right/ \operatorname{GL}_{N+n} &\rightarrow \left. \operatorname{St}^{\leq n}_{\ul{k}^1,\ul{k}^2} \right/ \operatorname{GL}_{N}.
\end{align}
making the following diagrams commutative:
\begin{align} \label{eq:comp_lusztig_corresp}
&\xymatrix{\widetilde{\on{Gr}}^{\ul{\la}} \ar[dr] \ar[r] & \left. \widetilde{\mathcal{N}}^{\leq n}_{'\ul{k}} \right/ \on{GL}_{N+n} \ar[d] \\
& \left. \widetilde{\mathcal{N}}^{\leq n}_{\ul{k}} \right/ \on{GL}_{N} }, &
&\xymatrix{\operatorname{St}^{\underline{\la}^1,\underline{\la}^2}_{\operatorname{GL}_n}  \ar[dr] \ar[r] & \left. \operatorname{St}^{n}_{'\ul{k}^1,'\ul{k}^2} \right/ \on{GL}_{N+n} \ar[d] \\
& \left. \operatorname{St}^{\leq n}_{\ul{k}^1,\ul{k}^2} \right/ \on{GL}_{N}.}
\end{align}
\end{lem}
\begin{proof}
Fix a point $(x,F_{\bullet}) \in \widetilde{\mathcal{N}}^{\leq n}_{'\ul{k}}$. By definition, $x \in \mathfrak{gl}_{N+n}^{\leq n}$, hence $\operatorname{dim}\on{ker}x \leqslant n$. On the other hand, $x(\Bbbk^{N+n}) \subset F_1$ and $\on{dim}F_1=N$. It follows that $\operatorname{dim}(\on{ker}x)=n$. We can then define the morphism as follows. 
Set $F'_0:=F_0/\on{ker}(x)$. The map $x$ induces an isomorphism $F_0' \iso F_1$. For $i=1,\ldots, m$, let $F_i' \subset F_0'$ be the preimage of $F_{i+1} \subset F_1$ under this isomorphism. We obtain the flag $F_\bullet'$. Clearly, $x$ still acts on $F_0'=F_0/\on{ker}(x)$; we denote the corresponding endomorphism by $\bar{x}$.
Then
$
(x,F_\bullet) \mapsto (\bar{x},F'_{\bullet}).
$
So, we obtain the  required morphisms 
\begin{align*}
\widetilde{\mathcal{N}}^{\leq n}_{'\ul{k}^i}/\on{GL}_{N + n} &\rightarrow  \widetilde{\mathcal{N}}^{\leq n}_{\ul{k}^i}/ \on{GL}_{N}~(i=1,2), & \mathfrak{gl}^{n}_{N+n}/\GL_{N+n} &\rightarrow \mathfrak{gl}^{\leq n}_{N}/\GL_N.
\end{align*}
They induce the desired morphism between the Steinberg versions.

The commutativity of \eqref{eq:comp_lusztig_corresp} is clear. 

It remains to show that the morphisms \eqref{eq:morphisms diff quotients of tilde N} are smooth. Both are base changes of the morphism $\mathfrak{gl}^n_{N+n}/\on{GL}_{N+n} \rightarrow \mathfrak{gl}^{\leq n}_N/\on{GL}_N$ so it remains to check that the induced morphism $\mathfrak{gl}^n_{N+n} \rightarrow  \mathfrak{gl}_N^{\leq n}/\on{GL}_N$
is smooth. By smooth descent, it is enough to prove that the morphism:
\begin{equation*}
\mathfrak{gl}_{N+n}^n \times_{\mathfrak{gl}_N^{\leq n}/\on{GL}_N} \mathfrak{gl}_N^{\leq n} \rightarrow \mathfrak{gl}_N^{\leq n}
\end{equation*}
is smooth. The fiber product $\mathfrak{gl}_{N+n}^n \times_{\mathfrak{gl}^{\leq n}_N/\on{GL}_N} \mathfrak{gl}_N^{\leq n}$ parametrizes the triples $(x,y,\phi)$, where $x \in \mathfrak{gl}_{N+n}$ is a linear operator with $n$-dimensional kernel, $y \in \mathfrak{gl}_{N}$ and $\phi\colon \Bbbk^{N+n}/\operatorname{ker}(x) \iso \Bbbk^N$ is an isomorphism such that $\phi \circ \bar{x} = y \circ \phi$. 
Equivalently, it parametrizes the triples $(x,y,q)$, where $q\colon \Bbbk^{N+n} \twoheadrightarrow \Bbbk^N$ is surjective such that $\operatorname{ker}q=\operatorname{ker}x$ and $yq=qx$. 
Fixing the choices of $y$ and $q$ (which are independent of each other), we see that the space of possible $x$'s is parametrized by an open subset of $\operatorname{Hom}(\Bbbk^{N}, \operatorname{ker}q)$. The smoothness follows.
\end{proof}

\subsection{Slices in the affine Grassmannian and the Mirkovi\'c--Vybornov isomorphism}\label{sec:slices in aff Gr and MV iso}
Slices in the Beilinson--Drinfeld affine Grassmannian are certain locally closed subvarieties $\mathbb{X}_N^\mu \subset \mathbb{X}_N$ depending on a coweight $\mu$.
The morphism $\mathbb{Y}_N \rightarrow \mathbb{X}_N$ splits naturally when restricted to $\mathbb{X}_N^\mu$. Basic properties of $\mathbb{X}_N^\mu$ allow us to prove below that all of the morphisms in (\ref{eq:comp_lusztig_corresp}) are smooth on the classical loci. Along the way, we generalize and present a simple proof of \cite[Proposition~2.3]{an} (where only the case of two-row nilpotent elements was considered).  The 
Mirkovi\'c--Vybornov isomorphism  originally appeared in \cite{MV07, MVK22}. 
We follow \cite[Section~3.3]{CK18a} as this version is very transparent and the proofs are simple.

Let $e_1,\ldots,e_n$ be the standard basis of $\Bbbk^n$. For each $m \in \mathbb{Z}_{\geqslant 0}$ set:
\begin{equation*}
W^m := \operatorname{Span}(e_1,\ldots,z^{m-1}e_1,\ldots,e_n,\ldots,z^{m-1}e_n) \subset L_0.
\end{equation*}
Fix $\mu=(a_1,\ldots,a_n) \in \bZ_{\geqslant 0}^n$ such that $\sum_i a_i = N$. It determines the lattice $L_\mu \subset L_0$ generated by $z^{a_i}e_i$.
Define:
\begin{multline*}
{\mathbb{X}}_N^\mu :=\{L \in \mathbb{X}_N\,|\, [e_1],\ldots,[z^{a_1-1}e_1],\ldots,[e_n],\ldots,[z^{a_n-1}e_n]~\text{form a basis of}~L_0 / L, \\
\text{and}~z^{a_i}e_i \in W^{a_i}+L~\text{for all}~i\}.
\end{multline*}

Note that the $\GL_N$-torsor $\mathbb{Y}_N \rightarrow \mathbb{X}_N$ splits naturally when restricted to $\mathbb{X}_N^\mu$. Namely, for $L \in \mathbb{X}_N^\mu$ the identification $\psi\colon L_0/L \simeq \Bbbk^N$ is given by sending $[z^{a}e_{i}]$ ($0\leqslant a < a_i$) to the standard basis of $\Bbbk^N$. Let $i_\mu\colon \mathbb{X}_N^\mu \hookrightarrow \mathbb{Y}_N$ be the induced embedding.

The following lemma is clear: 
\begin{lem}\label{lem:iso_ind_X_mu_gr}
The closed embedding (\ref{eq:emb_X_N_shift}) induces a closed embedding $\mathbb{X}^\mu_N \hookrightarrow \mathbb{X}^{\omega_n + \mu}_{N+n}$. Combined with the identification (\ref{eq:iso_gr_lambda_shift_by_omega_n}),  it induces an isomorphism:
\begin{equation}\label{eq:iso_shift_base_change_mu}
\widetilde{\on{Gr}}^{\ul{\la}} \times_{\mathbb{X}_N} \mathbb{X}_N^\mu \iso  \widetilde{\on{Gr}}^{(\omega_n, \ul{\la})} \times_{\mathbb{X}_{N+n}} \mathbb{X}_{N+n}^{\mu+\omega_n}.
\end{equation}
\end{lem}

\begin{rem}
Note that both fiber products in (\ref{eq:iso_shift_base_change_mu}) coincide with their derived versions, because $\mathbb{X}^\mu_{N} \subset {\mathbb{X}}_N$ is an exact base change, see Remark~\ref{rem: our M is indeed a slice} and Proposition~\ref{prop:MV_iso_CK} below and \cite[1.3.1]{BM13}. 
\end{rem}

Following \cite[Section 3.3]{CK18a}, define (see also Example~3 {\it loc.cit.}):
\begin{multline*}
{\mathbb{M}}_N^\mu :=  \{A=(A_{ij})\,|\, A_{ij}~\text{is a matrix of size}~a_i \times a_j~\text{where}\\ 
A_{ii}~\text{has}~1\text{s just below the diagonal and all other non-zero entries in the last column,}\\
A_{ij},~\text{for}~i \neq j,~\text{has all non-zero entries in the last column but not below row}~a_j\}. 
\end{multline*}

\begin{rem}\label{rem: our M is indeed a slice}
It is easy to see that ${\mathbb{M}}^\mu_{N} \subset \mathfrak{gl}_N$ is an MV slice in the sense of \cite[Definition~5.2]{WWY20}; in particular, it is a normal slice in the sense of \cite[Lemma 1.3.1 iii)]{BM13}.
\end{rem}

\begin{prop}{\cite[Theorem 3.1]{CK18a}}\label{prop:MV_iso_CK}
The composition of $i_\mu\colon \mathbb{X}_N^\mu \hookrightarrow \mathbb{Y}_N$ with the morphism $\mathbb{Y}_N \rightarrow \mathfrak{gl}^{\leq n}_N$ induces an isomorphism $\mathbb{X}^\mu_N \iso \mathbb{M}^\mu_N$.    
\end{prop}

Note that the following corollary is stated purely in terms of Springer resolutions, but the argument is most transparent by using the connection to the affine Grassmannian. This is a generalization of \cite[Proposition 2.3]{an}.  
\begin{cor}\label{cor:iso_pull_st}
There exists an isomorphism of symplectic varieties
\begin{equation}\label{eq:iso springer after base change to slice}
\widetilde{\mathcal{N}}_{'\ul{k}} \times_{\mathfrak{gl}_{N+n}} \mathbb{M}_{N + n}^{\omega_n + \mu} \iso \widetilde{\mathcal{N}}_{\ul{k}} \times_{\mathfrak{gl}_N} \mathbb{M}^\mu_N 
\end{equation}
making the following diagram commutative:
\begin{equation}\label{eq:comm diagr base change to slice}
\xymatrix{\ar[d] \ar[r] \widetilde{\mathcal{N}}_{'\ul{k}} \times_{\mathfrak{gl}_{N+n}} \mathbb{M}_{N + n}^{\omega_n + \mu} & \widetilde{\mathcal{N}}_{\ul{k}} \times_{\mathfrak{gl}_N} \mathbb{M}^\mu_N  \ar[d] \\
\left. \widetilde{\mathcal{N}}^{\leq n}_{'\ul{k}} \right/ \on{GL}_{N+n} \ar[r] & \left. \widetilde{\mathcal{N}}^{\leq n}_{\ul{k}} \right/ {\on{GL}_N}}.
\end{equation} 
\end{cor}
\begin{proof}
The isomorphism (\ref{eq:iso springer after base change to slice}) follows from 
the fact that after the identification of Proposition~\ref{prop:MV_iso_CK},  (\ref{eq:iso springer after base change to slice}) becomes (\ref{eq:iso_shift_base_change_mu}) which is an isomorphism by Lemma~\ref{lem:iso_ind_X_mu_gr}. 
The commutativity of (\ref{eq:comm diagr base change to slice}) then follows from the commutativity of (\ref{eq:comp_lusztig_corresp}). 
\end{proof}

\begin{rem}
In Section~\ref{sec:tiltings_for_lambda_and_shift}, we will use that the isomorphism~\eqref{eq:iso springer after base change to slice} respects symplectic structures and can be quantized to an isomorphism of the corresponding $\cW$-algebras. 
This is clear from the fact that the Mirkovi\'c--Vybornov isomorphism (Proposition~\ref{prop:MV_iso_CK}) is symplectic and admits a quantization (see \cite[Theorems~A,~C]{WWY20}).  
\end{rem}

\subsection{Smoothness of the morphisms}
The goal of this section is to show that all morphisms in (\ref{eq:comp_lusztig_corresp}) are smooth (a part of this claim is contained in Lemma \ref{lem_com_our_morphisms}). For the case $\ula = (\om_1, \hdots, \om_1)$, this is \cite[Lemma~4.9]{FF21}. We use the MV-isomorphism, which simplifies the proof and allows us to generalize it to arbitrary $\ul{\la}$.

\vspace{0.1cm}

Write $N=pn+r$ with $0 \leqslant r<n$. Set 
\begin{equation*}
\om =(\underbrace{p+1,p+1,\ldots,p+1}_{r},\underbrace{p,\ldots,p}_{n-r}).
\end{equation*}

\begin{lem}\label{lem_act_smooth}
The action morphism
$
\operatorname{GL}_{N} \times {\mathbb{M}}_N^\om \rightarrow
\mathfrak{gl}^{\leq n}_N
$
is smooth, its image is open and contains $\mathfrak{gl}_N^{\leq n} \cap \mathcal{N}$. 
\end{lem}
\begin{proof}
Smoothness is a formal corollary of the fact that $\mathbb{M}^\om_N$ is a normal slice in $\mathfrak{gl}_N$
 (see \cite[Section 3.2]{MVK22}, \cite[Section 5.1]{WWY20}). The fact that the image is open follows from \cite[Lemma 3.2(1)]{MVK22}. The last property follows from \cite[Lemma 3.2(2)]{MVK22}.
\end{proof}

\begin{cor}\label{cor:smooth_morph}
All morphisms in \eqref{eq:comp_lusztig_corresp} are smooth.
\end{cor}
\begin{proof}
We show that the morphism $\operatorname{St}^{\underline{\la}^1,\underline{\la}^2}_{\operatorname{GL}_n} \rightarrow  \operatorname{St}^{\leq n}_{\ul{k}^1,\ul{k}^2}/\on{GL}_{N}$ is smooth, and the argument for $\widetilde{\on{Gr}}^{\ul{\la}} \rightarrow \widetilde{\mathcal{N}}^{\leq n}_{\ul{k}}/\on{GL}_N$ is analogous.   
We show that the morphism $\operatorname{St}^{\underline{\la}^1,\underline{\la}^2}_{\operatorname{GL}_n} \times_{\mathbb{X}_N} \mathbb{Y}_N \rightarrow \operatorname{St}^{\leq n}_{\ul{k}^1,\ul{k}^2}$ is smooth. 
It follows from Lemma~\ref{lem_act_smooth} (combined with the smooth descent) that it is enough to show that the morphism
\begin{equation*}
(\operatorname{St}^{\underline{\la}^1,\underline{\la}^2}_{\operatorname{GL}_n} \times_{\mathbb{X}_N} \mathbb{Y}_N) \times^\bL_{\mathfrak{gl}^{\leq n}_N} (\on{GL}_N \times \mathbb{M}^\om_N) \rightarrow \operatorname{St}^{\leq n}_{\ul{k}^1,\ul{k}^2} \times^\bL_{\mathfrak{gl}^{\leq n}_N} (\on{GL}_N \times \mathbb{M}^\om_N)
\end{equation*}
is smooth. 
We claim that it 
is actually an isomorphism. 
To see this, recall that by \eqref{eq:cartesian_diagr_Y_to_X} we have $\operatorname{St}^{\underline{\la}^1,\underline{\la}^2}_{\operatorname{GL}_n} \times_{\mathbb{X}_N} \mathbb{Y}_N \iso \operatorname{St}^{\leq n}_{\ul{k}^1,\ul{k}^2} \times^\bL_{\mathfrak{gl}^{\leq n}_N} \mathbb{Y}_N$. 
So, it remains to show that 
\[\mathbb{Y}_N \times^\bL_{\mathfrak{gl}^{\leq n}_N}  (\on{GL}_N \times \mathbb{M}^\om_N) \rightarrow \on{GL}_N \times \mathbb{M}^\om_N\] 
is an isomorphism. Indeed, the inverse is induced by the natural map $\operatorname{GL}_N \times \mathbb{M}^\om_N \hookrightarrow \mathbb{Y}_N$ that comes from  the identification $\mathbb{M}^\om_N \simeq \mathbb{X}^\om_N$ and the section $i_\om\colon \mathbb{X}^\om_N \hookrightarrow \mathbb{Y}_N$. 
\end{proof}

\subsection{Induced maps on K-theories} \label{subsec:iso_K_thery}
We finish this section by proving a result for which our modification of the Lusztig correspondences is essential. It would not be true for the original versions of Section~\ref{subsec: Lusztig's correspondences}. Since we are only dealing with K-theory in this section, we make no distinction between a derived scheme and its classical locus.

\vspace{0.2cm}

Corollary~\ref{cor:smooth_morph} ensures that pullbacks for the morphisms in \eqref{eq:comp_lusztig_corresp} are well-defined in equivariant K-theory.

\begin{prop}\label{prop:iso_on_K_theory}
The morphisms 
\begin{align*}
\widetilde{\on{Gr}}^{(\omega_n,\ul{\la})} &\rightarrow \widetilde{\mathcal{N}}^{\leq n}_{'\ul{k}}/\on{GL}_{N+n}, & \operatorname{St}^{(\omega_n,\underline{\la}^1),(\omega_n,\underline{\la}^2)}_{\operatorname{GL}_n}  &\rightarrow \operatorname{St}^{\leq n}_{'\ul{k}^1,'\ul{k}^2}/\on{GL}_{N+n}
\end{align*}
induce isomorphisms:
\begin{equation}\label{iso_springer_K_theory}
K^{\operatorname{GL}_{N+n} \times \bG_m}(\widetilde{\mathcal{N}}^{\leq n}_{'\ul{k}}) \iso   K^{\operatorname{GL}_n(\cO) \rtimes \bG_m}(\widetilde{\on{Gr}}^{(\omega_n,\ula)}),
\end{equation}
\begin{equation}\label{iso_steinberg_K_theory}
K^{\operatorname{GL}_{N+n} \times \bG_m}(\operatorname{St}^{\leq n}_{'\ul{k}^1,'\ul{k}^2}) \iso K^{\operatorname{GL}_n(\cO) \rtimes \bG_m}(\operatorname{St}^{(\omega_n,\underline{\la}^1),(\omega_n,\underline{\la}^2)}).
\end{equation}
\end{prop}
\begin{proof}
We prove \eqref{iso_springer_K_theory}. The proof of \eqref{iso_steinberg_K_theory} is the same.

Consider the stratification of $\widetilde{\mathcal{N}}^{\leq n}_{'\ul{k}}$ by preimages of nilpotent orbits and the stratification of $\widetilde{\on{Gr}}^{(\omega_n,\ula)}$ by the preimages of $\on{GL}_n(\mathcal{O})$-orbits. The morphism (\ref{iso_springer_K_theory}) is compatible with filtrations induced by these stratifications, so it is sufficient to check that it becomes an isomorphism after passing to the associated graded. Recall also that the morphism (\ref{iso_springer_K_theory}) is given by the pullback:
\begin{equation*}
K^{\operatorname{GL}_{N+n} \times \bG_m}(\widetilde{\mathcal{N}}^{\leq n}_{'\ul{k}}) \rightarrow K^{(\operatorname{GL}_n(\cO) \rtimes \bG_m) \times \on{GL}_{N+n}}(\widetilde{\on{Gr}}^{(\omega_n,\ula)} \times_{\mathbb{X}_{N+n}} \mathbb{Y}_{N+n}).
\end{equation*}

Pick $\mu=(a_1,\ldots,a_n) \in \mathbb{Z}_{\geqslant 0}^{n}$ such that $\sum_{i} a_i=N+n$. It defines the  point $z^\mu \in \mathbb{X}_{N+n}$. Recall that $i_\mu(z^{\mu}) \in \mathbb{Y}_{N+n}$ is the natural lift of $z^\mu$ (see Section \ref{sec:slices in aff Gr and MV iso} above for the definition of $i_\mu$).
Let $e_{\mu} \in \mathfrak{gl}^{\leq n}_{N + n}$ be the image of $i_\mu(z^{\mu})$ under the map (\ref{eq:morphism Y to gl}). Note that the fibers of $\widetilde{\mathcal{N}}^{\leq n}_{'\ul{k}} \rightarrow \mathfrak{gl}^{\leq n}_{N+n}$ and $\widetilde{\on{Gr}}^{(\omega_n,\ula)} \times_{\mathbb{X}_{N+n}} \mathbb{Y}_{N+n} \rightarrow \mathbb{Y}_{N+n}$ over $e_{\mu}$ and $i_\mu(z^\mu)$ respectively are canonically isomorphic, see~\eqref{eq:cartesian_diagr_Y_to_X}.

So, it remains to prove that the natural morphism 
\begin{equation*}
\on{Stab}_{(\operatorname{GL}_n(\cO) \rtimes \bG_m) \times \on{GL}_{N+n}}(i_\mu(z^{\mu}))^{\mathrm{red}} \rightarrow \on{Stab}_{\on{GL}_{N+n} \times \bG_m}(e_{\mu})^{\mathrm{red}}
\end{equation*}
induced by the map $\mathbb{Y}_{N+n} \rightarrow \mathfrak{gl}^{\leq n}_{N+n}$ is an isomorphism, where by $\bullet^{\mathrm{red}}$ we mean the reductive part of an algebraic group. 

The point $i_\mu(z^{\mu})$ consists of a lattice $L_\mu \subset L_0$ together with a fixed choice of the basis in $L_0/L_\mu$. We may assume that the fiber of $\widetilde{\operatorname{Gr}}^{(\omega_n,\underline{\lambda})} \rightarrow \mathbb{X}_{N+n}$ over $z^\mu$ is nonempty, so $L_\mu \subset zL_0$, i.e., $a_i \geqslant 1$ for any $i$.

The reductive part of the stabilizer of $i_\mu(z^\mu)$ is isomorphic to the product $\prod_{i}\on{GL}_{r_i} \times \bG_m$, where $r_i$ is the number of times $i$ appears in $\mu$.


The nilpotent  $e_\mu$ corresponds to the action of $z$ on $L_0/L$ so it has $n$ Jordan blocks of sizes $a_1, a_2, \ldots, a_n$. Using that all $a_i$ are positive we see that the reductive part of the centralizer is again $\prod_{i} \on{GL}_{r_i} \times \bG_m$.
\end{proof}

\begin{cor} \label{cor: Psi define iso on K-theory}
The morphisms 
\begin{align*}
\widetilde{\on{Gr}}^{\ul{\la}} &\rightarrow \widetilde{\mathcal{N}}^{\leq n}_{'\ul{k}}/\on{GL}_{N+n}, &\operatorname{St}^{\underline{\la}^1,\underline{\la}^2}_{\operatorname{GL}_n}  &\rightarrow \operatorname{St}^{\leq n}_{'\ul{k}^1,'\ul{k}^2}/\on{GL}_{N+n}
\end{align*}
induce isomorphisms:
\begin{align*}
K^{\operatorname{GL}_{N+n} \times \bG_m}(\widetilde{\mathcal{N}}^{\leq n}_{'\ul{k}}) &\iso   K^{\operatorname{GL}_n(\cO) \rtimes \bG_m}(\widetilde{\on{Gr}}^{\ula}), & K^{\operatorname{GL}_{N+n} \times \bG_m}(\operatorname{St}^{\leq n}_{'\ul{k}^1,'\ul{k}^2}) &\iso K^{\operatorname{GL}_n(\cO) \rtimes \bG_m}(\operatorname{St}^{\underline{\la}^1,\underline{\la}^2}_{\operatorname{GL}_n}).
\end{align*}
\end{cor}
\begin{proof}
Follows from Proposition \ref{prop:iso_on_K_theory} and Lemma \ref{lem:iso_shifts_gr_st}.
\end{proof}

Let us finish this section by mentioning the geometric interpretation of the homomorphism defined in \cite[Section 10]{CK18b}. Let $\hat{\mathcal{H}}_N$ be the affine Hecke algebra for $\operatorname{GL}_N$ over the base ring $K^{\operatorname{GL}_n}(\operatorname{pt})$. In \cite{CK18b}, the authors defined a homomorphism of algebras
\begin{equation}\label{eq:homom sectio 10 CK}
\hat{\mathcal{H}}_N \rightarrow K^{\operatorname{GL}_n(\mathcal{O}) \rtimes \mathbb{G}_m}(\operatorname{St}^{\underline{\lambda},\underline{\lambda}}_{\operatorname{GL}_n}), \qquad \underline{\lambda}=(\underbrace{\omega_1,\ldots,\omega_1}_{N}) 
\end{equation}
and made three conjectures about it (two of them were proved in \cite{LTV}, namely, \cite[Conjectures 10.2 and 10.3]{CK18b}).

From the perspective of this paper, homomorphism (\ref{eq:homom sectio 10 CK}) should be  the pullback under the $(\operatorname{GL}_n(\mathcal{O}) \rtimes \mathbb{G}_m) \times \operatorname{GL}_N $-equivariant map:
\begin{equation}
\operatorname{St}^{\underline{\lambda},\underline{\lambda}}_{\operatorname{GL}_n} \times_{\mathbb{X}_N} {\mathbb{Y}}_N \rightarrow \operatorname{St}^{\leq n}_{\underline{k},\underline{k}} \hookrightarrow \operatorname{St}_{\underline{k},\underline{k}} 
\end{equation}
as above (note that the action of $\operatorname{GL}_n(\mathcal{O})$ on $\operatorname{St}_{\underline{k},\underline{k}}$ is trivial). Here $\underline{k}=(\underbrace{1,\ldots,1}_{N})$. A simple modification of the proof of Proposition \ref{prop:iso_on_K_theory} above implies that this pullback
 is surjective (as was conjectured in \cite[Section 10.1]{CK18b}). It would be interesting to see if this geometric realization helps  prove \cite[Conjecture 10.4]{CK18b} describing the kernel of (\ref{eq:homom sectio 10 CK}). 

\section{Noncommutative resolution of affine Schubert varieties in type A}\label{sec:noncomm_resol_conv_diagr}


\subsection{General setting and conjecture} \label{subsection: conjecture in any type}
Let $G$ be a connected reductive algebraic group, and let $\Gr = G(\cK) / G(\cO)$ be the affine Grassmannian of $G$. Let $\la^\svee_1, \hdots, \la^\svee_k$ be minuscule coweights, and let $\la^\svee = \la^\svee_1 + \hdots + \la^\svee_k$. Let $p\colon \Gr^{\la_1^\svee} \conv \hdots \conv \Gr^{\la_k^\svee} = \wti \Gr^{\underline{\la}^\svee} \rightarrow \ol \Gr^{\la^\svee}$ be the resolution of the affine Schubert variety by the convolution variety.

Let $\Gr^\om$ be the closed $G(\cO)$-orbit in $\ol \Gr^{\lambda^\svee}$, so $\om$ is a minuscule coweight (possibly 0). 
Consider the transversal slice $ \cW_\om^\lach$ in $\ol \Gr^\lach$ to $\Gr^\om$ at the point $t^\om$ (see \cite{KWWY14} for the definition). Let $\wti \cW_\om^\lach$ be its pre-image under $p$. Denote $p'\colon \wti \cW_\om^\lach \rightarrow  \cW_\om^\lach$ --- it is a conical symplectic resolution.

Recall the notion of \textit{relative tilting generator} from \cite[1.4.2]{BM13}. 
The morphism $p'$ admits an important tilting generator (the word ``relative'' may be omitted here, since $ \cW_\om^\lach$ is affine), related to quantizations in positive characteristics --- see a construction in \cite{Kal08} for general symplectic resolutions or in \cite{Web19} for Coulomb branches, which $ \cW^\lach_\om$ is in types ADE, due to \cite{BFN16}. Note that such a tilting generator is not unique, and below we mean {\it some} of these tiltings.

We propose the following
\begin{conj} \label{conjecture: tilting is restriction}
There is a $G(\cO) \rtimes \bG_m$-equivariant sheaf $\cE$ on $\wti \Gr^{\underline{\la}^\svee}$, whose restriction $\cE'$ to $\wti \cW_\om^\lach$ is the above mentioned tilting generator for $p'$.
\end{conj}

The main application of this conjecture is the following
\begin{cor} \label{if restriction of sheaf is tilting then it is tilting}
Suppose Conjecture \ref{conjecture: tilting is restriction} holds. Then $\cE$ is a relative tilting generator for $p$.
\end{cor}
\begin{proof}
For simplicity, we first consider the case $\om = 0$. The property of being relative tilting generator is local over the base, so it is sufficient to show it for an open cover. By assumption, it holds for $ \cW^\lach_0 \subset \ol \Gr^\lach$, and hence by equivariance, it holds for $g  \cW^\lach_0 \subset \ol \Gr^\lach$ for any $g \in G(\cO)$. It remains to note that these open subschemes form a cover, since any $G(\cO)$-orbit in $\ol \Gr^\lach$ intersects $\cW^\lach_0$ nontrivially.

Now consider the case of arbitrary $\om$. Let $P_\om = \Stab_G(t^\om) \subset G$ be a parabolic subgroup. Consider $G. \cW^\lach_\om = G \times^{P_\om} \cW^\lach_\om$ --- an open subscheme of $\ov \Gr^\lach$, as well as $G. \wti \cW^\lach_\om = G \times^{P_\om} \wti \cW^\lach_\om$ --- an open subscheme of $ \wti \Gr^{\underline{\la}^\svee}$. Since $\cE'$ is a relative tilting generator for $p': \wti \cW^\lach_\om \rightarrow \cW^\lach_\om$, it follows that the restriction of $\cE$ to $G. \wti \cW^\lach_\om$ is a relative tilting for $G \times^{P_\om} \wti \cW^\lach_\om \rightarrow G \times^{P_\om} \cW^\lach_\om$. Now the result follows by the same argument as for the $\om = 0$ case, using the open cover $\{ g G. \cW^\lach_\om | g \in G(\cO) \}$.
\end{proof}

Thus, Conjecture \ref{conjecture: tilting is restriction} can be used to define 
perversely-exotic t-structure on the category 
$D^b(\Coh^{G(\cO) \rtimes \bG_m} \wti \Gr^{\underline{\la}^\svee})$, similarly to \cite{BM13} for the Springer resolution.

\begin{rem}
Note that the property of a coherent sheaf on $\wti \cW^{\ula^\svee}_\om$ being the restriction of an equivariant sheaf on $\wti \Gr^{\ula^\svee}$ is equivalent to being equivariant under an action of a certain Lie groupoid, as in the proof of~\cite[Proposition~5.4]{Dum25}.
\end{rem}

Below we prove this conjecture for $G = \GL_n$, using the results of Section~\ref{sec: modified Lusztig's correspondence}, relating the geometry of affine Grassmannian to the geometry of nilpotent cone. 
This approach also establishes compatibility of t-structures and corresponding canonical bases under this relation; we elaborate it in Subsections \ref{subsec: titlings and simples},~\ref{subsec: description of basis for schuberts}.

\subsection{Proof in type A by relating to nilpotent orbits} \label{subsection: construction of tilting in type A}

Let $\Gr = \Gr_{\GL_n}$. It is acted by the group $\GL_n(\cO) \rtimes \bG_m$. 

Recall the notations of Section~\ref{sec: modified Lusztig's correspondence}. Let $\ula = (\la_1, \hdots, \la_m)$ be a tuple of minuscule (i.e., fundamental) coweights of $\GL_n$ and $\la = \la_1 + \hdots + \la_m$. Let $1 \leq k_i \leq n$ be such that $\la_i = \om_{k_i}$.
Let $p: \wti \Gr^{\ula} \rightarrow \ol \Gr^\la$ be the resolution of singularities.

As in Section~\ref{subsec: Lusztig's correspondences}, we consider the partial Springer resolution $s: \wti{\cN}_{\ukey} \rightarrow \cN_{\ukey}$ for the group $\GL_{N}$, $N = \sum k_i$.

In the notation of Section~\ref{sec: modified Lusztig's correspondence}, denote $\wti \cM_{\ukey}^{\ula} = \widetilde{\operatorname{Gr}}^{\ul{\la}} \times_{\mathbb{X}_{N}} \mathbb{Y}_{N}$, and let $\cM_{\ukey}^{\la} = \ol \Gr^\la  \times_{\bX_N} \bY_N$.
These spaces are endowed with a $\GL_{N} \times \GL_n(\cO)\rtimes \bG_m$-action, and we have the following diagram, see~\eqref{eq:lus_corresp} (compare with \cite[(4.3)]{FF21})
\begin{equation} \label{convolution space M}
\begin{tikzcd}
	{ \widetilde{\mathcal N}_{\ukey}} & {\widetilde{\mathcal M}^{\ula}_{\ukey}} & { \widetilde{\Gr}^{ \underline \lambda}} \\
	{ {\mathcal N}_{\ukey}} & {{\mathcal M}^{\la}_{\ukey}} & {\overline{\Gr}^{\lambda}}
	\arrow["s", from=1-1, to=2-1]
	\arrow["{\widetilde \psi_2}"', from=1-2, to=1-1]
	\arrow["{\widetilde \psi_1}", from=1-2, to=1-3]
	\arrow["m", from=1-2, to=2-2]
	\arrow["p", from=1-3, to=2-3]
	\arrow["{\psi_2}"', from=2-2, to=2-1]
	\arrow["{\psi_1}", from=2-2, to=2-3]
\end{tikzcd}
\end{equation}
The morphisms $\psi_2, \wti \psi_2$ are flat due to Corollary~\ref{cor:smooth_morph}.
The morphisms $\psi_1, \wti \psi_1$ are principal $\GL_{N}$-bundles, hence $\psi_1^*, \wti \psi_1^*$ induce equivalences of triangulated categories
\begin{align*}
\psi_1^*\colon D^b (\Coh^{\GL_n(\cO) \rtimes \bG_m} (\ol \Gr^{\la})) &\xrightarrow{\sim}  D^b (\Coh^{\GL_{N} \times \GL_n(\cO) \rtimes \bG_m} (\cM^{\la}_{\ukey})); \\
\wti \psi_1^*\colon D^b (\Coh^{\GL_n(\cO) \rtimes \bG_m} (\wti \Gr^{\ula})) &\xrightarrow{\sim}  D^b (\Coh^{\GL_{N} \times \GL_n(\cO) \rtimes \bG_m} (\wti \cM^{\la}_{\ukey})).
\end{align*}
The inverse equivalences $(\psi_1^*)^{-1}$ and $(\wti \psi_1^*)^{-1}$ are equivariant descents along torsors.

We define the functors 
\begin{align*}
{\Psi} = (\psi_1^*)^{-1} \circ \psi_2^* :  D^b (\Coh^{\GL_{N} \times \bG_m} (\cN_\ukey)) &\rightarrow D^b (\Coh^{\GL_n(\cO) \rtimes \bG_m} (\ol \Gr^\la)) ; \\
\wti \Psi = (\wti \psi_1^*)^{-1} \circ \wti \psi_2^*:  D^b (\Coh^{\GL_{N } \times \bG_m} ( \wti \cN_\ukey)) &\rightarrow D^b (\Coh^{\GL_n(\cO) \rtimes \bG_m} (\wti \Gr^\ula)).
\end{align*}
Equivalently, $\Psi$ is the pullback under the flat morphism of quotient stacks 
\begin{equation*}
[{\GL_n(\cO) \rtimes \bG_m} \backslash \ol \Gr^{\la}] \rightarrow [{\GL_{N }} \times \bG_m \backslash \cN_{\ukey}]
\end{equation*}
and similarly for $\wti \Psi$.

Let $\mu$ be a dominant coweight of $\GL_n$ such that $\la \geq \mu$ w.r.t. the dominance order. We consider $\cW^\la_\mu$ --- the transversal slice to an orbit in the affine Grassmannian (in notations of Section~\ref{sec: modified Lusztig's correspondence}, it is the intersection of $\bX^\mu_N$ with $\ol \Gr^\la$), and $\wti \cW^\ula_\mu = p^{-1}(\cW^\la_\mu)$ --- its resolution. Recall that we denote  $p'\colon \wti \cW^\ula_\mu \rightarrow \cW^\ula_\mu$.

If $N \om_1 - \mu = \sum_i m_i \al_i$, we define a partition of $N$:
\[
\pi = (\pi_1, \hdots, \pi_n),
\]
where
\[
\pi_1 = m_{n - 1},~ \pi_2 = m_{n - 2} - m_{n - 1},~ \hdots,~ \pi_{n - 1} = m_{1} - m_{2},~ \pi_n = N - m_{1}.
\]

Mirkovi\'c--Vybornov \cite{MVK22}, in particular, constructed a slice to the nilpotent orbit, corresponding to the partition $\pi$, which we denote $\cT_\pi$. We denote by $\cT_\pi^{\ukey}$ its intersection with $\cN_{\ukey}$ (in notations of Section~\ref{sec: modified Lusztig's correspondence}, this is the intersection of $\mathbb M^\mu_N$ with $\cN_\ukey$). Summarizing Section~\ref{sec: modified Lusztig's correspondence}, we have the following commutative diagram: 
\begin{equation} \label{Mirkovic-Vybornov}
\begin{tikzcd}
	&& {\widetilde{\mathcal M}^\lambda_\ukey} \\
	{\widetilde{\mathcal N}_{\ukey}} & {\widetilde {\mathcal T}^{\ukey}_\pi} && {\widetilde{\mathcal{W}}^{\underline \lambda}_\mu} & {\widetilde{\mathrm{Gr}}^{\underline \lambda}} \\
	&& {{\mathcal M}^\lambda_{\ukey}} \\
	{{\mathcal N}_{\ukey}} & {{\mathcal T}^{\ukey}_\pi} && {{\mathcal{W}}^{\lambda}_\mu} & {\overline{\mathrm{Gr}}^{\lambda}}
	\arrow["{\widetilde \psi_2}"', from=1-3, to=2-1]
	\arrow["{\widetilde \psi_1}", from=1-3, to=2-5]
	\arrow["m"{pos=0.7}, from=1-3, to=3-3]
	\arrow["s"', from=2-1, to=4-1]
	\arrow["{\widetilde j_\pi}", hook', from=2-2, to=2-1]
	\arrow["{s'}"', from=2-2, to=4-2]
	\arrow["{\widetilde \theta}", from=2-4, to=1-3]
	\arrow[equals, from=2-4, to=2-2]
	\arrow["{\widetilde i_\mu}"', hook, from=2-4, to=2-5]
	\arrow["{p'}", from=2-4, to=4-4]
	\arrow["p", from=2-5, to=4-5]
	\arrow["{\psi_2}"'{pos=0.7}, from=3-3, to=4-1]
	\arrow["{\psi_1}"{pos=0.7}, from=3-3, to=4-5]
	\arrow["{j_\pi}", hook', from=4-2, to=4-1]
	\arrow[equals, from=4-2, to=4-4]
	\arrow["\theta", from=4-4, to=3-3]
	\arrow["{i_\mu}"', hook, from=4-4, to=4-5]
\end{tikzcd}
\end{equation}


Recall the Bezrukavnikov--Mirkovi\'c tilting generators, discussed in Section~\ref{sec: coherent sheaves on parabolic steinbergs}. 
The restriction of this tilting generator along any exact base change (see \cite[1.3]{BM13} for the definition) is still a tilting generator. It is shown in \cite[Lemma 1.3.1 iii)]{BM13} that $\wti \cT^{\ukey}_\pi$, being a normal slice to an orbit, has the property of exact base change. 

We denote the Bezrukavnikov--Mirkovi\'c (BM) tilting on $\wti \cN_{\ukey}$ by $\cE_{{\ukey}}$. Its restriction to $\wti \cT^{\ukey}_\pi = \wti \cW^\ula_\mu$ is also a tilting. 
Denote also $\cA_\ukey = \End(\cE_\ukey)$ --- the noncommutative (parabolic) Springer resolution (see Section~\ref{sec: coherent sheaves on parabolic steinbergs}), which we view as a sheaf of algebras on $\cN_\ukey$.

The following proposition proves Conjecture \ref{conjecture: tilting is restriction} for $\GL_n$.

\begin{prop} \label{psi of tilting restrictied to slice is tilting}
The restriction of the sheaf $\cE_{\ula} := \wti \Psi(\cE_{{\ukey}})$ from $\wti \Gr^\ula$ to $\wti \cW^\ula_\mu$ coincides with a BM tilting under the identification $\wti \cW^\ula_\mu = \wti \cT^{\ukey}_\pi$.
\end{prop}

\begin{proof}
We show that $\wti i_\mu^*  \wti \Psi (\cE_{{\ukey}}) \simeq \wti j_\pi^* \cE_{{\ukey}}$. Unrevealing the definition of $\wti \Psi$, it is sufficient to verify that
\[
\wti i_\mu^* \circ (\wti \psi_1^*)^{-1} \circ  \wti \psi_2^*  \simeq \wti j_\pi^*.
\]
Recall that $\wti \theta$ is a section of a $\GL_N$-principal bundle $\wti \psi_1$, hence it trivializes this bundle over $\wti \cW^\ula_\mu$. Thus, $\wti i_\mu^* \circ (\wti \psi_1^*)^{-1}$ is a descent along the trivial torsor, and hence one has $\wti i_\mu^* \circ (\wti \psi_1^*)^{-1} = \wti \theta^*$. Thus, the required identity becomes
\[
\wti \theta^* \circ  \wti \psi_2^*  \simeq \wti j_\pi^*,
\]
which is evident from the commutativity of the diagram \eqref{Mirkovic-Vybornov}.
\end{proof}

Thus:
\begin{thm} \label{thm: noncommutative resolution of affine Schubert variety}
There is a $\GL_n(\cO) \rtimes \bG_m$-equivariant vector bundle $\cE_{\ula}$ on ${\wti \Gr^\ula}$, such that the functor $\cF \mapsto \bR p_* \bR\shhom(\cE_{\ula}, \cF)$ defines the derived equivalences
\begin{align*}
D^b(\Coh(\wti \Gr^\ula)) &\simeq D^b(\cA_\ula\mathrm{-mod});\\
D^b(\Coh^{\GL_n(\cO) \rtimes \bG_m} (\wti \Gr^\ula)) &\simeq D^b(\cA_{\ula}\mathrm{-mod}^{\GL_n(\cO) \rtimes \bG_m}),
\end{align*}
where $\cA_\ula = p_* \shhom(\cE_{\ula}, \cE_{\ula})$ is a coherent sheaf of algebras on $\ol \Gr^\la$.

Moreover, for any $\mu$, the restriction of $\cE_{\ula}$ to $\wti \cW^\ula_\mu$ defines a tilting generator for $\wti \cW^\ula_\mu$, which coincides with the Bezrukavnikov---Mirkovi\'c tilting generator under the Mirkovi\'c---Vybornov isomorphism.
\end{thm}
Here $\cA_{\ula}\mathrm{-mod}$ and $\cA_{\ula}\mathrm{-mod}^{\GL_n(\cO) \rtimes \bG_m}$ are the categories of coherent $\cA_\ula$-modules and equivariant coherent $\cA_{\ula}$-modules respectively.
\begin{proof}
We let $\cE_{\ula} = \wti \Psi( \cE_{\ukey})$. The second claim of the theorem follows from Proposition \ref{psi of tilting restrictied to slice is tilting}. The first claim then follows from Proposition \ref{if restriction of sheaf is tilting then it is tilting}.
\end{proof}


\begin{definition}
We call the sheaf of algebras $\cA_{\ula} = p_* \shhom(\cE_{\ula}, \cE_{\ula})$ on $\ol \Gr^\la$ the noncommutative resolution of the affine Schubert variety $\ol \Gr^\la$.
\end{definition}

Note that $\cA_\ula$ does depend on $\ula$ (not just on $\la$).

\subsection{Steinberg-type varieties} \label{subsec: tiltings for steinbergs}

Consider dominant coweights $\la^1, \la^2$. We say that $\la^1$ and $\la^2$ are {\it compatible} if $\la^1 = \sum_i \om_{k^1_i}$, $\la^2 = \sum_i \om_{k^2_i}$, and $\sum_i k^1_i = \sum_i k^2_i = N$. It is equivalent to the fact that under the Lusztig correspondence, the associated Springer resolutions are for the algebraic group $\GL_N$ (for the same~$N$).

In what follows, we study varieties $\wti \Gr^{\ula^1} \times_\Gr \wti \Gr^{\ula^2}$ with compatible $\la^1, \la^2$. 
More precisely, we take the derived product over the space $\bX_N$ (see Section~\ref{sec: modified Lusztig's correspondence}), which, as noted above, is the Beilinson--Drinfeld deformation of the Schubert variety $\ol \Gr^{N\om_1}$. The fact that we take the product over $\bX_N$ corresponds, under the Lusztig correspondence, to the fact that on the Springer side, we consider the derived product over $\g$ (and not over $\cN$). We abbreviate $\bX = \bX_N$.

Consider $\wti \Gr^{\ula^1}$ and $\wti \Gr^{\ula^2}$ for compatible $\la^1, \la^2$. Let $\cE_{\ula^1}$ and $\cE_{\ula^2}$ be the tilting generators on them, constructed above. Then the dual bundle $\cE_{\ula^2}^*$ is also a tilting generator of $\wti \Gr^{\ula^2}$, since this holds for the BM-tilting for $\wti \cT_\pi^\ukey$, see \cite{BM13} and Section~\ref{sec: coherent sheaves on parabolic steinbergs}. 
As in the Springer case in Section~\ref{sec: coherent sheaves on parabolic steinbergs}, we consider the sheaf $\cE_{\ula^1} \boxtimes \cE^*_{\ula^2}$ on $\wti \Gr^{\ula^1} \times^\bL_{\Gr} \wti \Gr^{\ula^2}$, which is a generator. 
We have the sheaf of dg-algebras 
$\cA_{\ula^1} \otimes^\bR_{\cO_\bX} \cA^{\op}_{\ula^2}$ 
on $\bX$, set-theoretically supported on $\ol \Gr^{N\om_1}$. 
Completely parallel to Corollary~\ref{cor: sheaves on steinberg equivalent to bimodules}, we derive from  Theorem~\ref{thm: noncommutative resolution of affine Schubert variety} that there is the derived equivalence
\begin{equation} \label{eq: derived equivalence for steinbergs}
D^b \Coh^{\GL_n(\cO) \rtimes \bG_m} ( \wti \Gr^{\ula^1} \times^\bL_{\bX} \wti \Gr^{\ula^2}) \simeq D^b(\cA_{\ula^1} \otimes^\bR_{\cO_\bX} \cA_{\ula^2}^\op\mods^{\GL_n(\cO) \rtimes \bG_m}). 
\end{equation}


\subsection{t-structures and simple objects} \label{subsec: titlings and simples}

Recall the definition of perversely-exotic t-structure on (partial) Springer resolutions, see Section~\ref{sec: coherent sheaves on parabolic steinbergs}.


We make a similar definition for the case of the resolution $\wti \Gr^\ula \rightarrow \ol \Gr^\la$.

\begin{definition}
perversely-exotic t-structure on the category $D^b(\Coh^{\GL_n(\cO) \rtimes \bG_m} (\wti \Gr^\ula))$ is defined as the image of the t-structure of perverse coherent $\cA_\ula$-modules under the equivalence $D^b(\Coh^{\GL_n(\cO) \rtimes \bG_m} (\wti \Gr^\ula)) \simeq D^b(\cA_\ula\mathrm{-mod}^{\GL_n(\cO) \rtimes \bG_m})$.
We denote its heart by $\cP_\exo^{\GL_n(\cO) \rtimes \bG_m}(\wti \Gr^\ula)$.

perversely-exotic t-structure on the category $D^b(\Coh^{\GL_n(\cO) \rtimes \bG_m} (\wti \Gr^{\ula^1} \times^\bL_{\bX} \wti \Gr^{\ula^2} ))$ is defined as the image of the t-structure of perverse coherent $\cA_{\ula^1} \otimes^\bR_{\cO_\bX} \cA_{\ula^2}^\op$-modules under the equivalence~\eqref{eq: derived equivalence for steinbergs}. 
We denote its heart by $\cP_\exo^{\GL_n(\cO) \rtimes \bG_m}(\wti \Gr^{\ula^1} \times^\bL_{\bX} \wti \Gr^{\ula^2})$.
\end{definition}

In the second part of this definition, we used Theorem~\ref{thm: perverse coherent t-structure for dg-algebras} about perverse coherent modules over dg-algebras.
Note also that the notion of coherent $\IC$-extension 
is well-defined in the case when there are finitely many group orbits, and the dimensions of adjacent orbits differ at least by 2 (\cite{Bez00, AB10}). These conditions are well known to hold for $\ol \Gr^\la$. Hence 
simple perversely-exotic sheaves on $\wti \Gr^\ula$ are parametrized by pairs $(\mu, V)$, where $\mu \leq \la$ dominant, and $V$ is an irreducible $(G(\cO) \rtimes \bG_m)_{t^\mu}$-equivariant $\left. \cA_\ula \right|_{t^\mu}$-module (here $(G(\cO) \rtimes \bG_m)_{t^\mu}$ is the stabilizer of $t^\mu \in \Gr$). 
Similarly for $\wti \Gr^{\ula^1} \times^\bL_{\bX} \wti \Gr^{\ula^2}$.

Note that $\wti \Psi$ is the composition of a flat pullback and an equivariant descent. Hence, it commutes with taking the $\shhom$-sheaf. Thus, we have $\wti \Psi(\cEnd \cE_{\ukey}) = \cEnd \cE_{\ula}$. Applying $\bR p_*$ and the flat base change \eqref{Mirkovic-Vybornov}, we obtain
\begin{equation*}
\Psi(\cA_{\ukey}) \simeq \cA_{\ula},
\end{equation*}
and hence the functor $\Psi$ induces the functor
\begin{equation*}
\Psi_\cA: D^b(\cA_{\ukey}\mods^{\GL_N \times \bG_m}) \rightarrow D^b(\cA_{\ula}\mods^{\GL_n(\cO) \rtimes \bG_m}),
\end{equation*}
such that we have $\Forg \circ \Psi_\cA \simeq \Psi \circ \Forg$ (here $\Forg$ stand for functors, forgetting the $\cA$-module structure on both sides). We have the following:
\begin{lem}
The following diagram commutes, with vertical functors being equivalences
\[\begin{tikzcd}[column sep = huge, row sep = large]
	{D^b(\mathrm{Coh}^{\GL_N \times \bG_m}(\widetilde{\mathcal N}_\ukey))} & {D^b(\mathrm{Coh}^{\GL_n(\mathcal O) \rtimes \bG_m}(\widetilde{\mathrm {Gr}}^{\underline \lambda}))} \\
	{D^b(\mathcal A_{\ukey}\mathrm{-mod}^{\GL_N \times \bG_m})} & {D^b(\mathcal A_{{\underline \lambda}}\mathrm{-mod}^{\GL_n(\mathcal O) \rtimes \bG_m})}
	\arrow["{\widetilde \Psi}", from=1-1, to=1-2]
	\arrow["{\bR s_* \bR\shhom( \cE_{\ukey}, -)}", "\cong"', from=1-1, to=2-1]
	\arrow["{\bR p_* \bR\shhom( \cE_{\ula}, -)}", "\cong"', from=1-2, to=2-2]
	\arrow["{\Psi_{\mathcal A}}", from=2-1, to=2-2]
\end{tikzcd}\]
\end{lem}

\begin{proof}
The claim follows immediately from the flat base change, applied to \eqref{Mirkovic-Vybornov}.
\end{proof}

So, instead of working with $\wti \Psi$, we work with $\Psi_\cA$, since it is better suited for the perversely-exotic t-structures. We have the following

\begin{prop} \label{prop: psi intertwines ic extensions}
\begin{enumerate}[a)]

\item $\Psi_\cA$ is $t$-exact with respect to the perversely-exotic $t$-structures of both sides. Therefore, it induces an exact functor between abelian categories:
\[
\Psi_\cA: \cP_{\exo}^{\GL_N \times \bG_m}\left(\mathcal{N}_{\ukey} \right) \rightarrow \cP_\exo^{\GL_n(\cO) \rtimes \bG_m}(\wti \Gr^\ula).
\]

\item  $\Psi_\cA$ is compatible with the $\IC$-extensions, i.e. for any $\GL_N$-orbit $j\colon O \hookrightarrow \cN_{\ukey}$, let $j'\colon O' \hookrightarrow \ol \Gr^\la$ be the corresponding $\GL_n(\cO)$-orbit; then
\[
\Psi_\cA \circ  j_{!*} \simeq j'_{!*} \circ \Psi_\cA.
\]

\end{enumerate}
\end{prop}

\begin{proof}
These claims are proved for categories of perverse $\cO$-modules in \cite[Proposition~4.12]{FF21}. Proofs for the categories of $\cA$-modules are the same.
\end{proof}

\begin{thm} \label{thm: psi for resolutions maps simples to simples}
The functor $\wti \Psi$ maps simple perversely-exotic sheaves on $\wti \cN_\ukey$ to simple perversely-exotic sheaves on $\wti \Gr^\ula$.
\end{thm}
\begin{proof}
In view of Proposition \ref{prop: psi intertwines ic extensions}, it is sufficient to show that restricted to a fixed orbit, $\Psi_\cA$ maps simples to simples. More formally, if $O$ is a $\GL_N$-orbit in $\cN_{\ukey}$, and $O'$ is the corresponding $\GL_n(\cO)$-orbit of $\ol \Gr^\la$, we need to check that 
\[
\Psi_\cA \vert_O: \cA_{\ukey}\vert_O\mods^{\GL_N \times \bG_m} \rightarrow \cA_{\ula}\vert_{O'}\mods^{\GL_n(\cO) \rtimes \bG_m}
\]
maps simple sheaves to simple. Since we work on group orbits now, it is sufficient to restrict to a point, what we now do.

Let $\Gr^\om$ be the closed orbit in $\ol \Gr^\la$. Take a closed point $p' \in O' \cap \cW^\la_\om$, which exists since $\cW^\la_\om$ intersects each orbit in $\ol \Gr^\la$, and let $p = \psi_2 \circ \theta (p') \in O$. We are interested in $\left. \cA_{\ula} \right|_{p'} = \left. \Psi (\cA_{\ukey})\right|_{p'}$. As in the proof of Proposition \ref{psi of tilting restrictied to slice is tilting}, we use that $i^*_\om \circ \Psi \simeq \theta^* \circ \psi_2^*$ (in Proposition~\ref{psi of tilting restrictied to slice is tilting} it was proved for functors on the top part of \eqref{Mirkovic-Vybornov}, the proof for the bottom part is the same). We get
\begin{equation} \label{eq: algebra at a point}
\left. \cA_{\ula} \right|_{p'} \simeq \left. \Psi (\cA_{\ukey})\right|_{p'} \simeq \left. \bigl( (\psi_2 \circ \theta)^* \cA_{\ukey} \bigr) \right|_{p'} \simeq \left. \cA_{\ukey} \right|_p.
\end{equation}

If $\om_n$ appears among $\om_{k_i}$, we showed during the proof of Proposition~\ref{prop:iso_on_K_theory} that 
the reductive parts of stabilizers $\Stab_{p'}(\GL_n(\cO) \rtimes \bG_m)$ and $\Stab_p(\GL_N \times \bG_m)$ are the same. In general, a similar easy computation shows that the reductive part of $\Stab_p(\GL_N \times \bG_m)$ is a direct factor in $\Stab_{p'}(\GL_N(\cO) \rtimes \bG_m)$ (see e.g. \cite[Section~5]{FF21}). Together with \eqref{eq: algebra at a point}, this  implies the result.
\end{proof}

Now let $\la^1$, $\la^2$ be compatible coweights and $\ula^1, \ula^2$ be any sequences of fundamental, defining the resolutions $\wti \Gr^{\ula^1} \rightarrow \ol \Gr^{\la^2}$, $\wti \Gr^{\ula^2} \rightarrow \ol \Gr^{\la^2}$. Let $\wti \cN_{\ukey^1} \rightarrow \cN_{\ukey^1}$ and $\wti \cN_{\ukey^2} \rightarrow \cN_{\ukey^2}$ the corresponding parabolic Springer resolutions for the Lie group $\GL_N$ ($N$ is the same since $\la^1$ and $\la^2$ are compatible). Then evidently, we have the functor 
\begin{equation*}
\wti \Psi_{\ula^1, \ula^2}: D^b(\Coh^{\GL_N \times \bG_m}(\wti \cN_{\ukey^1} \times^\bL_{\gl_N} \wti \cN_{\ukey^2})) \rightarrow D^b(\Coh^{\GL_n(\cO) \rtimes \bG_m}(\wti \Gr^{\ula^1} \times^\bL_{\bX} \wti \Gr^{\ula^2})),
\end{equation*}
defined as the product of $\wti \Psi_{\ula^1}$ and $\wti \Psi_{\ula^2}$. 

It immediately follows from the above that
\begin{equation*}
\Psi (\cA_{\ukey^1} \otimes^\bR_{\cO_{\gl_N}}  \cA_{\ukey^2}^{op} ) \simeq \cA_{\ula^1} \otimes^\bR_{\cO_{\bX}} \cA_{\ula^2}^{op},
\end{equation*}
and we have the following
\begin{thm} \label{thm: Psi for steinbergs}
The functor $\wti \Psi_{\ula^1, \ula^2}$ is t-exact w.r.t perversely-exotic t-structures on both sides.
It commutes with IC-extensions from orbits, and maps simple perversely-exotic sheaves to simple.
\end{thm}
\begin{proof}
Completely analogous to Proposition~\ref{prop: psi intertwines ic extensions} and Theorem~\ref{thm: psi for resolutions maps simples to simples}, working with the dg-algebra $\cA_{\ula^1} \otimes^\bR_{\cO_{\bX}} \cA_{\ula^2}^{op}$ instead of just $\cA_{\ula}$.
\end{proof}

Let us finish this section by mentioning that both functors $\widetilde{\Psi}, \widetilde{\Psi}_{\underline{\lambda}^1,\underline{\lambda}^2}$ as well as the versions $\Psi_\cA, \Psi_{\cA, \ula^1, \ula^2}$
actually factor through the quotient categories of equivariant complexes on the {\emph{open}} subsets $\widetilde{\mathcal N}_\ukey^{\leq n}$, $\wti \cN_{\ukey^1}^{\leq n} \times^\bL_{\gl_N} \wti \cN_{\ukey^2}^{\leq n}$. All of the structures we consider are compatible with taking these quotients. This will be important in the proof of Theorem~\ref{thm: exotic in convolution via Hecke}, as only after passing to these quotients the functors above induce isomorphisms on K-theory.

\subsection{Noncommutative resolutions for $\la$ and $\la + d\om_n$}\label{sec:tiltings_for_lambda_and_shift}
As in Section~\ref{subsec: modified lusztig}, we fix $d \in \bZ_{\geq 0}$ and consider
\begin{align*}
'\ul{\la}&=(\underbrace{\omega_n,\ldots,\omega_n}_d,\ul{\la}), & '\ul{k}&=(\underbrace{n,\ldots,n}_{d},\ul{k}), & '\la &= d \om_n + \la.
\end{align*}
We then have the modified Lusztig correspondence $\wti \Gr^{'\ula} \simeq \wti \Gr^{\ula} \rightarrow \left. \wti \cN_{'\ukey}^{\leq n} \right/ \GL_{N + dn}$ of Section~\ref{sec: modified Lusztig's correspondence}, and all the constructions of Sections~\ref{subsection: construction of tilting in type A},~\ref{subsec: tiltings for steinbergs},~\ref{subsec: titlings and simples} work equally well for arbitrary $d$ (not just $d = 0$), and all the statements remain true.

In this subsection, we prove that in fact these constructions are compatible, meaning that the noncommutative resolutions on $\wti \Gr^{'\ula} \simeq \wti \Gr^{\ula}$ coming from different $d$ are Morita-equivalent, hence t-structures 
discussed above are the same for all $d$. 

As in Section~\ref{subsec: modified lusztig}, we assume $d = 1$ to simplify notations, but all arguments are the same for arbitrary $d$. We state and prove the main result for Springer resolutions and provide Remark~\ref{rem: morita-equivalence for Steinbergs} about the corresponding Steinberg-type varieties. 
{The main result of this subsection is Proposition~\ref{prop: compatibility of t-structures for modified lusztigs}; we first need to establish some related general theory.}

First, let us recall the general construction \cite{Kal08} of tilting generators for conical symplectic resolutions $X \rightarrow Y$. One does the reduction to positive characteristic, considers a quantization of $X$ as an Azumaya algebra on the Frobenius twist of $X$, splits this Azumaya algebra on the formal fiber of the conical point of $Y$; spreading of this splitting to whole $X$ (using the conical action) is the desired tilting generator. Below we call it the ``tilting generator for $X$''. 

Now, for any reductive $G$, let $T^*\cB$  be the Springer resolution, and $T^*\cP$ be its parabolic version for some $B \subset P \subset G$. Let $S \subset \mathfrak{g}$ be an MV-slice to some nilpotent orbit (see  \cite[Definition~5.2]{WWY20}; Slodowy slices and the slices appearing in Section~\ref{sec: modified Lusztig's correspondence} are particular cases).
As in \cite{BM13}, we denote by $(T^*\cB)_S$ the base change along $S \rightarrow \g$, and similarly for other varieties and morphisms. Note that $S \rightarrow \g$ has the property of exact base change.

\begin{prop}\footnote{We learned this argument from Roman Bezrukavnikov. All possible mistakes are ours.} \label{prop: morita-equivalence on slices}
Let $\cE_\cP$ be the tilting generator for the conical symplectic resolution $T^* \cP$. Let $\cE_\cP'$ be the tilting generator for the conical symplectic resolution $(T^* \cP)_S$. Then the restriction of $\cEnd (\cE_\cP)$ to $(T^* \cP)_S$ is Morita-equivalent to $\cEnd(\cE_\cP')$.
\end{prop}

\begin{proof}
For the case of full flag variety, this is \cite[Lemma~4.7]{BL23}. We deduce the parabolic case using the correspondences of \cite[Section~4]{BM13}.

We have the diagram with Cartesian squares
\begin{equation} \label{eq: diagram of correspondences for parabolic springers and slodowy}
\begin{tikzcd}
	& {(T^*\mathcal P \times_{\mathcal P} \mathcal B)_S} & \\
	{(T^*\mathcal P)_S} && {(T^*\mathcal B)_S} \\
	& {T^*\mathcal P \times_{\mathcal P} \mathcal B} \\
	{T^*\mathcal P} && {T^* \mathcal B}
	\arrow["{(pr_1)_S}"', from=1-2, to=2-1]
	\arrow["{i_S}", hook, from=1-2, to=2-3]
	\arrow[hook, from=1-2, to=3-2]
	\arrow[hook, from=2-1, to=4-1]
	\arrow[hook, from=2-3, to=4-3]
	\arrow["{pr_1}"', from=3-2, to=4-1]
	\arrow["i", hook, from=3-2, to=4-3]
\end{tikzcd}
\end{equation}
Denote $\pi_\star = (pr_1)_* \circ i^*$, and let $\cE_\cB'$ be the tilting generator for symplectic resolution $(T^*\cB)_S$.

It is shown in \cite[Section~4]{BM13} that $\pi_\star \cE_\cB = \cE_\cP$. Also by the flat base change (recall that $(T^*\cP)_S \rightarrow (T^*\cP)/G$ is smooth) we have $(\cE_\cP)_S = (\pi_\star \cE_\cB)_S \simeq (\pi_S)_\star( (\cE_\cB)_S)$. Since $(\cE_\cB)_S$ and $\cE_\cB'$ are equiconstituted by \cite[Lemma~4.7]{BL23}, the desired result will follow once we  prove that $(\pi_S)_\star \cE_\cB' \simeq \cE_\cP'$.

Recall \cite{GG02} that the MV-slices can be realized as the Hamiltonian reduction by a unipotent group $U_\ell$ w.r.t. to an additive character $\chi$ (see \cite[Section~5.1]{WWY20} and \cite[Section~2.4]{HKM24} for the parabolic case). The bottom part of \eqref{eq: diagram of correspondences for parabolic springers and slodowy} is obviously $U_\ell$-equivariant, and if we take the Hamiltonian reduction, we get precisely the top part. For the left and right terms, this is just the construction of \cite{GG02}, and for the middle term it is formally deduced from \cite{GG02}: one shows that the action morphism $U_\ell \times (T^*\mathcal P \times_{\mathcal P} \mathcal B)_S \rightarrow \nu^{-1}(\chi)$ is an isomorphism, where $\nu$ is the moment map for the $U_\ell$-action on $ T^*\mathcal P \times_{\mathcal P} \mathcal B$.

Hence, just like the bottom part of \eqref{eq: diagram of correspondences for parabolic springers and slodowy}, studied in \cite[Section~4]{BM13}, comes from the studying pullback and pushforward of D-modules along the morphism $G / B \rightarrow G/P$, the top part is the direct analog for the morphism $U_\ell \backslash\!\!\backslash\!\!\backslash_\chi G / B \rightarrow U_\ell \backslash\!\!\backslash\!\!\backslash_\chi G / P$. Then completely analogously to \cite[Theorems~4.1,~4.2]{BM13} one deduces the required isomorphism $(\pi_S)_\star \cE_\cB' \simeq \cE_\cP'$ for tilting generators, coming from quantizations of $T^*(U_\ell \backslash\!\!\backslash\!\!\backslash_\chi G / B) \simeq (T^*\cB)_S$ and $ T^*(U_\ell \backslash\!\!\backslash\!\!\backslash_\chi G / P) \simeq (T^*\cP)_S$, as required.
\end{proof}

It is likely that the claim of Proposition~\ref{prop: morita-equivalence on slices} can be proved in a broader setting of~\cite{KT21} using the results {\it loc.cit}.

\begin{prop} \label{prop: compatibility of t-structures for modified lusztigs}
The diagrams~\eqref{eq:comp_lusztig_corresp} are compatible with perversely-exotic t-structures. That is, the pullback of a BM tilting under $\wti \Gr^\ula \rightarrow \left. \wti \cN^{\leq n}_\ukey \right/ \GL_N$, defines a noncommutative resolution, Morita-equivalent to the noncommutative resolution, defined by the pullback of a BM tilting under $\wti \Gr^{\ula} \simeq \wti \Gr^{'\ula} \rightarrow \left. \wti \cN^{\leq n}_{'\ukey} \right/ \GL_{N + n}$.
\end{prop}
\begin{proof}
Our above arguments in Section~\ref{subsection: conjecture in any type} show that the current group equivariant sheaf on $\wti \Gr^{\ula}$ is determined by its restriction to the maximal slice. This slice is isomorphic to the maximal slice in both $\wti \cN^{\leq n}_{\ukey}$ and $\wti \cN^{\leq n}_{'\ukey}$, and the morphism $\left. \wti \cN^{\leq n}_{'\ukey} \right/ \GL_{N + n} \rightarrow \left. \wti \cN^{\leq n}_{\ukey} \right/ \GL_{N}$ of Lemma~\ref{lem_com_our_morphisms} induces a symplectic isomorphism on these varieties (see Corollary~\ref{cor:iso_pull_st}).
So, it is sufficient to prove that the restrictions of BM tiltings to these pre-images of slices define Morita-equivalent noncommutative resolutions.

By Proposition~\ref{prop: morita-equivalence on slices} we see that both these restrictions define the algebra, Morita-equivalent to the one, defined in terms, internal for this symplectic resolution. The result follows.
\end{proof}

\begin{rem} \label{rem: morita-equivalence for Steinbergs}
Let us comment on the Steinberg case. Given compatible $\la^1, \la^2$, it follows from Proposition~\ref{prop: compatibility of t-structures for modified lusztigs} that the $\dg$-algebras $\cA_{\ula^1} \T^\bR_{\cO_{\bX}} \cA_{\ula_2}^{op}$ and $\cA_{{}'\ula^1} \T^\bR_{\cO_{\bX}} \cA_{{}'\ula_2}^{op}$ are Morita-equivalent. Further, there is a morphism $\cA_{{}'\ula^1} \T^\bR_{\cO_{{}'\bX}} \cA_{{}'\ula_2}^{op} \rightarrow \cA_{{}'\ula^1} \T^\bR_{\cO_{\bX}} \cA_{{}'\ula_2}^{op}$, which is clearly an isomorphism on $H^0$. In particular, it induces a bijection between classes of simple perversely-exotic sheaves, see Remark~\ref{rem: simples are supported on zero coh}.
\end{rem}

\subsection{Description of the basis via the Kazhdan--Lusztig canonical basis}\label{sec:descr_of_exotic_via_canonical}
In Section~\ref{subsec: titlings and simples} we defined the perversely-exotic t-structures on the derived categories of equivariant sheaves on $\wti \Gr^{\underline \la^1} \times^\bL_\bX \wti \Gr^{\underline \la^2}$ and $\wti \Gr^{\underline \la}$. Classes of simple objects in the hearts of these t-structures give bases in $K^{\GL_n(\cO) \rtimes \bG_m} (\wti \Gr^{\underline \la^1} \times_\Gr \wti \Gr^{\underline \la^2})$ and $K^{\GL_n(\cO) \rtimes \bG_m} (\wti \Gr^{\underline \la^1})$ respectively.
We are now ready to identify this basis 
with the Kazhdan--Lusztig basis.

Recall that to $\underline{\lambda}$ we associate $'\underline{k}$. We denote by ${}^{'\underline{k}}\mathcal{H}^{\mathrm{asph}}_{\leq n}$, ${}^{'\underline{k}^1}\mathcal{H}^{'\underline{k}^2}_{\leq n}$ the corresponding cell quotients of the spherical--anti-spherical and spherical--spherical modules, see Section~\ref{subsec: canonical bases in parabolic hecke}. Recall that, geometrically, 
\begin{equation*}
{}^{'\underline{k}^1}\mathcal{H}^{'\underline{k}^2}_{\leq n} = K^{\operatorname{GL}_{N+dn} \times \mathbb{G}_m}({\operatorname{St}_{'\ul{k}^1,'\ul{k}^2}^{\leq n}})=K^{\operatorname{GL}_{N+dn} \times \mathbb{G}_m}({\operatorname{St}_{'\ul{k}^1,'\ul{k}^2}})/K^{\operatorname{GL}_{N+dn} \times \mathbb{G}_m}({\operatorname{St}_{'\ul{k}^1,'\ul{k}^2}^{>n}}),
\end{equation*}
\begin{equation*}
{}^{'\underline{k}}\mathcal{H}^{\mathrm{asph}}_{\leq n} = K^{\operatorname{GL}_{N+dn} \times \mathbb{G}_m}({\widetilde{\mathcal{N}}_{'\ul{k}}}^{\leq n}) = K^{\operatorname{GL}_{N+dn} \times \mathbb{G}_m}({\widetilde{\mathcal{N}}_{'\ul{k}}})/K^{\operatorname{GL}_{N+dn} \times \mathbb{G}_m}({\widetilde{\mathcal{N}}_{'\ul{k}}}^{>n}),
\end{equation*}
where by $\bullet^{>n}$ we mean the preimages of the union of nilpotent orbits with at least $n+1$ Jordan blocks.

\begin{thm}\label{thm: exotic in convolution via Hecke}
\begin{enumerate}[(a)]
\
\item For any $d \geqslant 0$ there exist explicit homomorphisms
\begin{align}\label{eq: maps hecke to convolution not intro}
{}^{'\underline{k}^1}\mathcal{H}^{'\underline{k}^2}_{\leq n} &\rightarrow 
K^{\GL_n(\cO) \rtimes \bG_m} (\wti \Gr^{\ula^1} \times_{\Gr} \wti \Gr^{\ula^2}), & {}^{'\underline{k}}\mathcal{H}^{\mathrm{asph}}_{\leq n} &\rightarrow K^{\GL_n(\cO) \rtimes \bG_m} (\wti \Gr^\ula) 
\end{align}
compatible with algebra and module structures. The homomorphisms (\ref{eq: maps hecke to convolution not intro}) are injective for $d=0$ and isomorphisms for $d>0$.

\item Homomorphisms (\ref{eq: maps hecke to convolution not intro}) send Kazhdan--Lusztig canonical basis to the perversely-exotic basis.
\end{enumerate}
\end{thm}
\begin{proof}
Theorem \ref{thm: simple perversely-exotic are canonical in parabolic hecke} combined with Corollary \ref{cor: Psi define iso on K-theory} gives the desired homomorphisms \eqref{eq: maps hecke to convolution not intro}. 
 Theorems \ref{thm: Psi for steinbergs}, \ref{thm: psi for resolutions maps simples to simples}, \ref{thm: simple perversely-exotic are canonical in parabolic hecke} combined with Proposition \ref{prop: compatibility of t-structures for modified lusztigs} imply that these identifications send Kazhdan--Lusztig canonical basis to the perversely-exotic basis.
\end{proof}

\subsection{K-theoretic quantum affine skew Howe duality}

In this section, we reprove some results of Cautis--Kamnitzer (see \cite[Sections~5,~6]{CK18b}) using geometry.


Recall $\dot U_v(\widehat \gl_m)$ --- Lusztig's idempotent form of the quantum affine group. Recall \cite[5.1]{CK18b} that it may be viewed as a category with objects being sequences $(k_1, \hdots, k_m) \in \bZ^m$ and certain morphisms between them (our $v$ is their $q$). We note that there are no nonzero morphisms between $(k^1_1, \hdots, k^1_m)$ and $(k_1^2, \hdots, k_m^2)$ unless $\sum_i k^1_i = \sum_i k^2_i$, since the sum of coordinates of any simple root is 0.
As in \cite{CK18b}, we consider its quotient $\dot U_v(\widehat \gl_{m})^{n}$, obtained by killing all the objects except for those of the form $(k_1, \hdots, k_{m})$ with $0 \leq k_i \leq n$ for any $i$.
Depending on the context, we view it either  as a category, or as an algebra (with idempotents) --- the direct sum of $\Hom$'s in this category.

We denote $\wh \cH_{N} = \wh \cH_{\GL_N}$ --- the affine Hecke algebra for $\GL_N$. Recall from Section~\ref{sec:sph_sph} its  elements $C_{w_{L,0}}$ for a Levi $L$.

Recall the quantum affine Schur algebra:
\begin{equation*}
\mathbb S(m, N) = \End_{\widehat \cH_{N}}\Big(\bigoplus_{L} \widehat \cH_{N} C_{w_{L,0}}\Big),
\end{equation*}
where the sum runs over Levi $L = (k_1, \hdots, k_{m})$ with $\sum_i k_i = N$.
We can rewrite
\begin{equation} \label{eq: schur algebra as double sum of hecke two-sided quotients}
\mathbb S(m, N) \simeq \bigoplus_{L_1, L_2} \Hom_{\widehat \cH_{N}}(\widehat \cH_{N} C_{w_{L_1,0}}, \widehat \cH_{N} C_{w_{L_2,0}}) \simeq \bigoplus_{L_1, L_2} C_{w_{L_1,0}} \widehat \cH_{N} C_{w_{L_2,0}}.
\end{equation}

We can pack $\mathbb S(m, N)$ for all $N$ in the {\bf quantum affine Schur category} $\dot{\mathbb S}(m)$ (analogously to $\dot U_v(\widehat \gl_m)$) in the following way: the objects of the category $\dot{\mathbb S}(m)$ are elements $(k_1, \hdots, k_m) \in \bZ^m_{\geq 0}$; the $\Hom$ space between $(k^1_1, \hdots, k^1_m)$ and $(k_1^2, \hdots, k_m^2)$ is zero if $\sum k_i^1 \neq \sum k_i^2$, and is given by $C_{w_{L_1,0}} \widehat \cH_{N} C_{w_{L_2,0}}$ if $\sum k_i^1 = \sum k_i^2 = N$ (here $L_1 = (k_1^1, \hdots, k_m^1)$ and $L_2 = (k_1^2, \hdots, k_m^2)$). The composition is defined in an obvious way. Geometrically, the $\Hom$-space is given by $K^{\GL_N \times \bG_m} (\wti \cN_{\ukey^1} \times_{\gl_N} \wti \cN_{\ukey^2})$.

We define $\dot{\mathbb{S}}(m)^n$ as the quotient, obtained by killing all objects except for those with $0 \leq k_i \leq n$ for all $i$. 

\begin{prop}
The $\Hom$ spaces in $\dot{\mathbb{S}}(m)^n$ are given by $K^{\GL_N \times \bG_m} (\wti \cN_{\ukey^1}^{\leq n} \times_{\gl_N} \wti \cN_{\ukey^2}^{\leq n})$.
\end{prop} 
\begin{proof}
Let $I_n$ be the two-sided ideal in $\dot{\mathbb{S}}(m)$ generated by idempotents $1_{\underline{k}}$ corresponding to $\underline{k}$ with $k_i>n$ for some $i \in \{1,\ldots,m\}$. Let $J_n$ be the two-sided ideal in $\dot{\mathbb{S}}(m)$ corresponding to the union of Lusztig cells such that their partitions have at least $n+1$ blocks. Our goal is to show that $I_n=J_n$. Let $\underline{k}$ be such that $k_i>n$ for some $i$. Then $\wti{\mathcal{N}}_{\underline{k}}=\wti{\mathcal{N}}_{\underline{k}}^{> n}$, so $1_{\underline{k}} \in J_n$, hence,  $I_n \subset J_n$. It remains to prove that $1_{\underline{k}}$ as above generate $J_n$. 
That can be shown after passing to the associated graded under the cell filtration.

 Fix a nilpotent $e$, the corresponding subquotient in $K^{\GL_N \times \bG_m} (\wti \cN_{\ukey^1} \times_{\gl_N} \wti \cN_{\ukey^2})$ is isomorphic to $K^{Z_{\GL_N \times \bG_m}(e)}(s_{\ukey^1}^{-1}(e) \times s_{\ukey^2}^{-1}(e))$, where $s_{\ukey^i}\colon \wti \cN_{\ukey^i} \rightarrow \mathfrak{gl}_N$ is the natural map. Let $\underline{k}$ be such that $\widetilde{\mathcal{N}}_{\underline{k}} \rightarrow \overline{\mathbb{O}}_e$ is a resolution. Then, $s_{\underline{k}}^{-1}(e)$ consists of {\emph{one}} point and the  convolution map (compare with \cite[Theorem B.2]{TX06})
 \begin{multline*}
 K^{Z_{\GL_N \times \bG_m}(e)}(s_{\ukey^1}^{-1}(e) \times s_{\underline{k}}^{-1}(e)) \otimes_{K^{Z_{\GL_N \times \bG_m}(e)}(\operatorname{pt})} K^{Z_{\GL_N \times \bG_m}(e)}(s_{\underline{k}}^{-1}(e) \times s_{\ukey^2}^{-1}(e)) \rightarrow  \\ \rightarrow K^{Z_{\GL_N \times \bG_m}(e)}(s_{\ukey^1}^{-1}(e) \times s_{\underline{k}^2}^{-1}(e))
 \end{multline*}
 given by $x \otimes y \mapsto x \boxtimes y$
 is an isomorphism by the equivariant Kunneth formula which can be applied in this setting as all of the spaces above admit 
 ${Z_{\GL_N \times \bG_m}(e)}$-invariant
 affine pavings (same proof as of \cite[Theorem 3.7(3)]{NY04}). This proves that 
 after passing to the associated graded the ideal generated by the class of $1_{\underline{k}}$ contains $K^{Z_{\GL_N \times \bG_m}(e)}(s_{\ukey^1}^{-1}(e) \times s_{\underline{k}^2}^{-1}(e))$. 
\end{proof}

The quantum affine Schur algebra realizes the quantum affine Schur--Weyl duality, meaning that there is a natural functor $\dot U_v(\wh \gl_m) \rightarrow \dot{\mathbb{S}}(m)$, see e.g. \cite{SV00}. It obviously induces a functor 
\begin{equation} \label{eq: from quantum group to schur}
\Theta_m\colon  \dot U_v(\wh \gl_m)^n \rightarrow \dot{\mathbb{S}}(m)^n.  
\end{equation}


Also following \cite{CK18b}, we define the {\bf convolution category} $\mathrm{Conv}^n$. To be consistent with \cite{CK18b}, we work with Schubert varieties for $\PGL_n$ (they are {{isomorphic}} to the $\operatorname{GL}_n$-Schubert varieties we consider in the paper). The objects of $\mathrm{Conv}^n$ are sequences $(k_1, \hdots, k_\ell)$ with $1 \leq k_i \leq n - 1$ (for some $\ell$), and the morphisms are given by $K^{\GL_n(\cO) \rtimes \bG_m}(\wti \Gr_{\PGL_n}^{\ukey^1} \times_{\Gr_{\PGL_n}} \wti \Gr^{\ukey^2}_{\PGL_n})$. Note that again the $\Hom$ space is zero, unless $\sum k_i^1 - \sum k_i^2$ is divisible by $n$, since  otherwise the images of $\wti \Gr^{\ukey^1}_{\PGL_n}$ and $\wti \Gr^{\ukey^2}_{\PGL_n}$ lie in different connected components of $\Gr_{\PGL_n}$.


Recall also \cite[Section 5.4]{CK18b} the functors $\dot U_v(\wh \gl_m)^n \rightarrow \dot U_v(\wh \gl_{m + 1})^n$, which at the level of objects insert 0 or $n$ at some place. They are faithful by~\cite[Proposition~5.6]{CK18b}.

\begin{thm}[K-theoretic quantum affine skew Howe duality] \label{thm: quantum affine skew howe duality} \

\begin{enumerate}[1.]
\item \label{it: functor from U to conv}
For any $m$, there is a functor
\begin{equation} \label{eq: from quantum group to conv}
\dot U_v(\wh \gl_m)^n \rightarrow \mathrm{Conv}^n. 
\end{equation}
On objects, it removes 0's and $n$'s from sequences $(k_1, \hdots, k_m)$.

\item \label{it: embeddings of quantum groups are compatible with conv} All embeddings $\dot U_v(\gl_m)^n \rightarrow \dot U_v(\gl_{m + 1})^n$ are compatible with functors~\eqref{eq: from quantum group to conv}. 

\item \label{it: surjectivity for m + 2} For any $\ukey^1, \ukey^2 \in \bZ^{m}$, such that $\sum_i k_i^1=\sum_i k_i^2$, the map
\begin{equation*}
\Hom_{\dot U_v({\wh \gl_{m + 2}})}((0, \ukey^1, n), (0, \ukey^2, n)) \rightarrow \Hom_{\mathrm{Conv}^n}('\ukey^1, '\!\ukey^2)
\end{equation*}
is surjective. It remains surjective after adding arbitrarily more 0's or $n$'s.
\end{enumerate}
\end{thm}
Part~\ref{it: surjectivity for m + 2} should be thought of as a strengthening of \cite[Theorem~5.8]{CK18b}, where it is only claimed that the map becomes surjective after adding {\it some number} of 0's and $n$'s.
\begin{proof}
Let us prove \ref{it: functor from U to conv} We claim that there is a functor $\dot{\mathbb{S}}(m)^n \rightarrow \mathrm{Conv}^n$, and composing it with~\eqref{eq: from quantum group to schur} we get the claim. On objects, it sends $\underline{k}$ to $\underline{k}'$ obtained 
by removing all 0's and $n$'s. On morphisms it is given by:
\begin{multline*}
\Hom_{\dot{\mathbb S}(m)^n} (\ukey^1, \ukey^2) = 
K^{\GL_{N } \times \bG_m} (\wti \cN_{\ukey^1}^{\leq n} \times_{\gl_{N}} \wti \cN_{\ukey^2}^{\leq n}) \rightarrow \\ \rightarrow K^{\GL_n(\cO) \rtimes \bG_m} (\wti \Gr_{\operatorname{GL}_n}^{\ukey^1} \times_{\Gr_{\PGL_n}} \wti \Gr_{\GL_n}^{\ukey^2}) \iso \\
\iso
K^{\GL_n(\cO) \rtimes \bG_m} (\wti \Gr_{\PGL_n}^{'\ukey^1} \times_{\Gr_{\PGL_n}} \wti \Gr_{\PGL_n}^{'\ukey^2}) =
\Hom_{\mathrm{Conv}^n}('\ukey^1, '\ukey^2),
\end{multline*}
where the first nontrivial arrow is given by Lusztig correspondence of Section~\ref{sec: modified Lusztig's correspondence}, and the second one is the natural isomorphism.
This proves the claim.

Let us prove~\ref{it: embeddings of quantum groups are compatible with conv}
Let $\varepsilon\in\{0,n\}$, and let
\[
\iota^\varepsilon\colon
\dot U_v(\widehat{\mathfrak{gl}}_m)^n
\longrightarrow
\dot U_v(\widehat{\mathfrak{gl}}_{m+1})^n
\]
be one of the embeddings obtained by inserting $\varepsilon$ into a
sequence. Since all the constructions are equivariant with respect to
cyclic relabeling of the entries, it is enough to consider insertion in
the first position.

First, we consider the case $\varepsilon = 0$. Then in \cite{Li24}, the corresponding functor
\begin{equation} \label{eq: embedding of schurs}
\jmath^0 \colon
\dot{\mathbb S}(m)^n \rightarrow \dot{\mathbb S} (m+1)^n    
\end{equation}
is defined, and the commutativity of the top square of the diagram
\begin{equation} \label{eq: diagram embedding and conv}
\begin{tikzcd}
	{\dot U_v(\widehat{\mathfrak{gl}}_m)^n} && {\dot U_v(\widehat{\mathfrak{gl}}_{m+1})^n} \\
	{\dot{\mathbb S}(m)^n} && {\dot{\mathbb S}(m + 1)^n} \\
	& {\mathrm{Conv}^n}
	\arrow["{\iota^0}", from=1-1, to=1-3]
	\arrow["{\Theta_m}", from=1-1, to=2-1]
	\arrow["{\Theta_{m + 1}}", from=1-3, to=2-3]
	\arrow["{\jmath^0}", from=2-1, to=2-3]
	\arrow[from=2-1, to=3-2]
	\arrow[from=2-3, to=3-2]
\end{tikzcd}
\end{equation}
is proved, see~\cite[Proposition~2.3.1]{Li24}.
The commutativity of the lower triangle in~\eqref{eq: diagram embedding and conv} 
is obtained by deleting the repeated step on both sides of the Lusztig correspondence. This proves the claim for $\varepsilon = 0$.

To prove the claim for $\varepsilon = n$, we note that there is the Chevalley automorphism of $\dot U_v(\wh \gl_m)^n$, defined by 
\begin{equation*}
(k_1,\ldots,k_m)\longmapsto
(n-k_1,\ldots,n-k_m),
\qquad
E_i^{(s)}\longleftrightarrow F_i^{(s)}.
\end{equation*}
The corresponding automorphism of $\mathrm{Conv}^n$ is induced by the inverse-transpose automorphism of $\PGL_n$. It is easy to see that these automorphisms are intertwined by $\dot U_v(\wh \gl_m)^n \rightarrow \mathrm{Conv}^n$, and, moreover, they interchange the 0-insertion with $n$-insertion embeddings of $\dot U_v(\wh \gl_m)^n$. This proves the claim for $\varepsilon = n$.\footnote{Functors~\eqref{eq: embedding of schurs} and diagram~\eqref{eq: diagram embedding and conv} make perfect geometric meaning for both both $\varepsilon = 0$ and $\varepsilon = n$, and the commutativity may be shown using the affine version of the geometric language of~\cite{BLM90}. We do not do it here to save space.}

Let us prove~\ref{it: surjectivity for m + 2} Denote
$\widetilde{\ukey}^{a} = (0,\ukey^a,n)$ for $a=1,2$.
Put $\sum_i k_i^1=\sum_i k_i^2=N$.
The map in the statement factors as
\begin{equation} \label{eq: hom from quantum group to conv for m + 2 decomposition}
\operatorname{Hom}_{
  \dot U_v(\widehat{\mathfrak{gl}}_{m+2})}
  \bigl(\widetilde{\ukey}^{1},
        \widetilde{\ukey}^{2}\bigr) \rightarrow
\operatorname{Hom}_{\dot {\mathbb{S}}(m+2)^n}
  \bigl(\widetilde{\ukey}^{1},
        \widetilde{\ukey}^{2}\bigr)  \rightarrow
\operatorname{Hom}_{\operatorname{Conv}^n}
  (\ukey^1,\ukey^2).
\end{equation}
Below we show that our addition of 0 ensures that the first arrow is surjective, and the addition of $n$ ensures that the second arrow is an isomorphism.

By \cite[(2.2.2)]{DF10} the periodic basis of the block
\[
C_{w_{\widetilde{\ukey}^{2}, 0}}
\dot {\mathbb{S}}(m+2,  N + n)
C_{w_{\widetilde{\ukey}^{1}, 0}}
\]
is indexed by periodic matrices $A=(a_{ij})_{i,j\in\mathbb Z}$
whose row sums and column sums are
$\widetilde{\ukey}^{2}$ and
$\widetilde{\ukey}^{1}$, respectively. Such a matrix is called aperiodic if, for every $r\neq0$, there is an $i$ such that $ a_{i,i+r} = 0$.

The first component of both
$\widetilde{\ukey}^{1}$ and
$\widetilde{\ukey}^{2}$ is zero. Hence the corresponding row
of every matrix indexing this block is identically zero. In
particular, for every $r\neq0$, $a_{1,1+r}=0$.
Thus every periodic matrix occurring in this block is aperiodic.

By \cite[Section~4.3]{SV00}, the image of
\[
\dot U_v(\widehat{\mathfrak{gl}}_{m+2})
\longrightarrow
\dot{\mathbb S}(m+2, N + n)
\]
is spanned by the canonical basis elements indexed by aperiodic
matrices. It follows that the map onto the above block is
surjective. 
Therefore, the first arrow in~\eqref{eq: hom from quantum group to conv for m + 2 decomposition} is surjective.

The second arrow in~\eqref{eq: hom from quantum group to conv for m + 2 decomposition} is an isomorphism by the modified Lusztig correspondence, Corollary~\ref{cor: Psi define iso on K-theory}. 
\end{proof}



\begin{cor}\label{cor: action on module}
\begin{enumerate}[1.]
\item For any $d \geq 0$, there is a natural structure of a $U_v(\widehat \gl_{m + d})$-module on the space
\begin{equation} \label{eq: sum of K-theories of convolutions}
\bigoplus_{\underline k} K^{\GL_n(\cO) \rtimes \bG_m} (\wti \Gr^{\underline k}_{\operatorname{PGL}_n}),
\end{equation}
where the sum runs over all tuples 
$\underline k = (k_1, \hdots, k_m)$ with $1 \leq k_i \leq n-1$ and with a fixed $\sum_{i} k_i$ mod $n$. 

\item For $d_1, d_2 >0$ with $d=d_1+d_2$ each summand $K^{\GL_n(\cO) \rtimes \bG_m} (\wti \Gr^{\underline k})$ is cyclic as a module over 
\begin{equation}\label{eq:map_homs_U_v_conv}
\operatorname{Hom}_{\dot{U}_v(\widehat{\mathfrak{gl}}_{m+d})^n}\Big(\scalebox{1.4}{$1$}_{\underbrace{0,\ldots,0}_{d_1},\scalebox{1.1}{$\underline{k}$},\underbrace{n,\ldots,n}_{d_2}},
\scalebox{1.4}{$1$}_{\underbrace{0,\ldots,0}_{d_1},\scalebox{1.1}{$\underline{k}$},\underbrace{n,\ldots,n}_{d_2}}\Big) \twoheadrightarrow K^{\GL_n(\cO) \rtimes \bG_m} (\wti \Gr^{\underline k} \times_\Gr \wti \Gr^{\underline k}),
\end{equation}
where the map (\ref{eq:map_homs_U_v_conv}) is induced by (\ref{eq: from quantum group to conv}). The generator is the structure sheaf of the diagonal $\wti \Gr^{\underline k}  \hookrightarrow \wti \Gr^{\underline k} \times_\Gr \wti \Gr^{\underline k}$.
\end{enumerate}
\end{cor}
\begin{proof}
Both of the claims follow from Theorem~\ref{thm: quantum affine skew howe duality}. 
\end{proof}

\subsection{Description of the basis via quantum group} \label{subsec: description of basis for schuberts}

We are ready to prove our theorem on description of the basis in equivariant K-theory of the convolution diagram of affine Schubert varieties via quantum groups.



For two vector spaces $V_1, V_2$ (possibly algebras) with fixed bases, we say that a surjective morphism $V_1 \twoheadrightarrow V_2$ between them is {\it compatible} with the bases if every basis vector in $V_2$ has a unique basis vector in $V_1$ in its pre-image, and every basis vector in $V_1$ is mapped either to a basis vector or to zero.

As we explained in Section~\ref{subsec: canonical bases in parabolic hecke}, there is the Kazhdan--Lusztig canonical basis in each summand of \eqref{eq: schur algebra as double sum of hecke two-sided quotients}, and so $\mathbb{S}(m, N)$ possesses the Kazhdan--Lusztig basis. On the other hand, there is the Lusztig canonical basis \cite{Lus10} in $\dot U_v(\widehat \gl_{m})$. It is the main result of \cite{SV00} that the morphism~\eqref{eq: from quantum group to schur} is compatible with these canonical bases in the above sense. 
So the canonical basis in $\mathbb{S}(m,  N)$ resembles both the relation of this algebra to the affine Hecke algebra, and to the affine quantum group.




\begin{thm} \label{thm: main theorem description of basis}
\begin{enumerate}
\item[1.] The perversely-exotic basis in the Hom spaces of the category $\operatorname{Conv}^n$
is a part of the Lusztig canonical basis of $\dot U_v(\widehat \gl_{m})$. 
That is, the functor (\ref{eq: from quantum group to conv})
is compatible w.r.t. the perversely-exotic basis in $\operatorname{Conv}^n$ and the canonical basis in $\dot U_v(\widehat \gl_{m})$. 

\item[2.] The perversely-exotic basis in $K^{\GL_n(\cO) \rtimes \bG_m} (\wti \Gr^{\underline \la})$ coincides with the canonical basis, if we view this space as a cyclic module over the Hom-space (\ref{eq:map_homs_U_v_conv}). 
\end{enumerate}
\end{thm}
\begin{proof}
By Corollary~\ref{cor: Psi define iso on K-theory}, the functors $\wti \Psi_{\ula}$ (resp. $\wti \Psi_{\ula_1, \ula_2}$) define isomorphisms at the level of equivariant K-theory with the ($\leq n$)-part of the corresponding Springer (resp. Steinberg) varieties. By Theorems~\ref{thm: psi for resolutions maps simples to simples},~\ref{thm: Psi for steinbergs}, they map classes of simple perversely-exotic sheaves to the classes of simple perversely-exotic sheaves. By Theorem~\ref{thm: simple perversely-exotic are canonical in parabolic hecke}, classes of simples correspond to the Kazhdan--Lusztig basis. Both parts of the theorem follow.
\end{proof}

\subsection{Relation to the constructions of Ginzburg--Vasserot and Nakajima}
Let us finish this section by mentioning the relation between our approach and  Ginzburg--Vasserot's and Nakajima's geometric constructions of representations of quantum loop algebras via the geometry of quiver varieties.

In Ginzburg--Vasserot's approach (see \cite{GV93, Vas98}) one starts with the space $\mathscr{F}$ of all $m$-step partial flags:
\begin{equation*}
\mathscr{F} = \{\{0\}=F_0 \subseteq F_1 \subseteq \ldots \subseteq F_{m-1} \subseteq F_m = \mathbb{C}^N\}, \quad \mathscr{F}=\bigsqcup_{\underline{k}} \mathscr{F}_{\underline{k}},
\end{equation*}
where for $\underline{k}=(k_1,\ldots,k_m)$, $\mathscr{F}_{\underline{k}}$ is the connected component of $\mathscr{F}$ consisting of $F_{\bullet}$ with $\operatorname{dim}F_i/F_{i-1}=k_i$.
Consider the cotangent bundle $T^*\mathscr{F}$. In our notation,
\begin{equation*}
T^*\mathscr{F} = \bigsqcup_{(k_1,\ldots,k_m)}\widetilde{\mathcal{N}}_{\underline{k}},
\end{equation*}
so 
\begin{equation}\label{eq:K theory Schur}
K^{\operatorname{GL}_N \times \mathbb{G}_m}(T^*\mathscr{F} \times_{\mathfrak{gl}_N} T^*\mathscr{F}) = \mathbb{S}(m,N).    
\end{equation}

Then one considers a homomorphism  (compare with (\ref{eq: from quantum group to schur}) above)
\begin{equation}\label{eq:U to S}
U_v(\widehat{\mathfrak{gl}}_m) \rightarrow \mathbb{S}(m,N)
\end{equation}
and constructs a family of {\emph{standard}} modules over $U_v(\widehat{\mathfrak{gl}}_m)$ via this homomorphism as follows. For any nilpotent $e \in \mathfrak{gl}_N$ with $e^m=0$, let $\mathscr{F}_e \subset T^*\mathscr{F}$ be the fiber over $e$ under the Springer map. The identification (\ref{eq:K theory Schur}) combined with (\ref{eq:U to S})
gives the action 
\begin{equation}\label{eq:action of quantum affine on fiber}
U_v(\widehat{\mathfrak{gl}}_m) \curvearrowright K^{Z_{\operatorname{GL}_N \times \mathbb{G}_m}(e)}(\mathscr{F}_e)   
\end{equation}
by $K^{Z_{\operatorname{GL}_N \times \mathbb{G}_m}(e)}(\operatorname{pt})$-linear endomorphisms. Specializing to $s \in \operatorname{Spec}K^{Z_{\operatorname{GL}_N \times \mathbb{G}_m}(e)}(\operatorname{pt})$, we obtain the standard module corresponding to $(e,s)$.

Nakajima in \cite{Nak01} constructs the action (\ref{eq:action of quantum affine on fiber}) via the geometry of quiver varieties (note that his approach works for an arbitrary quiver $Q$). Nakajima's construction can be reformulated as follows in our type A case (compare with \cite{Sav06}). Let $\mathscr{S}_e := e + Z_{\mathfrak{gl}_n}(f)$ be the Slodowy slice associated with an $\mathfrak{sl}_2$-triple $(e,h,f)$. The action map 
$
\operatorname{GL}_N \times \mathscr{S}_e \rightarrow \mathfrak{gl}_N
$
is smooth, so it induces a (pullback) homomorphism of convolution algebras:
\begin{equation}\label{eq:homom conv alg}
K^{\operatorname{GL}_N \times \mathbb{G}_m}(T^*\mathscr{F} \times_{\mathfrak{gl}_N} T^*\mathscr{F}) \rightarrow K^{Z_{\operatorname{GL}_N \times \mathbb{G}_m}(e)^{\mathrm{red}}}(\widetilde{\mathscr{S}}_e \times_{\mathfrak{gl}_N} \widetilde{\mathscr{S}}_e),
\end{equation}
where $\widetilde{\mathscr{S}}_e \subset T^*\mathscr{F}$ is the preimage of $\mathscr{S}_e$ and $\bullet^{\mathrm{red}}$ refers to the reductive part of the centralizer. Compare with \cite{TX06, BM13}.
Note that the fibers of $T^*\mathscr{F} \rightarrow \mathfrak{gl}_N$, $\widetilde{\mathscr{S}}_e \rightarrow \mathfrak{gl}_N$ over $e$ are both equal to $\mathscr{F}_e$.
It is not hard to see that the convolution action 
\begin{equation*}
K^{\operatorname{GL}_N \times \mathbb{G}_m}(T^*\mathscr{F} \times_{\mathfrak{gl}_N} T^*\mathscr{F}) \curvearrowright K^{Z_{\operatorname{GL}_N \times \mathbb{G}_m}(e)}(\mathscr{F}_e)
\end{equation*}
we used above is induced by the convolution action 
\begin{equation}\label{eq:act of slice K theory}
K^{Z_{\operatorname{GL}_N \times \mathbb{G}_m}(e)^{\mathrm{red}}}(\widetilde{\mathscr{S}}_e \times_{\mathfrak{gl}_N} \widetilde{\mathscr{S}}_e) \curvearrowright K^{Z_{\operatorname{GL}_N \times \mathbb{G}_m}(e)}(\mathscr{F}_e)
\end{equation}
composed with the homomorphism (\ref{eq:homom conv alg}). So, the action (\ref{eq:action of quantum affine on fiber}) is induced by (\ref{eq:act of slice K theory}) and the homomorphism $U_v(\widehat{\mathfrak{gl}}_m) \rightarrow K^{Z_{\operatorname{GL}_N \times \mathbb{G}_m}(e)^{\mathrm{red}}}(\widetilde{\mathscr{S}}_e \times_{\mathfrak{gl}_N} \widetilde{\mathscr{S}}_e)$.

By \cite{Maf05, NT17}, components of $\wti{\mathscr{S}}_e$  are isomorphic to Nakajima quiver varieties corresponding to a fixed framing vector (determined by $e$) and varying dimension vectors.

Than the resulting action of $U_v(\widehat{\mathfrak{gl}}_m)$ boils down to the one defined in \cite{Nak01}, up to the conventions for the $\mathbb{G}_m$-action and the normalization of the quantum-loop generators.

The action (\ref{eq:action of quantum affine on fiber}) can be recovered from the action (\ref{eq: sum of K-theories of convolutions}) discussed in Corollary \ref{cor: action on module}. Namely, under the Lusztig correspondence, the choice of a nilpotent orbit $\mathbb{O}_e \subset \mathfrak{gl}_N$ with at most $n$ Jordan blocks determines the choice of the $\operatorname{PGL}_n(\mathcal{O})$-orbit $\operatorname{PGL}_n(\mathcal{O}) \cdot z^\eta \subset \overline{\operatorname{Gr}}_{\operatorname{PGL}_n}^{N\omega_1}$ (compare with the proof of Proposition \ref{prop:iso_on_K_theory}). In this way, we obtain a filtration of~\eqref{eq: sum of K-theories of convolutions} by $U_v(\widehat{\mathfrak{gl}}_m)$-submodules labeled by nilpotent orbits with at most $n$ Jordan blocks. The subquotient corresponding to $e$ is  isomorphic as a  $U_v(\widehat{\mathfrak{gl}}_m)$-module to $K^{Z_{\operatorname{GL}_N \times \mathbb{G}_m}(e)}(\mathscr{F}_e)$ appearing  in~\eqref{eq:action of quantum affine on fiber}.


\end{document}